\documentclass[11pt,reqno]{amsart}
\usepackage{amsmath,amsthm,amssymb,amsfonts,euscript,mathrsfs,graphics,color,latexsym,marginnote}
\usepackage[dvips]{graphicx}
\usepackage[margin=1.05in]{geometry}
\usepackage{hyperref}
\usepackage{relsize}
\usepackage[toc,page]{appendix}
\usepackage[shortlabels]{enumitem}
\usepackage{comment}
\usepackage[all]{xy}
\usepackage{esint}

\usepackage[dvipsnames]{xcolor}

\usepackage{hyperref}
\hypersetup{colorlinks=true, pdfstartview=FitV,linkcolor=blue!70!black,citecolor=red!70!black, urlcolor=green!60!black}
\definecolor{labelkey}{rgb}{0.6,0,0}

\def\R{{\mathbb R}}

\def\e{{\varepsilon}}

\def\r{{\rho}}

\def\dd{\mathrm{d}}
\def\p{{\prime}}

\def\supp{\,\mbox{supp}\,}

\def\+R{+_{_{ \!\! \R}}}

\def\bar{\overline}

\def\hat{\widehat}

\newtheorem{theorem}{Theorem}[section]
\newtheorem{corollary}{Corollary}[theorem]
\newtheorem{lemma}[theorem]{Lemma}
\newtheorem{proposition}[theorem]{Proposition}

\theoremstyle{definition}
\newtheorem{definition}{Definition}[section]
\theoremstyle{remark}
\newtheorem{remark}{Remark}[section]

\numberwithin{equation}{section}

\DeclareMathOperator*{\esssup}{ess\,sup}

\newcommand{\eqdef  }{\overset{\mbox{\tiny{def}}}{=}}

\providecommand{\abs}[1]{\left\lvert #1 \right\rvert}
\providecommand{\nm}[1]{\left\lVert #1 \right\rVert}
\providecommand{\br}[1]{\left\langle #1 \right\rangle}

\def\ud{\mathrm{d}}
\def\dt{\partial_t}
\def\p{\partial}
\def\ls{\lesssim}

\def\rt{\rightarrow}
\def\r{\mathbb{R}}
\def\no{\nonumber}
\def\ue{\mathrm{e}}

\def\ds{\displaystyle}

\def\pp{\mathcal{P}}

\def\zz{\mathcal{Z}_{\vartheta}}

\begin{document}

\title[Vlasov-Poisson-Landau with Specular Boundary]{The Vlasov-Poisson-Landau System with the Specular-Reflection Boundary Condition}

\author[H. Dong]{Hongjie Dong}
\address{Brown University}
\email{hongjie\_dong@brown.edu}

\author[Y. Guo]{Yan Guo}
\address{Brown University}
\email{yan\_guo@brown.edu}

\author[Z. Ouyang]{Zhimeng Ouyang}
\address{Institute for Pure and Applied Mathematics, University of California, Los Angeles}
\email{zhimeng\_ouyang@alumni.brown.edu}

%
%

\begin{abstract}
We consider the Vlasov-Poisson-Landau system, a classical
model for a dilute collisional plasma interacting through Coulombic
collisions and with its self-consistent electrostatic field.
We establish global stability and well-posedness near the Maxwellian equilibrium state with decay in time and some regularity results for small initial perturbations, in any general bounded domain (including a
torus as in a tokamak device), in the presence of specular reflection
boundary condition.
We provide a new improved $L^{2}\rightarrow L^{\infty }$ framework:
$L^{2}$ energy estimate combines only with $S_{ }^{p}$ estimate for the ultra-parabolic equation.

\end{abstract}

\maketitle


%
%
\setcounter{tocdepth}{1}

\begin{quote}
\tableofcontents
\end{quote}


%
%

\section{Introduction and Set-up} \label{Sec:Intro-Setup}

\subsection{The Vlasov-Poisson-Landau System}

In this paper, we are interested in the global well-posedness and stability of the Vlasov-Poisson-Landau (VPL) system in three dimensions, which is considered as a fundamental collisional plasma model.
In the absence of magnetic effects, the dynamics of dilute charged particles (e.g. electrons and ions) can be described by the Vlasov-Landau equations:
\begin{equation} \label{Eq:Vlasov-Landau-pm}
\begin{split}
& \partial_{t} F_+ + v\cdot \nabla_{\!x} F_+ + \frac{e_+}{m_+}\, \mathbf{E}\cdot \nabla_{\!v} F_+
\,=\, \mathcal{Q}\,[F_+,F_+] + \mathcal{Q}\,[F_-,F_+] \,,\\[3pt]
& \partial_{t} F_- + v\cdot \nabla_{\!x} F_- - \frac{e_-}{m_-}\, \mathbf{E}\cdot \nabla_{\!v} F_-
\,=\, \mathcal{Q}\,[F_+,F_-] + \mathcal{Q}\,[F_-,F_-] \,,\\[7pt]
& F_\pm(0,x,v) = F_{0,\pm}(x,v) . \\[5pt]
\end{split}
\end{equation}
Here the unknowns $F_\pm(t,x,v) \geq 0$ are the (real-valued) density distribution functions for ions ($+$) and electrons ($-$), respectively, at time $t\geq 0$, near the position $x=(x_1,x_2,x_3)\in\Omega\subset\R^3_x$, and having the velocity $v=(v_1,v_2,v_3)\in\R^3_v$, where $\Omega$ denotes a bounded domain in $\R^3$.
The function $\mathbf{E}(t,x)$\, is the self-consistent electrostatic field created by all plasma particles, which depends in a complex way on the distribution functions $F_\pm(t,x,v)$.
The constants $e_\pm$, $m_\pm$ represent the magnitude of the particles' electric charges and masses.

The classical VPL system models the dynamics of a collisional plasma interacting with its own electrostatic field as well as its grazing collisions. Such grazing collisions are given by the famous Landau (Fokker-Planck) collision operator, proposed by Landau in 1936\,\footnote{The constant $c_{12}=2\pi e_1^2 e_2^2 \ln\Lambda$, $\ln\Lambda=\ln\big(\frac{\lambda_D}{b_0}\big)$, where $\lambda_D=\big(\frac{T}{4\pi n_e e^2}\big)^{\!1/2}$ is the Debye shielding distance and $b_0=\frac{e^2}{3T}$\, a typical ``distance of closest approach'' for a thermal particle.}:
\begin{equation*}
\mathcal{Q}\,[F_1,F_2](v)
:= \,\frac{c_{12}}{m_1}\, \nabla_v\cdot\left\{\int_{\mathbb{R}^3}\Phi(v\!-\!v')\left[\frac{1}{m_1}F_1(v')\nabla_{\!v} F_2(v) - \frac{1}{m_2}F_2(v)\nabla_{\!v} F_1(v')\right]\dd v' \right\} ,
\end{equation*}
where the (generalized) Landau collision kernel $\Phi$ is defined as
\begin{equation} \label{Landau-kernel-general}
\Phi(v) := \left\{I_{3}-\frac{v}{|v|}\otimes\frac{v}{|v|}\right\}\cdot|v|^{\gamma+2},\quad \gamma\in[-d,1\,] .
\end{equation}
This is a non-negative symmetric $3\times 3$ matrix, and we will focus on the original physical case for Coulomb interaction corresponding to $\gamma=-d=-3$.

Since the Vlasov-Poisson equations are an approximation of the Vlasov-Maxwell equations in the non-relativistic zero-magnetic field limit, the self-consistent electric field $\mathbf{E}(t,x)$ is determined by a scalar electric potential $\phi(t,x)$\,:
\begin{equation} \label{E-field}
\mathbf{E} = -\nabla_{\!x}\phi ,
\end{equation}
and the electric potential $\phi$\, satisfies the Poisson equation for electrostatics, which is
\begin{equation} \label{Eq:Poisson}
-\Delta_x\phi = 4\pi\rho := 4\pi\int_{\R^3}\big[e_+ F_+ - e_- F_-\big]\dd v,
\end{equation}
with a Dirichlet or Neumann boundary condition (and some additional conditions to guarantee the solvability).
Here $\rho(t,x)$ stands for the electric charge density.

The coupled system \eqref{Eq:Vlasov-Landau-pm} with \eqref{E-field} and \eqref{Eq:Poisson} is called a Vlasov-Poisson-Landau (VPL) system.
For notational simplicity, from now on we will consider a model system (for single-species particles, e.g. ions (+)) with all constants normalized to be one (The analysis of the full two-species case is similar):
\begin{equation} \label{Eq:Vlasov-Landau-model}
\begin{split}
& \partial_{t} F + v\cdot \nabla_{\!x} F + \mathbf{E}\cdot \nabla_{\!v} F = \mathcal{Q}\,[F,F] ,\\[3pt]
& F(0,x,v) = F_0(x,v) ,
\end{split}
\end{equation}
where the distribution function $F(t,x,v) \geq 0$ is defined for $(t,x,v)\in [\,0,\infty)\times\Omega\times\R^3$,
and the Landau collision operator (for Coulomb interaction) takes the form
\begin{align*}
\mathcal{Q}\,[G,F](v)
:=& \;\nabla_v\cdot\left\{\int_{\mathbb{R}^3}\Phi(v\!-\!v')\left[G(v')\nabla_{\!v} F(v) - F(v)\nabla_{\!v} G(v')\right]\dd v' \right\} \\
=& \;\partial_i\int_{\mathbb{R}^3}\Phi^{ij}(v\!-\!v')\left[G(v')\partial_j F(v) - F(v)\partial_j G(v')\right]\dd v'
\end{align*}
with the Landau collision kernel ($\gamma=-3$)
\begin{equation} \label{Landau-kernel}
\Phi^{ij}(v) := \left\{\delta_{ij}-\frac{v_i v_j}{|v|^2}\right\}\cdot|v|^{-1}.
\end{equation}
The electrostatic field $\mathbf{E}(t,x) = -\nabla_{\!x}\phi(t,x)$, and the potential $\phi(t,x)$ satisfies the Poisson equation:
\begin{equation} \label{Eq:Poisson-model}
-\Delta_x\phi(t,x) = \rho(t,x) - \rho_0 := \int_{\R^3}F(t,x,v)\,\dd v -\rho_0 ,
\end{equation}
where the background density $\rho_0$ is a constant number (to be specified later).
The prescribed boundary condition for the Poisson equation is either of Dirichlet type (if we fix values of the potential at the boundary, for example,)
\begin{equation*} 
\phi(t,x) = 0 \quad\text{for}\;\, x\in\partial\Omega \;\text{ and }\, t\geq 0 ,
\end{equation*}
or of Neumann type (if we are given fixed normal component of the electric field across the surface)
\begin{equation*} 
\frac{\partial\phi(t,x)}{\partial n} = \nabla_{\!x}\phi(t,x)\cdot n_x = 0 \quad\text{for}\;\, x\in\partial\Omega \;\text{ and }\, t\geq 0 ,
\end{equation*}
which means that the electric field at the surface (of a perfect magnetic conductor) can only have a tangential component:
\begin{equation*}
\mathbf{E}_\perp =\, \mathbf{E}(t,x)\cdot n_x = 0 \quad\text{for}\;\, x\in\partial\Omega \;\text{ and }\, t\geq 0 .
\end{equation*}
For the Dirichlet boundary condition, no extra restriction is needed;
while solving Poisson equation for the potential $\phi$ with a Neumann boundary condition requires the neutral condition to ensure the existence:
\begin{equation*} 
\int_\Omega \left[\int_{\R^3}F(t,x,v)\,\dd v -\rho_0 \right]\!\dd x = 0 \quad\text{for all}\;\,t\geq 0 ,
\end{equation*}
which follows by setting $\rho_0 \!:= \frac{1}{|\Omega|} \iint_{\Omega\times\R^3}F_0(x,v)\,\dd v\dd x$ and the conservation law of mass (\ref{mass-conservation_F}).
Also a zero-mean condition
\begin{equation*} 
\int_\Omega \phi(t,x)\,\dd x = 0 \quad\text{for all}\;\,t\geq 0
\end{equation*}
guarantees the uniqueness of solutions.

\subsubsection{Maxwellian and Equations for the Perturbation}

Our goal in this article is to construct unique global solutions for the Vlasov-Poisson-Landau system \eqref{Eq:Vlasov-Landau-model} and \eqref{Eq:Poisson-model} near equilibrium state -- the (normalized) global Maxwellian:
\begin{equation*}
\mu(v) = e^{-|v|^2} .
\end{equation*}
Accordingly, we define the standard perturbation $f(t,x,v)$ to $\mu$ by
\begin{equation*}
F(t,x,v) = \mu(v) + \sqrt{\mu(v)}\,f(t,x,v) .
\end{equation*}
Letting $f(t,x,v)$ be our new unknown, the Vlasov-Poisson-Landau system for the perturbation now reads as follows (see the definitions of $L$ and $\Gamma$ in (\ref{L-opt}) - (\ref{Gamma-opt})):
\vspace{2pt}
\begin{equation} \label{Eq:Vlasov-Landau_f}
\begin{split}
& \partial_t f + v\cdot\nabla_{\!x}f + \mathbf{E}_f\cdot\nabla_{\!v}f + Lf = \Gamma[f,f] + \big\{\mathbf{E}_f\cdot v\big\}f + 2\big\{\mathbf{E}_f\cdot v\big\}\sqrt{\mu} ,\\[3pt]
& f(0,x,v) = f_0(x,v) ,\\[3pt]
\end{split}
\end{equation}
where $\mathbf{E}_f = -\nabla_{\!x}\phi_f$ is determined by
\begin{equation} \label{Eq:Poisson_f}
-\Delta_x\phi_f = \int_{\R^3} \sqrt{\mu}\,f\,\dd v
\end{equation}
with the Dirichlet BC $$\phi_f = 0 \quad\text{on}\;\, \partial\Omega$$
or the Neumann BC $$\frac{\partial\phi_f}{\partial n} = 0 \quad\text{on}\;\, \partial\Omega$$
with $\int_\Omega\phi_f \,\dd x = 0,\, \forall\,t\geq 0$,
and the neutral condition for solvability now becomes
\begin{equation*} 
\iint_{\Omega\times\R^3} f(t,x,v)\sqrt{\mu(v)} \,\dd v\dd x = 0 \quad\text{for all}\;\,t\geq 0 ,
\end{equation*}
which is automatically satisfied from the conservation of mass (\ref{mass-conservation_f}) under the assumption that the initial data $F_0$ has the same mass as the Maxwellian $\mu$, i.e.
\begin{equation*}
\int_{\R^3}\mu \,\dd v = \frac{1}{|\Omega|} \iint_{\Omega\times\R^3}\!F_0(x,v)\,\dd v\dd x = \rho_0 .
\end{equation*}

By expanding the bilinear collision operator $\mathcal{Q}\,[G,F] = \mathcal{Q}\,[\,\mu +\! \sqrt{\mu}g,\mu +\! \sqrt{\mu}f\,]$ and noting that $\mathcal{Q}\,[\mu,\mu]=0$, we can decompose it into a linear part and a nonlinear part:
\begin{align*}
\mathcal{Q}\,[\,\mu +\! \sqrt{\mu}g,\mu +\! \sqrt{\mu}f\,]
&\,=\, \mathcal{Q}\,[\mu,\mu] + \mathcal{Q}\,[\,\mu,\sqrt{\mu}f\,] + \mathcal{Q}\,[\sqrt{\mu}g,\mu\,] + \mathcal{Q}\,[\sqrt{\mu}g,\sqrt{\mu}f\,] \\
&\,=:\, \mu^{1/2} \big\{Af + Kg + \Gamma[g,f] \big\} .
\end{align*}
We define the linear operator
\begin{equation} \label{L-opt}
L := -A-K
\end{equation}
with the expressions
\begin{equation} \label{A-opt}
\begin{split}
Af := \mu^{\!-1/2} \mathcal{Q}\,[\,\mu,\mu^{1/2}f\,]
&= \mu^{\!-1/2}\,\partial_i\!\left\{\mu^{1/2}\sigma^{ij}[\partial_j f+v_j f\,]\right\} \\
&= \partial_i\left[\sigma^{ij}\partial_j f\right] - \sigma^{ij}v_i v_j f + \partial_i \sigma^i f \\
&= \nabla_{\!v}\cdot(\sigma\nabla_{\!v} f)+(\nabla_{\!v} -v)\cdot(\sigma v)f , \qquad
\end{split}
\end{equation}
\begin{equation} \label{K-opt}
\qquad\quad\quad\quad
Kf := \mu^{\!-1/2}\mathcal{Q}\,[\,\mu^{1/2}f,\mu\,]
= -\mu^{\!-1/2}\,\partial_i\!\left\{\mu\bigg[\Phi^{ij}\!\ast\!\left\{\mu^{1/2}\big[\partial_j f+v_j f\big]\right\}\bigg]\right\} ,
\end{equation}
and the nonlinear operator
\begin{equation} \label{Gamma-opt}
\begin{split}
\Gamma[g,f] := \mu^{\!-1/2}\mathcal{Q}\,[\,\mu^{1/2}\!g,\mu^{1/2}f\,]
=& \;\partial_i\left[\left\{\Phi^{ij}\!\ast\!\big[\mu^{1/2}g\big]\right\}\partial_j f \right] - \left\{\Phi^{ij}\!\ast\!\big[v_i\mu^{1/2}g\big]\right\}\partial_j f \\
&-\partial_i\left[\left\{\Phi^{ij}\!\ast\!\big[\mu^{1/2}\partial_j g\big]\right\}f\, \right] + \left\{\Phi^{ij}\!\ast\!\big[v_i\mu^{1/2}\partial_j g\big]\right\}f .
\end{split}
\end{equation}
Here we introduce notation for the diffusion matrix (collision frequency) that captures the dissipation of the Landau collision kernel:
\begin{align}
& \sigma_{u}^{ij}(v) := \left[\Phi^{ij}\!\ast u\right]\!(v) = \int_{\R^3}\Phi^{ij}(v-v')\,u(v')\,\dd v' , \label{diffusion-matrix-sigma-1} \\
& \sigma_u(v) := \big(\sigma_{u}^{ij}(v)\big),
\qquad \sigma := \sigma_{\mu} = \Phi\ast\mu , \label{diffusion-matrix-sigma-2} \\[5pt]
& \sigma^i := \sigma^{ij}v_j = \left[\Phi^{ij}\!\ast\mu\right]v_j = \Phi^{ij}\!\ast[v_j\mu] . \label{diffusion-matrix-sigma-3}
\end{align}
The last equality is due to the property that for any fixed $i$,
\begin{equation} \label{Landau-kernel-property}
\sum_j \Phi^{ij}(w)\, w_j \equiv 0,\quad w:=v-v'
\end{equation}
using the specific structure (\ref{Landau-kernel}) of the Landau kernel $\Phi$.

\subsubsection{Domain and Boundary Condition}

Throughout this paper, our domain $\Omega := \{x\in\R^3: \zeta(x)<0\,\}$ is connected and bounded with $\zeta(x)$ being a smooth function. We also assume that $\nabla\zeta(x) \neq 0$ on the boundary $\partial\Omega = \{x:\zeta(x)=0\}$.
The outward unit normal vector $n_x$ at $x\in\partial\Omega$ is given by
\begin{equation*}
n_x := \frac{\nabla \zeta(x)}{|\nabla\zeta(x)|}, \\[3pt]
\end{equation*}
and it can be extended smoothly near the boundary $\partial\Omega$.
Additionally, we say that $\Omega$ has a rotational symmetry if there exists a point $x_0$ and a vector $\omega$ such that
\begin{equation} \label{rotational-symmetry}
\big[(x-x_0)\times\omega\big] \cdot n_x = 0
\end{equation}
for all $x\in\partial\Omega$.
It is worth noting that our results and methods apply to a general class of ``non-convex'' and ``non-simply connected'' domains.

We denote also by $\gamma := \partial\Omega\times\R^3$ the phase boundary,
and then we split $\gamma$ into an outgoing boundary $\gamma_+$, an incoming boundary $\gamma_-$, and the singular boundary $\gamma_0$ (i.e., the ``grazing set'') for grazing velocities, respectively defined as
\begin{equation*}
\begin{split}
\gamma_+ &:= \{(x,v)\in\gamma:\, n_x\cdot v>0\,\} ,\\
\gamma_- &:= \{(x,v)\in\gamma:\, n_x\cdot v<0\,\} ,\\
\gamma_0 &:= \{(x,v)\in\gamma:\, n_x\cdot v=0\,\} .
\end{split}
\end{equation*}

Our problem concerns the \textit{specular-reflection boundary condition}, which in terms of the density distribution function $F$, can be formulated as
\begin{equation} \label{Specular-BC_F}
F(t,x,v)|_{\gamma_-} = F(t,x,\mathcal{R}_x v)|_{\gamma_+}
\end{equation}
for all $t\geq 0$,
where for $(x,v)\in\gamma$, $$\mathcal{R}_x v := v-2(n_x\!\cdot v)n_x.$$
This is equivalent to a specular-reflection boundary condition satisfied by the perturbation $f$:
\begin{equation} \label{Specular-BC_f}
f(t,x,v)|_{\gamma_-} = f(t,x,\mathcal{R}_x v)|_{\gamma_+} ,\quad \forall\; t\geq 0 .
\end{equation}

\subsubsection{Conservation Laws and H-Theorem}

The conservation laws will play crucial role in the energy estimates.

The conservation of electric charge is described by the continuity equation
\begin{equation} \label{continuity-eq}
\partial_t \rho \,+\, \nabla_{\!x}\cdot \mathbf{j} \,=\, 0 ,
\end{equation}
where $\rho(t,x) \,:=\, \int_{\R^3} F\,\dd v$ is the charge density,
and $\mathbf{j}(t,x) \,:=\, \int_{\R^3} v\, F\,\dd v$ is called the current density.
In terms of the perturbation $f$, we may take
\begin{align*}
\rho[f](t,x) &\,:= \int_{\R^3}\! \sqrt{\mu}\,f\,\dd v , \\
\mathbf{j}[f](t,x) &\,:= \int_{\R^3}\! v\sqrt{\mu}\,f\,\dd v .
\end{align*}

Under the specular-reflection boundary condition (\ref{Specular-BC_F}), It is well-known that both total mass and total energy are conserved for the Vlasov-Poisson-Landau system \eqref{Eq:Vlasov-Landau-model}:
\begin{align}
& \frac{\dd}{\dd t} \iint_{\Omega\times\R^3} F(t) \,\dd v\dd x \,\equiv 0 , \label{mass-conservation_F} \\
& \frac{\dd}{\dd t} \left\{\iint_{\Omega\times\R^3} |v|^2 F(t) \dd v\dd x + \int_{\Omega}|\mathbf{E}(t)|^2 \dd x \right\} \equiv 0.\notag
\end{align}
Assuming (without loss of generality, by some rescaling) that initially $F_0$ has the same mass and total energy as the Maxwellian $\mu$, we can then rewrite the mass-energy conservation laws for the system \eqref{Eq:Vlasov-Landau_f} - \eqref{Eq:Poisson_f} with (\ref{Specular-BC_f}), in terms of the perturbation $f$:
\begin{align}
& \iint_{\Omega\times\R^3} f(t)\sqrt{\mu} \,\dd v\dd x \,\equiv 0 , \label{mass-conservation_f} \\
& \iint_{\Omega\times\R^3} |v|^2 f(t)\sqrt{\mu} \,\dd v\dd x \,\equiv -\int_{\Omega}\big|\mathbf{E}_f(t)\big|^2 \dd x . \label{energy-conservation_f}
\end{align}

Through the coupled Poisson equation, we may further deduce
\begin{equation} \label{flux-conservation}
\begin{split}
\frac{\dd}{\dd t} \int_{\partial\Omega} \mathbf{E}_f(t)\cdot n \,\dd S
&
\;=\;  -\,\frac{\dd}{\dd t} \int_{\partial\Omega}\! \frac{\partial\phi_f}{\partial n}(t) \,\dd S \\
&\;=\; -\,\frac{\dd}{\dd t} \int_{\Omega} \Delta\phi_f(t)\,\dd x
\;=\; \frac{\dd}{\dd t} \int_{\Omega} \left(\int_{\R^3}\! \sqrt{\mu}\,f(t)\,\dd v\right)\dd x \\
&\;=\; \frac{\dd}{\dd t} \iint_{\Omega\times\R^3}\! \sqrt{\mu}\,f(t) \,\dd v\dd x
\,\equiv\, 0 ,
\end{split}
\end{equation}
which means that the flux of the self-consistent electric field is conserved as well. This property will be used in the energy estimate in Section\;\ref{Sec:Energy-est-decay}.
Under the same assumption, we also have
\begin{equation*}
\int_{\partial\Omega} \mathbf{E}_f(t)\cdot n \,\dd S \,\equiv\, 0 .
\end{equation*}

Moreover, if the domain $\Omega$ has a rotational symmetry (\ref{rotational-symmetry}), then we will assume additionally that the corresponding conservation of angular momentum is valid for all $t\geq 0$:
\begin{equation} \label{angular-momentumconservation_f}
\iint_{\Omega\times\R^3} \big[(x-x_0)\times\omega\big] \cdot v\, f(t,x,v)\sqrt{\mu(v)}\, \dd v\dd x = 0 .
\end{equation}

We also recall the celebrated H-Theorem of Boltzmann --
the ``entropy'' $$\mathcal{H}(t) := \iint_{\Omega\times\R^3} F(t) \ln\! F(t) \,\dd v\dd x$$ for the system \eqref{Eq:Vlasov-Landau-model}\eqref{Eq:Poisson-model} is non-increasing as time passes:
\begin{equation*}
\frac{\dd}{\dd t}\,\mathcal{H}(t) \leq 0 ,\quad \forall\;t\geq 0.
\end{equation*}
This suggests the asymptotic stability of the Maxwellian $\mu$, for which the entropy is minimized
(cf.\;\cite{Strain.Guo2004}).

\subsection{Linearization and Reformulation}

To apply the Picard iteration argument,
in the equation \eqref{Eq:Vlasov-Landau_f} we replace the nonlinear dependence on $f(t,x,v)$ with a given $g(t,x,v)$:
\vspace{2pt}
\begin{equation} \label{Eq:linearized-VL_f}
\begin{split}
& \partial_t f + v\cdot\nabla_{\!x}f + \mathbf{E}_g\cdot\nabla_{\!v}f + Lf = \Gamma[g,f] + \big\{\mathbf{E}_g\cdot v\big\}f + 2\big\{\mathbf{E}_f\cdot v\big\}\sqrt{\mu} ,\\[3pt]
& f(0,x,v) = f_0(x,v) ,\\[3pt]
\end{split}
\end{equation}
where $\mathbf{E}_g = -\nabla_{\!x}\phi_g$ and $\phi_g$ is solved from
\begin{equation} \label{Eq:Poisson_g}
-\Delta_x\phi_g = \int_{\R^3} \sqrt{\mu}\,g\,\dd v
\end{equation}
with either Dirichlet BC or Neumann BC and $\int_\Omega\phi_g \,\dd x =0,\, \forall\,t\geq 0$.

Next we rearrange the terms $-Lf = (A+K)f$ and $\Gamma[g,f]$ as
\begin{equation*}
-Lf + \Gamma[g,f] \,=\, \bar{A}_g f + \bar{K}_g f
\end{equation*}
with
\begin{equation*}
\begin{split}
\bar{A}_g f :=
&\; \partial_i\left[\left\{\Phi^{ij}\!\ast\!\big[\mu+\mu^{1/2}g\big]\right\}\partial_j f  \right] \\
&-\left\{\Phi^{ij}\!\ast\!\big[v_j\mu^{1/2}g\big]\right\}\partial_i f
-\left\{\Phi^{ij}\!\ast\!\big[\mu^{1/2}\partial_j g\big]\right\}\partial_i f \;\;\, \\
=: &\; \nabla_{\!v}\cdot\big(\sigma_{\!G}\nabla_{\!v} f\big) + a_g\cdot\nabla_{\!v}f
\end{split}
\end{equation*}
and
\begin{equation} \label{K_g-opt}
\begin{split}
 \bar{K}_{\!g} f :=
 &\; Kf + \partial_i \sigma^i f - \sigma^{ij}v_i v_j f \\
 & -\partial_i\left\{\Phi^{ij}\!\ast\!\big[\mu^{1/2}\partial_j g\big]\right\}f + \left\{\Phi^{ij}\!\ast\!\big[v_i\mu^{1/2}\partial_j g\big]\right\}f ,
\end{split}
\end{equation}
so that all the terms of $\bar{A}_g f$ contain at least one $v$-derivatives of $f$, while $\bar{K}_{\!g} f$ has no $v$-derivatives of $f$.
Notice that the $v$-derivatives $\partial_j f$ in $Kf$ can always be moved to $\mu^{1/2}$ via integration by parts and then outside the convolution by the property (\ref{Landau-kernel-property}) of the Landau kernel.

Combining terms of the same order in $v$, we reformulate the equation \eqref{Eq:linearized-VL_f} as
\begin{equation} \label{Eq:reformulated-VL_f}
\partial_t f + v\cdot\nabla_{\!x}f
= \nabla_{\!v}\cdot\big(\sigma_{\!G}\nabla_{\!v} f\big)
+ \big\{a_g - \mathbf{E}_g\big\}\cdot\nabla_{\!v}f
+ \Big\{ \bar{K}_{\!g} f + \big(v\cdot\mathbf{E}_g\big)f + 2\sqrt{\mu}\,v\cdot\mathbf{E}_f \Big\} ,
\end{equation}
where
\begin{align}
\sigma_{\!G} &:= \Phi \ast G = \Phi\ast\big[\mu+\mu^{1/2}g\big] , \label{sigma_G} \\[2pt]
a_g^i &:= -\, \Phi^{ij}\!\ast\big[v_j\mu^{1/2}g + \mu^{1/2}\partial_j g\big] , \label{a_g^i} \\
\mathbf{E}_g &:= \nabla_{\!x}\, \Delta_{x}^{-1}\! \int_{\R^3}\!\sqrt{\mu}\,g\,\dd v,\notag
\end{align}
and the operator $\bar{K}_{\!g}$ is defined in (\ref{K_g-opt}).
This equation can be viewed as a kinetic Fokker-Planck equation of the form (cf. \cite{Golse.Imbert.Mouhot.Vasseur2019})
\begin{equation} \label{Eq:kinetic-FP}
\partial_t f + v\cdot\nabla_{\!x} f = \nabla_{\!v}\cdot\big(\mathbf{A}\,\nabla_{\!v} f \big) + \mathbf{B}\cdot\nabla_{\!v}f + \mathbf{C}f ,
\end{equation}
where $g$ appears only in the coefficients of the operator for $f$.
We can also fit it into a more general class of hypoelliptic/ultraparabolic operators of Kolmogorov type (with rough coefficients), which allows us to apply some known estimates and results regarding this kind of operators (see Section \ref{Sec:Lemmas}).

When the size of $g$ (e.g. $\|g\|_{\infty,\vartheta} \sim\varepsilon$) is sufficiently small, the coefficients in \eqref{Eq:kinetic-FP} have the following basic properties (see Section \ref{Sec:Prelim} for more details):
\begin{equation*}
\begin{split}
\mathbf{A}(t,x,v) &:=\, \sigma_{\!G}(t,x,v) \\
&\,\,= \left\{\Phi\ast\big[\mu+\mu^{1/2}g(t,x,v)\big]\right\}(v)
\end{split}
\end{equation*}
is a $3\times3$ non-negative matrix, retaining the same analytical properties (e.g. eigenvalues/spectrum) as $\sigma$: (cf. 
Lemma 2.4 in \cite{Kim.Guo.Hwang2020})
$$0 < (1\!+\!|v|)^{-3}\,\mathbf{I}  \,\lesssim\,  \mathbf{A}(v) \,\lesssim\, (1\!+\!|v|)^{-1}\,\mathbf{I} .$$
This makes the second-order $v$-derivatives $\nabla_{\!v}\cdot\big(\mathbf{A}\,\nabla_{\!v}\circ \big)$ an elliptic operator (but \textit{not uniformly} elliptic in $v$).
Moreover,
$$\mathbf{B}(t,x,v) := a_g(t,x,v) - \mathbf{E}_g(t,x)$$
is a uniformly bounded 3-dimensional vector with (cf. Appendix\,A in \cite{Golse.Imbert.Mouhot.Vasseur2019})
\begin{equation*}
\|\mathbf{B}[g]\|_{\infty} \,\leq\, \|a_g\|_{\infty} + \|\mathbf{E}_g\|_{\infty} \,\lesssim\, \|g\|_{\infty} < \varepsilon .
\end{equation*}
Finally,
$$\mathbf{C}f := \bar{K}_{\!g} f + \big(v\cdot\mathbf{E}_g\big)f + 2\sqrt{\mu}\,v\cdot\mathbf{E}_f$$
is an operator bounded in some weighted $L^p$ space (cf. Lemmas 2.9 and 7.1 in \cite{Kim.Guo.Hwang2020}).

\subsection{Definition and Notation}

Throughout the paper, $C$ will generally denote an universal constant that may vary from line to line.
The notation $A \lesssim B$ means that $A\leq CB$ for some universal constant $C>0$; we will use $\gtrsim$ and $\simeq$ in a similar standard way.
Below is a collection of definitions and notation:

\subsubsection{Weighted Norms}
For $w(v) := \langle v\rangle=\sqrt{1+|v|^2}$,
\begin{align*}
|f|_{p,\vartheta} &:= \big|\langle v\rangle^\vartheta f\big|_{L^p(\R^3)}
= \left(\int_{\R^3}\langle v\rangle^{p\vartheta} |f(v)|^p \,\dd v \right)^{\frac{1}{p}} , \\
\|f\|_{p,\vartheta} &:= \big\|\langle v\rangle^\vartheta f\big\|_{L^p(\Omega\times\R^3)}
= \left(\iint_{\Omega\times\R^3}\langle v\rangle^{p\vartheta} |f(x,v)|^p \,\dd v\dd x \right)^{\frac{1}{p}} ,
\end{align*}
\begin{align*}
|f|_{\infty,\vartheta} &:= \big|\langle v\rangle^\vartheta f\big|_{L^\infty(\R^3)}
= \esssup_{\R^3}\, \langle v\rangle^{\vartheta} |f(v)| , \\
\|f\|_{\infty,\vartheta} &:= \big\|\langle v\rangle^\vartheta f\big\|_{L^\infty(\Omega\times\R^3)}
= \esssup_{\Omega\times\R^3}\, \langle v\rangle^{\vartheta} |f(x,v)| .
\end{align*}

\subsubsection{Dissipation and Energy} \label{SubsubSec:Dissipation-Energy}

\begin{align*}
\langle f,g\rangle
&:= \int_{\R^3} fg \,\dd v , \\
(f,g)
&:= \iint_{\Omega\times\R^3} fg \,\dd v\dd x ,
\end{align*}
\begin{align*}
\langle f,g\rangle_{\sigma}
&:= \int_{\R^3} \left[\sigma^{ij}\partial_{i}f\partial_{j}g + \sigma^{ij}v_{i}v_{j}fg\right] \dd v , \\
(f,g)_{\sigma}
&:= \iint_{\Omega\times\R^3} \left[\sigma^{ij}\partial_{i}f\partial_{j}g + \sigma^{ij}v_{i}v_{j}fg\right] \dd v\dd x ,
\end{align*}
\begin{align*}
|f|_{\sigma,\vartheta}^2 = \langle f,f\rangle_{\sigma,\vartheta}
&:= \int_{\R^3} \langle v\rangle^{2\vartheta}\left[\sigma^{ij}\partial_{i}f\partial_{j}f + \sigma^{ij}v_{i}v_{j}f^2\right] \dd v , \\
\|f\|_{\sigma,\vartheta}^2 = (f,f)_{\sigma,\vartheta}
&:= \iint_{\Omega\times\R^3} \langle v\rangle^{2\vartheta}\left[\sigma^{ij}\partial_{i}f\partial_{j}f + \sigma^{ij}v_{i}v_{j}f^2\right] \dd v\dd x ,
\end{align*}
$^{_{_\bullet}}$ Instant Energy:
\begin{equation*}
\mathcal{I}_{\vartheta}[f(t)]
:=\, \big\|f(t)\big\|_{2,\vartheta}^{2} + \big\|\mathbf{E}_f(t)\big\|_{L^2_x}^2 ,
\end{equation*}
$^{_{_\bullet}}$ Dissipation Rate:
\begin{equation*}
\mathcal{D}_{\vartheta}[f(t)]
:=\, \big\|f(t)\big\|_{\sigma,\vartheta}^{2} + \big\|\mathbf{E}_f(t)\big\|_{L^2_x}^2 ,
\end{equation*}
$^{_{_\bullet}}$ Total Energy (Instant\;$+$\;Accumulative):
\begin{equation*}
\mathcal{E}_{\vartheta}[f(t)]
:=\, \mathcal{I}_{\vartheta}[f(t)] \,+ \int_{0}^{t}\! \mathcal{D}_{\vartheta}[f(\tau)] \,\dd\tau .
\end{equation*}

\subsubsection{Trace and Integral over Phase-Boundary}

\begin{equation*}
\gamma f := f|_{\gamma}, \qquad \gamma_\pm f := f|_{\gamma_\pm} ,
\end{equation*}
\begin{equation*}
\left\| \gamma_\pm f \right\|_{L^p(\gamma_\pm)} :=  \left(\iint_{\gamma_\pm} \big| \gamma_\pm f(x,v) \big|^p |v\!\cdot\! n_x| \dd v \dd S_x \right)^{\frac{1}{p}} .
\end{equation*}

\section{Main Results and Ideas} \label{Sec:Main-results}

\subsection{Background and Motivation}

Despite the important role kinetic theory plays for describing a plasma, we
are not aware any global well-posed theory for classical kinetic models in a
``donut shape'' torus, which resembles a tokamak device. The fundamental
difficulty lies in the singularity for kinetic equations near the grazing
set $\gamma _{0}$. Even for the linear free streaming operator $\partial
_{t}+v\cdot \nabla _{x},$ the grazing set $\gamma _{0}$ is always
characteristic but not uniformly characteristic, the most challenging case
in the regularity study of hyperbolic PDE in a bounded domain. In general,
singularity and discontinuity is expected to develop from the grazing set
$\gamma _{0},$ and propagate inside for a non-convex domain
\cite{Guo1995,Kim2011,Guo.Kim.Tonon.Trescases2016},
and the solutions are of BV at the best.
On the other hand, regularity is expected to deteriorate near the grazing set
$\gamma _{0},$ but could be confined and localized in a convex domain
\cite{Hwang2004,Guo.Kim.Tonon.Trescases2017}.

Since the regularity plays an important role in establishing uniqueness in
the kinetic theory, such a loss of regularity creates fundamental
difficulties in the PDE study. Many techniques with Sobolev norms developed
for the free space are not suitable, and completely new mathematical
frameworks need to be developed. Furthermore, such a loss of regularity can
even alter our basic understanding of classical boundary behavior for
kinetic theory \cite{AA003}.

Thanks to the possibility of regularity in a convex domain for kinetic
theory, there have been important advances in various kinetic models for a
plasma in convex domains.
In \cite{Hwang2004,Hwang.Velazquez2010},
global well-posedness is established for the collisionless Vlasov-Poisson system.
In recent work of \cite{Chen.Kim.Li2020,Cao.Kim.Lee2019},
global well-posedness is established for the collisional Vlasov-Poisson-Boltzmann
system near Maxwellian distributions.

An $L^{2}\rightarrow L^{\infty}$ framework for the Boltzmann equation is
developed in \cite{Guo2010} for study of well-posedness in bounded domains,
which is based on the basic $L^{2}$ energy estimate and a bootstrap to
$L^{\infty}$ bound thanks to an interaction between free streaming
characteristics and averaging in velocity in the collision operator.
See subsequent developments along this direction in
\cite{Esposito.Guo.Marra2018,Esposito.Guo.Kim.Marra2018,Briant.Guo2016}.
In particular, a quantitative estimate of macroscopic part $\mathcal{P}f$
in terms of the microscopic part $(I-\mathcal{P})f$ for the
$L^{2}$-\,elliptic/positivity estimate is established in
\cite{Esposito.Guo.Kim.Marra2013}.

Our paper continues the recent program to study an $L^{2}\rightarrow
L^{\infty }$ framework for studying the Landau equation
\cite{Kim.Guo.Hwang2020,Guo.Hwang.Jang.Ouyang2020}.
It should be noted that due to nonlinear diffusion in the
velocity space, $L^{\infty}$ is not quite enough to ensure uniqueness as in
the Boltzmann equation; some control of $\nabla_{v}f$ is needed.
The framework relies on $L^{2}$ energy estimate, and a recent novel
De Giorgi--Nash--Moser theory for the Landau equation, which is developed in
\cite{Golse.Imbert.Mouhot.Vasseur2019}.
Finally, De Giorgi--Nash--Moser theory leads to application of $S_{ }^{p}$
estimates (counterpart of $W^{2,p}$ estimates for parabolic PDEs) to control
$\nabla_{v}f$ for the uniqueness. The main novelty of our paper is to give a
more direct and simplified approach with only $L^{2}$ energy and
$S_{ }^{p}$ estimates in the perturbative regime, bypassing the
recent De Giorgi--Nash--Moser theory for the Landau equation.

\subsection{Statement of the Main Theorem}

We are now ready to state our main theorem, which is a low-regularity global well-posedness and stability result for the model Vlasov-Poisson-Landau system \eqref{Eq:Vlasov-Landau-model} and \eqref{Eq:Poisson-model}.

\begin{theorem}[Main Theorem] \label{Thm:Main-thm}
Assume that for some sufficiently large velocity-weight exponent $\vartheta_0\in\R$, the initial data $f_0(x,v):\Omega\times\R^3 \rightarrow\R$ satisfies the smallness assumption
\begin{equation} \label{smallness-assumption}
\|f_0\|_{\infty,\vartheta_0}
\,+\, \|\nabla_{\!x} f_0\|_{\infty,\vartheta_0} \,+\,\|\nabla_{\!v}f_0\|_{\infty,\vartheta_0}\,+\,\|\nabla_{\!v}^2f_0\|_{\infty,\vartheta_0}
\,\leq\, \varepsilon_0
\end{equation}
for a sufficiently small constant $\varepsilon_0\ll 1$.
Let $F_0(x,v) = \mu + \sqrt{\mu}\,f_0(x,v) \geq 0$ and has the same mass as the Maxwellian $\mu$.
Then we have the following conclusions:
\begin{itemize}
  \item {\rm (Existence \& Uniqueness).} There exists a unique global solution $f(t,x,v)$ to the Vlasov-Poisson-Landau system \eqref{Eq:Vlasov-Landau_f}-\eqref{Eq:Poisson_f} for perturbation with the specular-reflection boundary condition \eqref{Specular-BC_f} and the conservation laws \eqref{mass-conservation_f}-\eqref{energy-conservation_f}.
  Also,
  $$
  F(t,x,v) = \mu + \sqrt{\mu}\,f(t,x,v) \geq 0
  $$
  satisfies the system \eqref{Eq:Vlasov-Landau-model}-\eqref{Eq:Poisson-model} with \eqref{Specular-BC_F}.
  \vspace{2pt}
  \item {\rm (Energy estimates \& $L^2$ decay).} Moreover, the solution $f(t,x,v)$ satisfies the uniform weighted energy bounds, for $\vartheta < \vartheta_0 - \frac{3}{2}$,
  \begin{equation} \label{energy-bound}
  \sup_{t\in [0,\infty)} \mathcal{E}_{\vartheta}[f(t)] \,\leq\, C_{\vartheta}\, \mathcal{E}_{\vartheta}[f(0)] \,\lesssim\, \varepsilon_0^{\,2} ,
  \end{equation}
  and the (almost exponential) time-decay, if $\vartheta + k < \vartheta_0 - \frac{3}{2}$,
  \begin{equation} \label{L^2-decay}
  \|f(t)\|_{2,\vartheta} + \big\|\mathbf{E}_f(t)\big\|_{H^1_x}
  \,\lesssim\, \varepsilon_0 \left(1+\frac{t}{2k}\right)^{-k}
  \end{equation}
  for any $t\geq 0$.
  \vspace{2pt}
  \item {\rm ($L^\infty$ bounds \& pointwise decay).} For any $k\geq 0$, there is an increasing function $\vartheta'=\vartheta'(k)$ such that when $\vartheta + \vartheta' \leq \vartheta_0$, the weighted pointwise decay bounds hold:
  \begin{equation} 
  \|f(t)\|_{\infty,\vartheta} + \|\mathbf{E}_f(t)\|_{\infty} \,\lesssim\, \varepsilon_0 \left(1+t\right)^{-k}
  \end{equation}
  for all $t\geq 0$.
  \vspace{2pt}
  \item {\rm (Regularity results).} In addition, it holds that for any $t>0$
  \begin{equation} \label{Holder-estimate}
  \|f(t)\|_{C^{0,\alpha}\left(\Omega\times\R^3\right)} + \|\mathbf{E}_f(t)\|_{C^{1,\alpha}\left(\Omega\right)}
  \,\lesssim\, \varepsilon_0
  \end{equation}
  for some $\alpha\in (0,1]$, and
  \begin{equation} \label{Dv-L^infty}
  \|\nabla_v f\|_{L^\infty\left((0,\infty)\times\Omega\times\R^3\right)} \,\lesssim\, \varepsilon_0 .
  \end{equation}
  \item
  $\dt f$ and $\dt \mathbf{E}_f$ also satisfy all the estimates above.
\end{itemize}
\end{theorem}

\begin{remark}
(1) There exists $\bar\vartheta_0$ such that the system is globally well-posed for initial data satisfying $\vartheta_0\geq\bar\vartheta_0$. \\
(2) For fixed $\vartheta_0$, this theorem guarantees that the solution decays at least in the rate of $k(\vartheta_0)$, which goes to infinity as  $\vartheta_0\rt\infty$. Hence, if we desire a faster decaying solution, we should require that the initial data belongs proper weighted $L^{\infty}$ space. As \cite{Guo.Tice2013==} states, this is called the almost exponential decay. Better data, faster decay.
\end{remark}

\subsection{Bootstrap Proposition}

We implement the following bootstrap scheme:

\begin{proposition}[Bootstrap Estimates]
                \label{Prop:Bootstrap-est}
Let $f(t,x,v)$ be a solution to the VPL-specular problem \eqref{Eq:Vlasov-Landau_f}, \eqref{Eq:Poisson_f}, and \eqref{Specular-BC_f} (or \eqref{Eq:linearized-VL_f}, \eqref{Eq:Poisson_g}, and \eqref{Specular-BC_f}) on some time interval $[0,T],\, T\!\geq\! 1$, with initial data $f_0$ satisfying the assumption \eqref{smallness-assumption}.

Assume also that, for any $t\in [0,T]$, the solution $f$ (and $g$) satisfies the bootstrap hypothesis:
\begin{align}
\langle t\rangle^{k_1} \big\|f(t)\big\|_{2,\vartheta_1}
\,+\, \langle t\rangle^{k_1} \big\|\partial_t f(t)\big\|_{2,\vartheta_1} &\,\lesssim\, \varepsilon_1 , \label{Bootstrap-assp:smallness-decay-L^2} \\
\langle t\rangle^{k_2} \big\|f(t)\big\|_{\infty,\vartheta_2}
\,+\, \langle t\rangle^{k_2} \big\|\partial_t f(t)\big\|_{\infty,\vartheta_2} &\,\lesssim\, \varepsilon_1 , \label{Bootstrap-assp:smallness-decay-L^infty}
\end{align}
and
\begin{align}
\|f\|_{S^{p}_{ }\left((0,T)\times\Omega\times\R^3\right)}
\,+\, \|\partial_t f\|_{S^{p}_{ }\left((0,T)\times\Omega\times\R^3\right)} &\,\lesssim\, \varepsilon_1 ,  \label{Bootstrap-assp:finteness-S^p} \\
\|D_v f\|_{L^\infty\left((0,T)\times\Omega\times\R^3\right)}
\,+\, \|D_v \partial_t f\|_{L^\infty\left((0,T)\times\Omega\times\R^3\right)}
&\,\lesssim\, \varepsilon_1 , \label{Bootstrap-assp:finteness-Dv}
\end{align}
where $\langle t\rangle := \sqrt{1+t^2}$, $\vartheta_1,\vartheta_2$ are two sufficiently large constants, $k_i \geq 0\; (i=1,2)$, $p>14$, $\varepsilon_0 = \varepsilon_1^{\,q} \ll 1$ for some
$q>1$, and see \eqref{S^p-norm} for the definition of the $S^{p}_{ }$ norm.

Then the following improved bounds hold for any $t\in [0,T]$:
\begin{align}
\langle t\rangle^{k_1} \big\|f(t)\big\|_{2,\vartheta_1}
+ \langle t\rangle^{k_1} \big\|\partial_t f(t)\big\|_{2,\vartheta_1}
&\,\lesssim\, \varepsilon_1^{\,r} , \label{Bootstrap-est:smallness-decay-L^2} \\
\langle t\rangle^{k_2} \big\|f(t)\big\|_{\infty,\vartheta_2}
+ \langle t\rangle^{k_2} \big\|\partial_t f(t)\big\|_{\infty,\vartheta_2}
&\,\lesssim\, \varepsilon_1^{\,r} , \label{Bootstrap-est:smallness-decay-L^infty}
\end{align}
and
\begin{align}
\|f\|_{S^{p}_{ }\left((0,T)\times\Omega\times\R^3\right)}
\,+\, \|\partial_t f\|_{S^{p}_{ }\left((0,T)\times\Omega\times\R^3\right)} &\,\lesssim\, \varepsilon_1^{\,s} ,  \label{Bootstrap-est:finteness-S^p} \\
\|D_v f\|_{L^\infty\left((0,T)\times\Omega\times\R^3\right)}
\,+\, \|D_v \partial_t f\|_{L^\infty\left((0,T)\times\Omega\times\R^3\right)}
&\,\lesssim\, \varepsilon_1^{\,s} , \label{Bootstrap-est:finteness-Dv}
\end{align}
and
\begin{equation}
                    \label{eq12.33}
    \|f(t)\|_{C^{\,0,1}\left(\Omega\times\R^3\right)}
     +\|D_vf(t)\|_{C^{\,0,\alpha}\left(\Omega\times\R^3\right)}\,\lesssim\, \varepsilon_1^{\,s},\quad\alpha=1-14/p,
\end{equation}
for some $r,s>1$, where the implicit constants in \eqref{Bootstrap-est:smallness-decay-L^2}-\eqref{Bootstrap-est:finteness-Dv} are independent of $\varepsilon_1$.
\end{proposition}

\begin{remark}
Given that the bootstrap proposition above holds true, Theorem \ref{Thm:Main-thm} follows from a local existence result combined with a continuity argument (see Section\,\ref{Sec:Main-Thm-Pf} for the proof).
The local existence is presented in Theorem \ref{local-wellposedness} and it can be obtained by using the well-posedness for the linear equation (cf. \cite[Section\;2\,--\,5]{Guo.Hwang.Jang.Ouyang2020} and  \cite{Dong.Guo.Yastrzhembskiy2021}) with a standard iteration argument under the initial assumption \eqref{smallness-assumption}.
The majority of the rest of this paper is devoted to the proof of Proposition \ref{Prop:Bootstrap-est}.
\end{remark}

The bounds (\ref{Bootstrap-assp:smallness-decay-L^2})-(\ref{Bootstrap-assp:smallness-decay-L^infty}) and (\ref{Bootstrap-est:smallness-decay-L^2})-(\ref{Bootstrap-est:smallness-decay-L^infty}) are our (low-order) weighted energy and pointwise decay estimates, which guarantees existence. They are essentially Sobolev-type estimates, which can be obtained through the so-called ``two-tier'' energy method. It is presented in Section \ref{Sec:Energy-est-decay}.

The bounds (\ref{Bootstrap-assp:finteness-S^p})-(\ref{Bootstrap-assp:finteness-Dv}) and (\ref{Bootstrap-est:finteness-S^p})-(\ref{Bootstrap-est:finteness-Dv}) provide uniform control on the velocity-derivatives, which ensures the uniqueness. Also, it provides further regularity results. The proof is based on the ultra-parabolic equations and the $S^{p}$ theory.

\subsection{Overview and Idea of the Proof}
We now discuss the strategy for our proofs and some of the main features of our arguments.

{\it \underline{Upshots of the Proof:}}
\begin{itemize}
    \item Bootstrap scheme: The control of various norms in our paper is tangled together: \\
    (weighted)-$L^2$ decay $\rightarrow$ $S^p$ $\rightarrow$ $C^{\alpha}$ \& $D_v$\;regularity $\rightarrow$ (weighted)-$L^\infty$\;decay $\rightarrow$ (weighted)-$L^2$ decay.
    In fact, each step requires almost all the other estimates.
    In particular, the justification of uniqueness relies on (weighted)-$L^2$, $D_v$ and (weighted)-$L^\infty$ bounds.
    \item Energy estimate: We delicately implement the celebrated ``macro-micro decomposition method'' (kernel estimate by contradiction argument) to the Vlasov-Poisson-Landau system to obtain the weighted $L^2$ decay:
    \begin{itemize}
        \item In order to handle the problematic nonlinear term containing the electric field, we need to multiply $e^{\phi}$ to combine it with $v\cdot\nabla_{\!x}f$. This further requires the estimates of $\partial_t \phi \,\sim\, \partial_t f$.
        \item Since the $\sigma$-norm is not strong enough to control the $L^2$ norm, we need to introduce weights and use interpolation, leading to almost exponential decay.
        Also, we require decay of $\partial_t f$ to be sufficiently fast.
        \item ``Two-tier energy method'': decay $\rightarrow$ energy bound $\rightarrow$ decay
        \item $\partial_t f$ estimate: combine the counterpart of $f$
    \end{itemize}
    \item $S^p$ estimate: We greatly simplify the $L^2$--$L^\infty$--$C^{\alpha}$ framework in \cite{Kim.Guo.Hwang2020} and \cite{Guo.Hwang.Jang.Ouyang2020}. Instead of the De Giorgi-Nash-Moser method, we use $S^p$ embedding theorem to obtain H\"{o}lder continuity.
    \item Specular boundary problem: The $S^p$ estimate is based on an extension of the domain $\Omega$ beyond the boundary, which is achievable for the specular case, using the flattening-reflection technique in \cite{Guo.Hwang.Jang.Ouyang2020}.
\end{itemize}

\subsection{Organization of the Paper}

Our paper is organized as follows. In Section\;\ref{Sec:Prelim}, we present some preliminary lemmas about the Landau operators. The proof of Proposition \ref{Prop:Bootstrap-est} is given in Sections \ref{Sec:Energy-est-decay} -- \ref{Sec:S^p-est}.
In Section \ref{Sec:Energy-est-decay}, we justify the weighted energy estimate (\ref{Bootstrap-est:smallness-decay-L^2}). In Sections \ref{Sec:Lemmas} -- \ref{Sec:S^p-est}, we apply the $S^p$ theory in order to justify the uniqueness. In particular, in Sections \ref{Sec:Lemmas} and \ref{Sec:Boundary-Extension}, we adapt the general $S^p$ theory to our equations. In Section \ref{Sec:S^p-est}, we prove the $S^p$ bounds (\ref{Bootstrap-est:finteness-S^p}), the $L^\infty$ bounds  (\ref{Bootstrap-est:finteness-Dv}) for $D_v f$ and $D_v\partial_t f$, and the H\"older estimates \eqref{eq12.33}. We then justify the weighted $L^{\infty}$ estimates \eqref{Bootstrap-est:smallness-decay-L^infty}, which is also necessary for the uniqueness proof.
Finally, in Section \ref{Sec:Main-Thm-Pf}, we prove the main theorem.

\section{Preliminaries: Basic Properties and Estimates} \label{Sec:Prelim}

For convenience, we collect some basic facts and estimates in this subsection,
including the structure of the Landau collision kernel, from which we may characterize the $\sigma$-norm of dissipation and hereby estimate the Landau operators in terms of this norm.
We will omit some of those proofs but instead refer the reader to \cite[Section\;2,\,3]{Guo2002(=)}, \cite[Section\;2]{Kim.Guo.Hwang2020}, and \cite[Section\;2]{Degond.Lemou1997} for more details.
Moreover, we prepare some preliminary $L^p$ and elliptic estimates for our proofs of the energy and $S^p$ bounds in later sections.
Throughout this subsection, if not otherwise stated,
$f$, $g$, and $g_i$ represent general functions of certain variables.

\begin{lemma}[Embedding of Weighted Spaces]
Let $f$ be a function defined for $(x,v)\in\Omega\times\R^3$, and $p\in (1,\infty)$.
For any $\vartheta\in\R$ and $l > {3}/{p}$, there is a uniform constant $C_p >0$ (depending only on the domain $\Omega$, $p$, and $l$) such that
\begin{equation*}
\|f\|_{p,\vartheta} \,\leq\, C_p \|f\|_{\infty,\vartheta+l} .
\end{equation*}
\end{lemma}

\begin{remark}
As a corollary when $p=2$, we see that the weighted $L^\infty$ spaces can be embedded into weighted $L^2$ spaces at the cost of (at least) a finite $\left(\frac{3}{2}\right)^{\!+}$order of velocity-weight loss.
In particular, for initial data, our assumption on weighted $L^\infty$ norm guarantees control of $L^2$ norm.
\end{remark}

\begin{proof}
We estimate
\begin{align*}
\|f\|_{p,\vartheta}^{p}
& = \iint_{\Omega\times\R^3} \langle v\rangle^{p\vartheta} |f(x,v)|^p \,\dd v\dd x \\
& = \iint_{\Omega\times\R^3} \langle v\rangle^{-pl} \left|\langle v\rangle^{\vartheta+l} f(x,v)\right|^p \dd v\dd x \\
& \leq\, \|f\|_{\infty,\vartheta+l}^{p}\cdot\! \iint_{\Omega\times\R^3} \langle v\rangle^{-pl} \,\dd v\dd x \\
& =\, C_{p}^{p}\, \|f\|_{\infty,\vartheta+l}^{p} ,
\end{align*}
where the constant
$$C_p = C_p(\Omega,l) :=\, |\Omega|^{1/p} \left(\int_{\R^3} \langle v\rangle^{-pl} \,\dd v \right)^{\frac{1}{p}} <\infty $$
for $l > {3}/{p}$, and $C_p \sim |\Omega|^{1/p} \big(\frac{1}{pl-3}\big)^{1/p}$.
\end{proof}

Next is an estimate to be applied repeatedly in the following lemmas about the properties related to the Landau kernel
(cf.\;\,\cite[Proposition\;2.2\,--\,2.4]{Degond.Lemou1997}, \cite[Lemma\;2,\,3]{Guo2002(=)} and \cite[Lemma\;2.2\,--\,2.5]{Kim.Guo.Hwang2020}).

\begin{lemma} \label{Lem:sigma-model-est}
Let $\vartheta >-3$, $k(v)\in C^{\infty}\big(\R^3\backslash\{0\}\big)$ and $m(v)\in C^{\infty}\big(\R^3\big)$.
Assume that for any multi-index $\beta\geq 0$,
\begin{equation} \label{k-m-condition}
\begin{split}
\big|D^{\beta} k(v)\big| &\;\leq\; C_{\beta}'\, |v|^{\vartheta-|\beta|} , \\
\big|D^{\beta} m(v)\big| &\;\leq\; C_{\beta}'\, e^{-\tau_{\beta} |v|^2} 
\end{split}
\end{equation}
for some $C_{\beta}' >0$ and $\tau_{\beta} >0$.
Then there is $C_{\beta} >0$ such that
\begin{equation} \label{k*m-est}
\big|D^{\beta} \big[k\ast m\big](v)\big| \;\leq\; C_{\beta}\, \langle v\rangle^{\vartheta-|\beta|} .
\end{equation}
\end{lemma}

\begin{corollary}[Pointwise Estimate of $\sigma_{\!G}^{ij}$ and $a_g^{i}$] \label{Cor:sigma_G-a_g}
Let the (generalized) Landau kernel $\Phi(v)$ be given by \eqref{Landau-kernel-general}\, for $-3 \leq \gamma \leq 1$.
Then $\sigma^{ij}(v)$ and $\sigma^{i}(v)$ defined in \eqref{diffusion-matrix-sigma-1}\,--\,\eqref{diffusion-matrix-sigma-3} are smooth functions of $v$ such that
\begin{equation} \label{sigma-est-1}
\big|D_{v}^{\beta} \sigma^{ij}(v) \big| \,+\, \big|D_{v}^{\beta} \sigma^{i}(v) \big|
\;\leq\; C_{\beta}\, \langle v\rangle^{\gamma+2-|\beta|} 
\end{equation}
for some $C_{\beta} >0$ with $|\beta|\in\mathbb{N}\cup \{0\}$.

In our case of the Coulomb interaction (\,i.e. $\gamma=-3$), if $G = \mu +\! \sqrt{\mu}g$ with \:$\sup_{t}\|g\|_{\infty} \leq\varepsilon \ll 1$, then
\begin{equation} \label{sigma-est-2}
\big|D_{v}^{\beta} \sigma_{\!G}^{ij}(t,x,v) \big|
\;\lesssim\; \langle v\rangle^{-1-|\beta|} \,\leq\, 1 
\end{equation}
for $|\beta| = 0,1$, and
\begin{equation} \label{sigma-est-3}
\big|a_g^{i}(t,x,v)\big|
\;\lesssim\; \sup_{t}\|g\|_{\infty}\cdot \langle v\rangle^{-1} \,\leq\, \varepsilon \;,
\end{equation}
where $\sigma_{\!G}^{ij}$ is given in \eqref{sigma_G} and $a_g^i$ in \eqref{a_g^i}.
\end{corollary}

\begin{proof}
Applying Lemma\;\ref{Lem:sigma-model-est}\, to $k = \Phi^{ij}$ with
\begin{equation*}
\big|D^{\beta} \Phi^{ij}(v)\big| \;\leq\; C_{\beta}'\, |v|^{\gamma+2-|\beta|}\quad \text{for}\,\,v\neq 0
\end{equation*}
and $m = \mu, v_j\mu$ satisfying (\ref{k-m-condition}),
the bound (\ref{sigma-est-1}) follows immediately from (\ref{k*m-est}).

For (\ref{sigma-est-2}) and (\ref{sigma-est-3}), from the definitions we clearly have
\begin{equation*}
\big|D_{v}^{\beta} \sigma_{\!G}^{ij}(t,x,v) \big|
\;\leq\; \big|D_{v}^{\beta} \sigma^{ij}(v) \big|
\,+\, \big|D_{v}^{\beta} \sigma_{\!\!\sqrt{\mu}g}^{ij}(t,x,v) \big|
\;\leq\; C_{\beta} \big(1+\,\sup_{t}\|g\|_{\infty} \big)\!\cdot \langle v\rangle^{-1-|\beta|} \;,
\end{equation*}
by following the proof of Lemma\;\ref{Lem:sigma-model-est} with slight modification,
and
\begin{equation*}
\big|a_g^{i}(t,x,v)\big|
\,\leq\, \big|\Phi^{ij}\!\ast\!\big[v_j\mu^{1/2}g\big] \big|
+ \big|\Phi^{ij}\!\ast\!\big[\mu^{1/2}\partial_j g\big] \big|
\;\lesssim\; \sup_{t}\|g\|_{\infty}\!\cdot \langle v\rangle^{-1} \;,
\end{equation*}
where (for the second term) we transfer the derivatives on $g$ using integration by parts.
\end{proof}

\begin{lemma}[Structure of the Diffusion Matrix] \label{Lem:Landau-kernel-structure}
Let $\sigma = \big[\sigma^{ij}\big]_{1\leq i,\,j\leq 3}$ be the diffusion matrix in \eqref{diffusion-matrix-sigma-2}.
For any $v\in\R^3$, $\sigma(v)$ is a $3\!\times\! 3$ real symmetric matrix.
Thus it can be diagonalized by an orthogonal matrix consisting of associated eigenvectors.

The spectrum of $\sigma(v)$ consists of a simple eigenvalue $\lambda_1(v)>0$ associated with the eigenvector $v$, and a double eigenvalue $\lambda_2(v)>0$ associated with the eigenspace $v^{\perp}$.
Furthermore, the eigenvalues can be written explicitly as
\begin{align}
\lambda_1(v) &\,=\, \int_{\R^3}\! \Big\{1-|\hat{v}\cdot\hat{w}|^2 \Big\}\,\mu(v\!-\!w)\,|w|^{\gamma+2}\,\dd w \;, \label{sigma-eigenvalues-1} \\
\lambda_2(v) &\,=\, \int_{\R^3}\! \Big\{1-\frac{1}{2}\,|\hat{v}\times\hat{w}|^2 \Big\}\,\mu(v\!-\!w)\,|w|^{\gamma+2}\,\dd w \label{sigma-eigenvalues-2} \;,
\end{align}
where $\hat{v} := \frac{v}{|v|}$, $\hat{w} := \frac{w}{|w|}$.
Asymptotically, as $|v|\rightarrow\infty$,
\begin{equation*}
\lambda_1(v) \,\sim\, c_1 \langle v\rangle^{\gamma} , \quad
\lambda_2(v) \,\sim\, c_2 \langle v\rangle^{\gamma+2} 
\end{equation*}
for some constants $c_1, c_2 >0$.

For $\gamma=-3$, and $G = \mu +\! \sqrt{\mu}g$ with $\sup_{t}\|g\|_{\infty} \leq\varepsilon \ll 1$,
the matrix $\sigma_{\!G}(t,x,v)$ has three positive eigenvalues satisfying that for all $t\geq 0$, $x\in\Omega$, and for any $v\in\R^3$,
\begin{equation} \label{sigma-eigenvalue-bound}
\frac{1}{C}\, \langle v\rangle^{-3} \,\leq\, \lambda_{G}(t,x,v) \,\leq\, C \langle v\rangle^{-1}
\end{equation}
for some constant $C>0$,
meaning that $\sigma_{\!G}$ is positive-definite/elliptic (but not uniformly elliptic in $v\in\R^3$, i.e. does not have a strictly positive lower bound independent of $v$).
\end{lemma}

\begin{corollary}[Lower Bound for the $\sigma$-norm] \label{Cor:sigma-norm-bound}
For any $3d$-\,vector field $\mathbf{g}(v) = \langle g_1,g_2,g_3\rangle$, let
$\mathsf{q}_{\sigma}[\mathbf{g}](v) :=\, \mathbf{g}^{T}\!\sigma\mathbf{g} \,= \sum_{i,j=1}^{3}\! \sigma^{ij}g_i\, g_j$ denote the quadratic form associated with the diffusion matrix $\sigma(v)$.
Denote also by $P_{v\,}\mathbf{g} := (\mathbf{g}\cdot\hat{v})\:\hat{v}$ the projection of $\mathbf{g}(v)$ onto the subspace spanned by the vector $v$.
Then
\begin{equation*}
\mathsf{q}_{\sigma}[\mathbf{g}](v) \,=\; \lambda_1(v)\, \big|P_{v\,}\mathbf{g}\big|^2
\,+\, \lambda_2(v)\, \big|[I\!-\!P_{v}]\,\mathbf{g}\big|^2 ,
\end{equation*}
where $\lambda_1(v), \lambda_2(v)$ are eigenvalues of $\sigma(v)$ given by \eqref{sigma-eigenvalues-1} and \eqref{sigma-eigenvalues-2}.

Let $|\cdot|_{\sigma,\vartheta}$ and $\|\cdot\|_{\sigma,\vartheta}$ be the weighted $\sigma$-norms defined in Section\;\ref{SubsubSec:Dissipation-Energy}.
Then there exists $C_{\vartheta} >0$ such that
\begin{equation*}
\begin{split}
|g|_{\sigma,\vartheta}^2
&\,\geq\; C_{\vartheta} \left\{\, \left| \langle v\rangle^{\vartheta-\frac{3}{2}} \big|P_{v}(\nabla_{\!v}g)\big| \right|_{2}^2
+\, \left| \langle v\rangle^{\vartheta-\frac{1}{2}} \big|[I\!-\!P_{v}](\nabla_{\!v}g)\big| \right|_{2}^2
+\, \left| \langle v\rangle^{\vartheta-\frac{1}{2}} g \right|_{2}^2 \,\right\} \\
&\;\gtrsim\; |g|_{2,\vartheta-\frac{1}{2}}^2 \;,
\end{split}
\end{equation*}
and therefore,
\begin{equation*}
\|g\|_{\sigma,\vartheta} \,\gtrsim\, \|g\|_{2,\vartheta-\frac{1}{2}} \;.
\end{equation*}
\end{corollary}

\begin{remark}
Although the $\sigma$-norm contains first-order velocity derivatives, it is still not strong enough to control the $L^2$\;norm. However, we manage to bound the $L^2$ by this $\sigma$-norm at the cost of a minimal half power of weight loss.
\end{remark}

The next two lemmas on the Landau operators are useful in the energy estimates in Sections \ref{Sec:Energy-est-decay} \ref{Sec:Main-Thm-Pf}).
We first record a basic estimate about the linear collision operator $L$
(cf. the proofs of \cite[Lemma\;6]{Guo2002(=)} and \cite[Lemma\;2.7]{Kim.Guo.Hwang2020}).

\begin{lemma}[Linear Estimate for $L$] \label{Lem:L-est}
Let $L = -A-K$ be given by \eqref{L-opt}\,--\,\eqref{K-opt}, and let $\vartheta\in\R$.
For any small $\delta>0$, there exists $C_{\vartheta,\delta} = C_{\vartheta}(\delta) >0$ such that
\begin{align*}
(1-\delta)|g|_{\sigma,\vartheta}^2 \,-\, C_{\vartheta,\delta}\, |\mu g|_{2}^2 \;\leq -\, \big\langle w^{2\vartheta} Ag, g \big\rangle \leq (1+\delta)|g|_{\sigma,\vartheta}^2 \,+\, C_{\vartheta,\delta}\, |\mu g|_{2}^2 \;
\end{align*}
and
\begin{align*}
\big|\big\langle w^{2\vartheta} Kg_1, g_2 \big\rangle \big|
&\;\leq\; \big(\delta|g_1|_{\sigma,\vartheta} + C_{\vartheta,\delta} |\mu g_1|_{2} \big)\, |g_2|_{\sigma,\vartheta} \;.
\end{align*}
Here the Maxwellian $\mu$ is a convenient choice of a mollified characteristic function of a ball, which can be replaced by any function of $v$ with sufficiently fast decay at infinity.
As a consequence (by taking $\delta\leq \frac{1}{4}$), we have
\begin{equation} \label{L-est}
\frac{1}{2}\, |g|_{\sigma,\vartheta}^2 \,-\, C_{\vartheta} |g|_{\sigma}^2
\;\leq\; \big\langle w^{2\vartheta} Lg, g \big\rangle \,
 \;\leq\; \frac{3}{2}\, |g|_{\sigma,\vartheta}^2 \,+\, C_{\vartheta} |g|_{\sigma}^2 \;,
\end{equation}
and
\begin{equation} \label{L-est-2}
\frac{1}{2}\, |g|_{\sigma,\vartheta}^2 \,-\, C_{\vartheta} |g|_{2,\vartheta}^2
\;\leq\; \big\langle w^{2\vartheta} Lg, g \big\rangle \,
 \;\leq\; \frac{3}{2}\, |g|_{\sigma,\vartheta}^2 \,+\, C_{\vartheta} |g|_{2,\vartheta}^2 \;.
\end{equation}
\end{lemma}

For the nonlinear collision operator $\Gamma$, we estimate $\left(w^{2\vartheta} \Gamma[g_1,g_2],g_3 \right)$ in terms of $\|\cdot\|_{\infty}$, $\|\cdot\|_{2,\vartheta}$, and $\|\cdot\|_{\sigma,\vartheta}$ without higher-order regularity.
Note that $\Gamma[g_1,g_2]$ is non-symmetric, so we need to estimate this nonlinear term in two different ways, with the energy norm on two variables respectively.
(cf. \cite[Theorem\;3]{Guo2002(=)} and \cite[Theorem\;2.8]{Kim.Guo.Hwang2020})

\begin{lemma}[Nonlinear Estimate for $\Gamma$] \label{Lem:Gamma-est}
Let $\Gamma$ be defined as in \eqref{Gamma-opt}, then
for every $\vartheta\in\R$, there exists $C_{\vartheta} >0$ such that
\begin{equation} \label{Gamma-est-1}
\left|\left(w^{2\vartheta} \Gamma[g_1,g_2],g_3 \right)\right|
\;\leq\; C_{\vartheta}\, \|g_1\|_{\infty} \|g_2\|_{\sigma,\vartheta} \|g_3\|_{\sigma,\vartheta} \;.
\end{equation}
Moreover, for any $\vartheta \leq -2$, we have
\begin{equation} \label{Gamma-est-2}
\left|\left(w^{2\vartheta} \Gamma[g_1,g_2],g_3 \right)\right|
\;\leq\; C_{\vartheta} \big( \|g_2\|_{\infty} \!+\! \|D_v g_2\|_{\infty} \big) \!\cdot
\min\big\{\|g_1\|_{2,\vartheta}, \|g_1\|_{\sigma,\vartheta} \big\}\, \|g_3\|_{\sigma,\vartheta} \;.
\end{equation}
\end{lemma}

The following estimates will be used in the proof of the $S^p_{ }$ bound in Section\;\ref{SubSec:S^p-bound}.

\begin{lemma}[$L^p$ Estimate of $\bar{K}_{\!g} f$] \label{Lem:K_g-L^p}
Let $\bar{K}_{\!g} f$ be defined as in \eqref{K_g-opt} with $g$ satisfying the assumption \eqref{Bootstrap-assp:smallness-decay-L^infty}.
Then for every $\vartheta\geq 0$ and $1\leq p\leq\infty$, it holds that
\begin{equation} \label{Est:K_g-L^p}
\begin{split}
\big\| \bar{K}_{\!g} f \big\|_{L^p_{t,x,v}(|v|\sim n)}
&\,\lesssim\; n^{-\vartheta} \left(\|f\|_{L^p_{t,x,v}} + \big\|D_vf\big\|_{L^p_{t,x,v}} + \big\|\langle v\rangle^{\vartheta}f\big\|_{L^p_{t,x,v}(|v|\sim n)} \right) \\
&\,\lesssim\; n^{-\vartheta} \left(\big\|D_vf\big\|_{L^p_{t,x,v}} + \big\|\langle v\rangle^{\vartheta}f\big\|_{L^p_{t,x,v}} \right) 
\end{split}
\end{equation}
for any $n\in\mathbb{N}$.
\end{lemma}

\begin{proof}
We split the operator as $\bar{K}_{\!g} = K + J_g$, where
\begin{align*}
Kf :=&\, -\mu^{\!-1/2}\,\partial_i\!\left\{\mu\bigg[\Phi^{ij}\!\ast\!\left\{\mu^{1/2}\big[\partial_j f+v_j f\big]\right\}\bigg]\right\} \\
=&\; 2v_i\mu^{1/2}\bigg[\Phi^{ij}\!\ast\!\left\{\mu^{1/2}\big[\partial_j f+v_j f\big]\right\}\bigg]
- \mu^{1/2}\bigg[\partial_i \Phi^{ij}\!\ast\!\left\{\mu^{1/2}\big[\partial_j f+v_j f\big]\right\}\bigg] \\[3pt]
J_g f :=&\; \partial_i \sigma^i f - \sigma^{ij}v_i v_j f
 -\partial_i\left\{\Phi^{ij}\!\ast\!\big[\mu^{1/2}\partial_j g\big]\right\}f + \left\{\Phi^{ij}\!\ast\!\big[v_i\mu^{1/2}\partial_j g\big]\right\}f .
\end{align*}

We will first estimate $Kf$ as
\begin{equation} \label{K-L^p-est}
\big\|K f \big\|_{L^p(|v|\sim n)}
\,\lesssim\, n^{-\vartheta} \left(\,\|f\|_{L^p} + \big\|D_vf\big\|_{L^p} \right).
\end{equation}
Note that for $\alpha \in (0,1)$, $\vartheta\geq 0$, and $k=0,1$,
\begin{equation*}
\langle v\rangle^{k} \mu^\alpha \big|_{|v|\sim n}
\,\lesssim\, n^{-\vartheta} \langle v\rangle^{\vartheta+k} \mu^\alpha
\,\lesssim\, n^{-\vartheta} .
\end{equation*}
Also, in light of Lemma\;\ref{Lem:Landau-kernel-structure} which provides spectrum of the Landau kernel $\Phi(v)$,
it suffices to show that
\begin{equation*}
\left\|\mu^\alpha\cdot \left\{|v|^{\lambda}\ast \big[\mu^{\!\gamma} h \big] \right\}(v) \right\|_{L^p_{v}} \,\lesssim\, \|h\|_{L^p_{v}} ,
\end{equation*}
where $\alpha, \gamma \in (0,1)$, $\lambda\in [-3,-1]$, and $h(t,x,v)$ represents either $f$ or $\partial_{v_i}f$.
Using the H\"{o}lder inequality, the $p$-th power of LHS above is bounded by
\begin{equation*}
\begin{split}
& \int_{\R^3} \mu^{\alpha p}(v) \left|\int_{\R^3} \big|v-v'\big|^{\lambda} \mu^{\!\gamma}(v') h(v') \,\dd v'\right|^p \!\dd v \\
\,\leq\,& \int_{\R^3} \mu^{\alpha p}(v) \left(\int_{\R^3}\!\big|v-v'\big|^{\lambda p'}\mu^{\gamma p'}(v') \dd v'\right)^{\frac{p}{p'}}\! \left(\int_{\R^3}\!\big|h(v')\big|^p \dd v'\right) \dd v \\
\,\lesssim\,&\, \|h\|_{L^p_{v}}^p \cdot \int_{\R^3} \mu^{\alpha p}(v) \langle v\rangle^{\lambda p} \,\dd v \\
\,\lesssim\,&\, \|h\|_{L^p_{v}}^p .
\end{split}
\end{equation*}
The second inequality above follows in a similar manner as the proof of Lemma\;\ref{Lem:sigma-model-est} (cf. \cite[Lemma\;2]{Guo2002(=)}\,) with obvious modifications.

Now for the other part, we have
\begin{equation} \label{J_g-L^p-est}
\big\|J_g f \big\|_{L^p(|v|\sim n)} \,\leq\, \big(C +\, \sup_{t}\|g\|_{\infty}\big)\, \|f\|_{L^p(|v|\sim n)}
\,\lesssim\, n^{-\vartheta} \big\|\langle v\rangle^{\vartheta}f\big\|_{L^p(|v|\sim n)} .
\end{equation}
Here for the third term in $J_g f$, we use the fact that $\partial_i\partial_j \Phi^{ij}$ is a multiple of the Dirac delta function.
Finally, combining (\ref{K-L^p-est}) and (\ref{J_g-L^p-est}) gives us the first inequality of (\ref{Est:K_g-L^p}), which immediately yields the second one.
\end{proof}

The last estimate for the electric field will be used repeatedly across the paper.

\begin{lemma}[$L^p$ Estimate of $\mathbf{E}_f$] \label{Lem:E_f-L^p}
Let \;$\mathbf{E}_f := -\nabla_{\!x}\phi_f \,=\, \nabla_{\!x}\, \Delta_{x}^{-1}\! \int_{\R^3}\!\sqrt{\mu}\,f\,\dd v$\; be the electric field with the potential $\phi_f$ solved from the Poisson equation
\begin{equation*}
-\Delta_x\phi_f(t,x) = \int_{\R^3} \sqrt{\mu}\,f\,\dd v \,=: \rho[f](t,x)
\end{equation*}
with either zero-Dirichlet or zero-Neumann BC, then for $1<p<\infty$, it holds that for any $t\geq 0$,
\begin{equation*} 
\big\|\mathbf{E}_f(t)\big\|_{W^{1,p}_{x}} \,\lesssim\, \big\|f(t)\big\|_{L^p_{x,v}} .
\end{equation*}
Moreover, when $p=\infty$, we have
\begin{equation} \label{E_f-L^infty}
\|\mathbf{E}_f\|_{L^\infty_{t,x}} \,\lesssim\, \|f\|_{L^\infty_{t,x,v}} .
\end{equation}
\end{lemma}

\begin{proof}
For $1<p<\infty$, based on the $L^p$ estimate for elliptic equations, we have
\begin{equation*}
\begin{split}
\big\|\mathbf{E}_f(t)\big\|_{W^{1,p}_{x}} \,\lesssim\, \big\|\phi_f(t)\big\|_{W^{2,p}_{x}} &\,\lesssim\, \big\|\rho[f](t)\big\|_{L^p_{x}} \,=\, \left\|\int_{\R^3}\! \sqrt{\mu(v)}\,f(t,\cdot,v)\,\dd v \right\|_{L^p_{x}} \\
&\,\leq\, \int_{\R^3}\! \left\|\sqrt{\mu(v)}\,f(t,\cdot,v)\right\|_{L^p_{x}} \dd v \,=\,
\int_{\R^3}\! \sqrt{\mu(v)}\, \big\|f(t,\cdot,v)\big\|_{L^p_{x}} \dd v \\[3pt]
&\,\leq\, \big\|f(t)\big\|_{L^p_{x,v}} \!\cdot |\sqrt{\mu}\,|_{L^{p'}_v}
\,\lesssim\, \big\|f(t)\big\|_{L^p_{x,v}} .
\end{split}
\end{equation*}
The third and fourth inequalities are due to Minkowski's integral inequality and H\"{o}lder's inequality, respectively.

Additionally, in view of the continuous embedding $W^{1,p}_{x}(\Omega) \hookrightarrow L^{\infty}_{x}(\Omega)$ for $p>3$, we also have
\begin{equation*}
\begin{split}
\big\|\mathbf{E}_f(t)\big\|_{L^\infty_{x}} \,\lesssim\, \big\|\mathbf{E}_f(t)\big\|_{W^{1,p}_{x}}
&\,\lesssim\, \big\|\rho[f](t)\big\|_{L^p_{x}}
\,=\, \left\|\int_{\R^3}\! \sqrt{\mu(v)}\,f(t,\cdot,v)\,\dd v \right\|_{L^p_{x}} \\
&\,\leq\, \big\|f(t)\big\|_{\infty} \!\cdot |\Omega|^{\frac{1}{p}} \int_{\R^3}\!\! \sqrt{\mu}\,\dd v \\[3pt]
&\,\lesssim\, \big\|f(t)\big\|_{\infty} \,\leq\; \sup_{t\geq 0} \big\|f(t)\big\|_{\infty} 
\end{split}
\end{equation*}
for any $t\geq 0$, then (\ref{E_f-L^infty}) follows.
\end{proof}

\section{Energy Estimates: Weighted \texorpdfstring{$L^2$}{L2} Decay} \label{Sec:Energy-est-decay}

In this section we establish uniform weighted energy estimates and $L^2$ time-decay for the VPL-specular problem \eqref{Eq:Vlasov-Landau_f}, \eqref{Eq:Poisson_f}, and \eqref{Specular-BC_f} assuming small weighted $L^\infty$ norms and decay of the solutions $f$ and its time-derivative $\partial_t f$.

Our strategy is to adapt arguments in \cite{Cao.Kim.Lee2019}, \cite{Guo.Hwang.Jang.Ouyang2020}, and \cite{Kim.Guo.Hwang2020}.
In addition, we will utilize techniques in \cite[Section\;7]{Cao.Kim.Lee2019} to control the extra terms related to the electric field.

To start with, we rearrange the Vlasov-Landau equation \eqref{Eq:Vlasov-Landau_f} as
\begin{equation} \label{Eq:linearized-VL_f=g}
\partial_t f + v\cdot\nabla_{\!x}f + Lf - 2\sqrt{\mu}\,v\cdot \mathbf{E}_f
\,=\, \Gamma[f,f] - \mathbf{E}_f\cdot\nabla_{\!v}f + \big(v\cdot \mathbf{E}_f \big)f ,
\end{equation}
so that LHS has purely linear terms, leaving all the nonlinear terms on RHS.
Throughout this section, we will focus on the a priori estimates of the solution.

\begin{theorem}[Energy Estimates and $L^2$\;Time\,-Decay] \label{Thm:L^2-decay}
Let $f(t,x,v)$ be a solution to the VPL-specular problem  \eqref{Eq:Vlasov-Landau_f}, \eqref{Eq:Poisson_f}, and \eqref{Specular-BC_f} (or \eqref{Eq:linearized-VL_f}, \eqref{Eq:Poisson_g}, and \eqref{Specular-BC_f} with $g=f$) and additional conservation law of angular momentum \eqref{angular-momentumconservation_f} if $\Omega$ has a rotational symmetry on some time interval $[0,T],\, T\!\geq\! 1$, with initial data $f_0$ satisfying \eqref{smallness-assumption}.
Suppose also the bootstrap decay bounds \eqref{Bootstrap-assp:smallness-decay-L^2}-\eqref{Bootstrap-assp:smallness-decay-L^infty}, and in particular that
\begin{equation*} 
\sup_{t\in[0,T]} \|f(t)\|_{\infty,\bar{\vartheta}}
\,+ \sup_{t\in[0,T]} \|\partial_t f(t)\|_{\infty,\bar{\vartheta}}
\,+ \sup_{t\in[0,T]} \big\|\nabla_{\!v}\partial_t f(t) \big\|_{\infty} \,\lesssim\, \varepsilon_1
\end{equation*}
for some $\bar{\vartheta} \geq 3$, and where $\varepsilon_1 \ll 1$ is small enough.

Then
the uniform weighted energy bound holds
\begin{equation} \label{energy-bound-2}
\begin{split}
\sup_{t\in [0,T]} \mathcal{E}_{\vartheta}[f(t)]
\,:= \sup_{t\in [0,T]} \left( \big\|f(t)\big\|_{2,\vartheta}^{2} + \big\|\mathbf{E}_f(t)\big\|_{L^2_x}^2
\,+ \int_{0}^{t}\! \big\|f(\tau)\big\|_{\sigma,\vartheta}^{2} \,\dd\tau + \int_{0}^{t}\! \big\|\mathbf{E}_f(\tau)\big\|_{L^2_x}^2 \,\dd\tau \right)
\,\lesssim\,   \varepsilon_0^{2}.
\end{split}
\end{equation}
Furthermore, the solution decays almost exponentially with respect to time:
\begin{equation} \label{L^2-decay-2}
\left\|f(t)\right\|_{2,\vartheta} + \big\|\mathbf{E}_f(t)\big\|_{H^1_x}
\lesssim\, \varepsilon_0 \langle t\rangle^{-k} 
\end{equation}
for any $t\in [0,T]$, where $k$ was introduced in \eqref{L^2-decay}.  Moreover, we also have
\begin{align*}
    \mathcal{E}_{\vartheta}[\dt f(t)]\lesssim\, \varepsilon_0^{2},
\end{align*}
and
\begin{align*}
    \left\|\dt f(t)\right\|_{2,\vartheta} + \big\|\dt \mathbf{E}_f(t)\big\|_{H^1_x}
\lesssim\, \varepsilon_0 \langle t\rangle^{-k} .
\end{align*}
\end{theorem}

\begin{remark} \label{Rmk:L^2-decay}
The time-decay result (\ref{L^2-decay-2}) is called almost-exponential decay, in the sense that the solution decays with any polynomial rate in time.
This decay rate is optimal for the Landau-Poisson system in view of our current setting as well as method,
and the reason is twofold:

(1) On one hand, the instant energy functional at each time is stronger than the dissipation rate due to the particular structure of Landau operator, so we have to perform interpolation between hierarchies of weighted energy norms
such that dissipation can bound a power of energy, which yields only algebraic decay.

(2) On the other hand, in comparison to periodic domains, where more regular initial assumption grants faster decay (for example, exponential decay, cf. \cite{Strain.Guo2008}),
we can only resort to low-regularity techniques in the presence of specular-reflection boundary, which at best yield algebraic decay.

When the domain is rotational invariant, the conservation of angular momentum is necessary to ensure the validity of this theorem. This is mainly owing to the proof of Lemma \ref{Lem:abc}, which leads to the positivity of $L$ in Corollary \ref{Cor:coercivity-est-L}.
\end{remark}

It is known that the linear Landau operator $L$ given by (\ref{L-opt}) is a self-adjoint nonnegative operator on $L^2_v(\R^3)$.
Its null space (kernel) is a five-dimensional subspace of $L^2_v(\R^3)$ spanned by
$\left\{\sqrt{\mu},\, v\sqrt{\mu},\, |v|^2\!\sqrt{\mu}\, \right\}$,
which are called collision invariants.
We introduce the following notation:
\begin{definition}[Projection onto the null space of $L$]
Let $\mathbf{N}(L) := \big\{\,h\in L^2_v(\R^3) :\, Lh=0 \,\big\}$ denote the null space of the linear operator $L$ with basis
$e_0:=c_0\sqrt{\mu},\, e_i:=c_iv_i\sqrt{\mu}\;\,(i=1,2,3),\, e_4:=c_4|v|^2\sqrt{\mu}$ (with suitable normalization constants $c_0,\ldots,c_4$),
and write
\begin{equation*} 
\mathbf{N}(L) \,=\,  {\rm span} \left\{\sqrt{\mu},\, v_i\sqrt{\mu}\;\,(i=1,2,3),\, |v|^2\sqrt{\mu}\, \right\} .
\end{equation*}
For the function $f(t,x,v)$ with fixed $(t,x)$, we define the projection of $f$ in $L^2_v(\R^3)$ onto $\mathbf{N}(L)$ as
$$\mathcal{P}f\,(t,x,v) \,:=\, \sum_{i=0}^4 \big\langle f(t,x,\cdot\,), e_i \big\rangle\, e_i
\,=:\, \left\{\mathfrak{a}_f(t,x) + v\cdot \mathfrak{b}_f(t,x) + |v|^2\, \mathfrak{c}_f(t,x) \right\}\sqrt{\mu} ,$$
where
\begin{align*}
\mathfrak{a}_f(t,x) &\,:=\, \big\langle f(t,x,\cdot\,), e_0 \big\rangle
\,=\, c_0\int_{\R^3}\! f(t,x,\cdot\,)\sqrt{\mu}\, \dd v , \\
\mathfrak{b}_f^{i}(t,x) &\,:=\, \big\langle f(t,x,\cdot\,), e_i \big\rangle
\,=\, c_i\int_{\R^3}\! f(t,x,\cdot\,) v_i\sqrt{\mu}\, \dd v , \\
\mathfrak{c}_f(t,x) &\,:=\, \big\langle f(t,x,\cdot\,), e_4 \big\rangle
\,=\, c_4\int_{\R^3}\! f(t,x,\cdot\,) |v|^2\sqrt{\mu}\, \dd v .
\end{align*}
\end{definition}

We then recall a basic property of $L$ (see \cite[Lemma\;5]{Guo2002(=)} for the proof).
\begin{lemma}[Semi-Positivity of $L$] \label{Lem:L-semi-positivity}
There exists $\delta>0$ such that
\begin{equation*} 
\big\langle Lh, h \big\rangle \,\geq\, \delta\, \big|(I-\mathcal{P})h\big|_{\sigma}^2 
\end{equation*}
for a general function $h(v)$.
\end{lemma}

\subsection{Positivity of $L$}

A crucial step toward proving the $L^2$ decay is to estimate $\mathcal{P}f$ in terms of $(I-\mathcal{P})f$.
The following lemma is an adaptation of \,\cite[Proposition\;1]{Guo2010}\, and \,\cite[Proposition\;2]{Guo.Hwang.Jang.Ouyang2020(=)}. Since most of the proof is identical to that of \cite{Guo.Hwang.Jang.Ouyang2020(=)}, we will skip the details and only provide the key ideas and steps. The new terms related to the fields can be controlled in an obvious way.

\begin{lemma} \label{Lem:abc}
Under the same assumptions as in Theorem\;\ref{Thm:L^2-decay}, we have
\begin{equation} \label{Est:P<(I-P)}
\int_s^t\!\big\|\mathcal{P}f(\tau)\big\|_{\sigma}^2 \,\dd\tau \,+ \int_s^t\!\big\|\mathbf{E}_f(\tau)\big\|_{L^2_x}^2 \,\dd\tau
\;\lesssim\;  \int_s^t\!\big\|(I-\mathcal{P})f(\tau)\big\|_{\sigma}^2 \,\dd\tau 
\end{equation}
for all $0\leq s<t$ with $\abs{t-s} \in\mathbb{Z}^+$.
\end{lemma}

As an immediate consequence of the lemma above, we obtain the positivity of $L$ with Landau dissipation, a key ingredient in the proof of Theorem\;\ref{Thm:L^2-decay}.
\begin{corollary}[Coercivity Estimate on $L$] \label{Cor:coercivity-est-L}
With the same assumptions above, we have a sufficiently small constant $\delta'>0$ such that
\begin{equation} \label{coercivity-est-L}
\int_s^t\! \big(Lf(\tau),f(\tau) \big) \,\dd\tau
\;\geq\; \delta' \left\{\int_s^t\!\big\|f(\tau)\big\|_{\sigma}^2 \,\dd\tau \,+ \int_s^t\!\big\|\mathbf{E}_f(\tau)\big\|_{L^2_x}^2 \,\dd\tau\right\} 
\end{equation}
for all $0\leq s<t$ with $\abs{t-s} \in\mathbb{Z}^+$.
\end{corollary}

\begin{proof}[Sketch of Proof]
Recalling that $L$ is semi-positive by Lemma\;\ref{Lem:L-semi-positivity}, we clearly have
\begin{equation*} 
\big( Lf, f \big) \,\geq\, \delta\, \big\|(I-\mathcal{P})f\big\|_{\sigma}^2 \end{equation*}
for some $\delta>0$.
Then it suffices to bound the remaining part on $\mathcal{P}f$ by the counterpart of $(I-\mathcal{P})f$ using Lemma\;\ref{Lem:abc}.
\end{proof}

\begin{proof}[Proof of Lemma\;\ref{Lem:abc}]
    This proof is an adaptation of the proof of \cite[Proposition\;2]{Guo.Hwang.Jang.Ouyang2020(=)},
    so we only point out the main difference.

    We consider the equation \eqref{Eq:linearized-VL_f=g}.
    To handle the most problematic nonlinear term $(v\cdot\, \mathbf{E} )f$, we combine it with the linear streaming term $v\cdot\nabla_{\!x}f$ (which also contains an  extra $v$ factor) by multiplying both sides of \eqref{Eq:linearized-VL_f=g}\, by $e^\phi$, so we can rewrite the equation as
\begin{equation} \label{Eq:linearized-VL_f=f-e^phi=}
\partial_t \big(e^\phi f\big) + v\cdot\nabla_{\!x}\big(e^\phi f\big) + L\big(e^\phi f\big)
- 2 e^\phi \!\sqrt{\mu}\,v\cdot \mathbf{E}
\;=\; e^\phi\,\Gamma[f,f] - e^\phi\,\mathbf{E}\cdot\nabla_{\!v}f + e^\phi f\, \partial_t\phi .
\end{equation}

    We first justify the lemma for $s\!=\!0,\, t\!=\!1$.
    If the estimate (\ref{Est:P<(I-P)}) is not true no matter how small $\varepsilon_1$ is, then there exist a sequence of solutions $f_n$ to \eqref{Eq:linearized-VL_f=g} with $f=f_n$ such that for any $n$,
 	\begin{equation}\label{contradiction} \int
 	_{0}^{1} \|(I-\mathcal{P})f_n(\tau)\|_{\sigma}^{2}\,\dd\tau\le \frac{1}{n}\int_{0}^{1} \|\mathcal{P} f_n(\tau)\|_{\sigma}^{2}\,\dd\tau ,\end{equation}
 	and $\|f_n\|_{\infty,\vartheta}<\frac{1}{n}$ for given $\vartheta>\frac{3}{2}$.

We first prove the weak compactness of $f_n$.
	For any fixed $\vartheta<0$, we multiply \eqref{Eq:linearized-VL_f=f-e^phi=} for $f=f_n$ by $\ue^{\phi_n}\!\br{v}^{2\vartheta}\!f_n$ and integrate both sides of the resulting equation.
	Similar to the argument in \cite[Proposition\;2]{Guo.Hwang.Jang.Ouyang2020(=)} and using the fact that $\ue^{\phi_n}\sim 1$ by the bootstrap assumption, we obtain
		 \begin{equation*}
\|\br{v}^{\vartheta}f_n(t)\|_{L^2}^2+\int_{t_0}^t \|f_n(\tau)\|^2_{\sigma,\vartheta}\,\dd\tau\le C e^{t-t_0} \|\br{v}^{\vartheta}f_n(t_0)\|^2_{L^2}
\end{equation*}
		 and
		 \begin{align*}
		 	\frac{\dd}{\dd t}\int_{t_0}^t\|f_n(\tau)\|^2_\sigma \,\dd\tau=\|f_n(t)\|^2_\sigma
		 	\ge C\nm{\br{v}^{-\frac{1}{2}}f_n(t_0)}^2_{L^2}-C'\int_{t_0}^t \|f(\tau)\|^2_\sigma \,\dd\tau
		 \end{align*}
		 for some $C,C'>0$.
	 By the bootstrap assumption \eqref{Bootstrap-assp:smallness-decay-L^infty} and the Gr\"onwall inequality, we obtain that
\begin{equation*}
	 \int_{t_0}^t \|f_n(\tau)\|^2_\sigma \,\dd\tau \geq C(1-\ue^{-C'(t-t_0)})
	 \|\br{v}^{-\frac{1}{2}}f_n(t_0)\|^2_{L^2}.
\end{equation*}
Now we define the normalized term $Z_n$ of $f_n$ as
$$Z_n:= \frac{f_n}{C_n}$$
with $$C_n:=\sqrt{\int_0^1 \|\mathcal{P}f_n(\tau)\|_\sigma^2 \,\dd\tau}.$$
Then $Z_n$ satisfies the equation
\begin{equation*} 
\begin{split}
&\partial_t \big(e^{C_n\phi_n} Z_n\big) + v\cdot\nabla_{\!x}\big(e^{C_n\phi_n} Z_n\big) + L\big(e^{C_n\phi_n} Z_n\big)
- 2 e^{C_n\phi_n} \!\sqrt{\mu}\,v\cdot \mathbf{E}_n \\
\;=\;& e^{C_n\phi_n}\,\Gamma[f_n,Z_n] - C_n e^{C_n\phi_n}\,\mathbf{E}_n\cdot\nabla_{\!v}Z_n + C_n e^{C_n\phi_n} Z_n\, \partial_t\phi_n ,
\end{split}
\end{equation*}
where $C_n$ are uniformly bounded by a small constant,
and $\mathbf{E}_n := -\nabla_{\!x}\phi_n$ with $\phi_n$ determined by the Poisson equation
\begin{equation*} 
-\Delta_x\phi_n = \int_{\r^3}\! \sqrt{\mu}\,Z_n\,\ud v \,=: \rho[Z_n] \;.
\end{equation*}
Thus, following a similar argument in \cite[Proposition\;2]{Guo.Hwang.Jang.Ouyang2020(=)}, we obtain the uniform bound $$\sup_{0\le \tau\le 1}\|\br{v}^{-\frac{1}{2}}Z_n(\tau)\|^2_{L^2}\le C$$ for some $C>0$. Also, by the normalization we already have $\int_0^1 \|Z_n(\tau)\|^2_\sigma \,\dd\tau=1$. Since the eigenvalues $\lambda(v)$ of the matrix $\sigma(v)$ satisfies the bound \eqref{sigma-eigenvalue-bound}, the normed vector space with the norm $\|\cdot \|_\sigma$ can be understood as a weighted $L^2$ Sobolev space and is reflexive.
	Therefore, there exists the weak limit $Z$ of $Z_n$ in $\int_0^1 \|\cdot \|_\sigma^2 \,\dd\tau$. Also, by \eqref{contradiction}, we have
\begin{equation*}
\int_0^1 \|(I-\pp)Z_n(\tau) \|_\sigma^2 \,\dd\tau\,\le\, \frac{1}{n}\,\rightarrow\, 0.
\end{equation*}
    By the triangle inequality, we also have that $\int_0^1 \|\pp Z_n(\tau)\|^2_\sigma \,\dd\tau$ is uniformly bounded from above. In addition, the norm $\|\cdot\|_\sigma$ is an anisotropic Sobolev norm with respect to the direction of the velocity $v$ by definition. Then by Alaoglu's theorem and Eberlein–Šmulian's theorem, $\pp Z_n$ converges weakly to $\pp Z$ in $\int_0^1 \|\cdot \|_\sigma^2 \,\dd\tau$ up to a subsequence.
    Thus, we conclude that $(I-\pp)Z =0$ and $Z=\pp Z$.

    Also, by taking the limit $n\rightarrow \infty$, we note that the limit $Z$ satisfies
    \begin{equation}\label{limit eq}\partial_t Z+v\cdot \nabla_x Z
    \,-\,2\sqrt{\mu}v\cdot\mathbf{E}_Z=0\end{equation} in the sense of distribution.

    Now our main strategy has two steps:
    \begin{itemize}
        \item
        Step 1: Show that the convergence $Z_n\rt Z$ is actually strong in $\int_0^1 \|\cdot \|_\sigma^2 \,\dd\tau$ by proving the compactness. This is almost the same as the argument in \cite[Proposition\;2]{Guo.Hwang.Jang.Ouyang2020(=)}, so we omit the details. The basic idea is to show the compactness in the interior and rule out the possibility of concentration near $t=0$, $t=1$, or $x\in\p\Omega$.
        \item
        Step 2: We use the equation \eqref{limit eq} and conservation laws to show that $Z$ is actually zero. This will be our focus below.
    \end{itemize}
    Eventually, we confirm that this leads to a contradiction and thus Lemma\;\ref{Lem:abc} must holds in $[0,1]$. Finally, we extend the domain to arbitrary $[s,t]$ with $\abs{t-s}\geq 1$.

\smallskip
{\it \underline{$Z$ is indeed zero}}.

The limit equation for $Z$ is (since all nonlinear terms in the equation for $f$ vanish after normalization)
\begin{align}\label{tt 1}
    \dt Z+v\cdot\nabla_xZ+2\sqrt{\mu}v\cdot\mathbf{E}=0.
\end{align}
Here assume that
\begin{align}\label{tt 2}
    Z=\sqrt{\mu}\Big(a+b\cdot v+c\abs{v}^2\Big),
\end{align}
and
\begin{align*}
    \mathbf{E}=\nabla_x\Delta_x^{-1}\int_{\r^3}\sqrt{\mu}Z\ud v:=-\nabla_x\phi,
\end{align*}
with either DBC or NBC on $\phi$.

The conservation laws for mass and energy are as follows:
\begin{align}
    &\dt\int_{\Omega\times\r^3}\sqrt{\mu}Z=0,\quad \dt\int_{\Omega\times\r^3}\abs{v}^2\sqrt{\mu}Z=0. \label{tt 9}
\end{align}
Also, combined with the normalized initial condition, we have
\begin{align}
    &\int_{\Omega\times\r^3}\sqrt{\mu}Z(t)=\int_{\Omega\times\r^3}\sqrt{\mu}Z(0)=0,\label{tt 3}\\
    &\int_{\Omega\times\r^3}\abs{v}^2\sqrt{\mu}Z(t)=\int_{\Omega\times\r^3}\abs{v}^2\sqrt{\mu}Z(0)=-\int_{\Omega}\abs{\mathbf{E}}^2(0).\label{tt 4}
\end{align}
If the domain $\Omega$ is rotational invariant
\begin{align*}
    [(x-x_0)\times\omega]\cdot n=0,
\end{align*}
then the angular momentum is conserved
\begin{align*}
    \int_{\Omega\times\r^3}[(x-x_0)\times\omega]\cdot v\sqrt{\mu}Z=0.
\end{align*}
This is essentially $\int_{\Omega\times\r^3}[(x-x_0)\times v\sqrt{\mu}Z]\cdot\omega=0$, which is consistent with the classical definition of angular momentum.
Note that all conservation laws can be directly derived from the equation \eqref{tt 1} itself, with the normalization \eqref{tt 3} on initial data.

We can analyze $Z$ as follows:
\ \\
Step 1:
    We insert \eqref{tt 2} into \eqref{tt 1} and compare the coefficients of $v$ for each order. It is easy to check that $\dt a=0$ (at lowest order) and $\nabla_xc=0$ (at highest order). Hence, we have
$a=a(x)$ and $c=c(t)$.
    Using conservation of mass \eqref{tt 9}, we know $\dt c=0$, which means $c$ is a constant.

    Based on \eqref{tt 3}, we know
    \begin{align}\label{tt 11}
        \int_{\Omega}a=0.
    \end{align}
\ \\
Step 2:
    Inserting \eqref{tt 2} into \eqref{tt 1} and checking the terms of order $v_i$, we get
    \begin{align}\label{tt 7}
        \dt b_i+\p_{x_i}(a-\phi)=0.
    \end{align}
    Inserting \eqref{tt 2} into \eqref{tt 1} and checking the terms of order $v_i^2$ and $v_iv_j$, we get
    \begin{align*}
        \p_{x_i}b_i=0,\quad \p_{x_j}b_i+\p_{x_i}b_j=0\ \ \text{for}\ \ i\neq j.
    \end{align*}
    Hence, $\nabla_xb$ is a skew-symmetric matrix. Hence, using the Helmholtz decomposition (or direct computation), $b$ must take the form
    \begin{align*}
        b=\bar b(t)+\bar\omega(t)\times x
    \end{align*}
    for some $\bar b(t)$ and $\bar\omega(t)$.

    From \eqref{tt 7}, we know that
    \begin{align*}
        \dt (\p_{x_j}b_i)=\dt(\p_{x_i}b_j).
    \end{align*}
    Hence, $\dt\nabla_xb$ is a symmetric matrix. From the last step, we know $\dt\nabla_xb$ is a skew-symmetric matrix. These two requirements imply that $\dt\nabla_xb$ is a zero matrix. Hence, $\nabla_xb$ is time-independent. Therefore, $\bar\omega=\text{const}$. Thus, we get
    \begin{align*}
        b=\bar b(t)+\bar\omega\times x.
    \end{align*}
    Considering \eqref{tt 7}, we know that $\dt b$ is independent of time. Hence, $\bar b=b_1+b_0t$.
    At the boundary, due to the specular reflection boundary, we always have $b\cdot n=0$. Hence, by taking a boundary point $x$ at which $n_x$ is parallel to $b_0$, we get $b_0=0$. In summary, we have
    \begin{align*}
        b=b_1+\bar\omega\times x.
    \end{align*}
\ \\
Step 3:
    Since $b$ is independent of time, from \eqref{tt 7} we have
    \begin{align*}
        \p_{x_i}(a-\phi)=0,
    \end{align*}
    which implies
    \begin{align} \label{tt 9-2}
        a-\phi = \text{const}.
    \end{align}
    Considering the definition
    \begin{align}\label{tt 10}
        \Delta_x\phi=a,
    \end{align}
    we know \begin{align*}
        \Delta_x\phi=\phi+\text{const}.
    \end{align*}
    \begin{itemize}
        \item[-]
        For NBC case, using the normalization condition $\ds\int_{\Omega}\phi=0$, a directly energy estimate yields
        \begin{align*}
            \int_{\Omega}\abs{\nabla_x\phi}^2+\int_{\Omega}\abs{\phi}^2=0.
        \end{align*}
        Hence, we know $\phi=0$ and thus $a$ is a constant. Based on \eqref{tt 11}, such a constant must be zero. Hence, $a=0$.
        \item[-]
        For DBC, taking the Laplacian in \eqref{tt 9-2}, we obtain
        \begin{align*}
            \Delta_xa-\Delta_x\phi=0.
        \end{align*}
        Considering \eqref{tt 10}, we have
        \begin{align*}
            \Delta_xa-a=0.
        \end{align*}
        The DBC implies $\phi=0$ on $\p\Omega$, which implies $a=\text{const}$ on $\p\Omega$.
        Due to the maximum principle of the elliptic equation, we see that $a$ does not change sign in $\Omega$.
        Then using the average property \eqref{tt 11}, we know the constant must be zero. Hence, $a=0$ and $\phi=\text{const}$.
    \end{itemize}
    All in all, in both cases, we confirm that $\phi$ is a constant and $a=0$. This further implies $\mathbf{E}=0$ and thus $c=0$ from \eqref{tt 4}.
\\
\ \\
Step 4:
    So far, we have confirmed that $a=c=0$ and $b=b_1+\bar\omega\times x$. Based on \cite[Page 748-749]{Guo2010} and \cite{Guo.Hwang.Jang.Ouyang2020(=)}, and using the conservation of angular momentum, we obtain that $b=0$. In detail, we may deduce that $Z=\bar\omega\times (x-x_0)v\sqrt{\mu}$ in $\Omega$ and $[\bar\omega\times(x-x_0)]\cdot n=0$ on $\p\Omega$ by decomposing $b_1$. If the domain is not rotational invariant, then there is no nonzero $\bar\omega$ to make the above true. On the other hand, if the domain is rotational invariant, we need an additional conservation of angular momentum to justify $\bar\omega=0$.

In summary, we conclude that $Z=0$ and this leads to a contradiction.

\smallskip
{\it \underline{Generalizing to arbitrary $0\leq s<t$ with $\abs{t-s}\in\mathbb{Z}^+$}}.

Let $t-s=N$ for some positive integer $N$. We split
\begin{align*}
    [s,t]=\left(\bigcup_{j=0}^{N-1}[s+j,s+j+1]\right).
\end{align*}
	On each interval $[s+j,s+j+1]$ for $j=0,1,...,N-1,$ we define
	$f^j(r,x,v)\eqdef f(r+s+j,x,v).$
Then clearly $f^j(r,x,v)$ is a solution on the time interval $r\in[0,1]$ with the new initial condition $f^j(0,x,v)= f(s+j,x,v)$. Note that due to the bootstrap assumption, $\mathcal{E}_\vartheta(f^j(t))$ is uniformly (in $j$) bounded from above (which is used to show the boundedness of sequence $Z_n$). Hence, using the previous argument, we obtain the positivity estimates.

Hence, by assembling all these intervals, the positivity estimate holds for the full interval $[s,t]$ with $\abs{t-s}=N$.
\end{proof}

\subsection{Weighted $L^2$ Bound and Decay}
We are now ready to prove the main theorem of this section.
\begin{proof}[Proof of Theorem\;\ref{Thm:L^2-decay}]
The proof is a modification of \cite[Theorem\;1.4]{Kim.Guo.Hwang2020}, whose idea is brought from \cite[Theorem\;5.1]{Strain.Guo2006}.
Our main observation based on the relationship
\begin{equation} \label{energy-dissipation-relation}
\|h\|_{\sigma,\vartheta} \,\gtrsim\, \|h\|_{2,\vartheta-\frac{1}{2}}
\end{equation}
is that even though $\|\cdot\|_{\sigma,\vartheta}$ is not stronger than $\|\cdot\|_{2,\vartheta}$ (but instead dominates the energy norm at the cost of losing some weight), it is still possible to bound the instant energy by a fractional power of the dissipation rate via interpolations with stronger energy norms of higher weight powers. This leads to the result of almost exponential decay.

It is also worth pointing out that our coupled electric field poses considerable difficulty in the energy-decay estimate.
To overcome this, we combine various delicate techniques and tricks in \cite[Sections\;3\,\&\,4]{Guo2012} and \cite[Section\;7]{Cao.Kim.Lee2019}
to handle the extra terms related to the electric field.

If $t\in[0,1)$, then by the local well-posedness result, the estimates naturally follow since the initial data is sufficiently small, so we focus on the case when $t\geq 1$.

\smallskip
{\it \underline{Step\;1}. Energy Estimate without Weight ($\vartheta=0$)\,}.
Throughout this proof, we consider at a rearrangement of the Vlasov-Landau equation \eqref{Eq:Vlasov-Landau_f}
with all linear terms on the left and nonlinearities on the right:
\begin{equation} \label{Eq:linearized-VL_f=f}
\partial_t f + v\cdot\nabla_{\!x}f + Lf - 2\sqrt{\mu}\,v\cdot \mathbf{E}
\;=\; \Gamma[f,f] - \mathbf{E}\cdot\nabla_{\!v}f + \big(v\cdot \mathbf{E} \big)f ,
\end{equation}
where $\mathbf{E} := \mathbf{E}_f = -\nabla_{\!x}\phi_f$ with $\phi := \phi_{f}$ determined by the Poisson equation
\begin{equation} \label{Pf-Eq:Poisson}
-\Delta_x\phi = \int_{\R^3}\! \sqrt{\mu}\,f\,\dd v \,=: \rho[f] \;.
\end{equation}

To handle the most problematic nonlinear term $(v\cdot\, \mathbf{E} )f$, we combine it with the linear streaming term $v\cdot\nabla_{\!x}f$ (which also contains an  extra $v$ factor) by multiplying both sides of \eqref{Eq:linearized-VL_f=f}\, by $e^\phi$, so we can rewrite the equation as
\begin{equation} \label{Eq:linearized-VL_f=f-e^phi}
\partial_t \big(e^\phi f\big) + v\cdot\nabla_{\!x}\big(e^\phi f\big) + L\big(e^\phi f\big)
- 2 e^\phi \!\sqrt{\mu}\,v\cdot \mathbf{E}
\;=\; e^\phi\,\Gamma[f,f] - e^\phi\,\mathbf{E}\cdot\nabla_{\!v}f + e^\phi f\, \partial_t\phi .
\end{equation}
Multiplying $e^\phi f$ on both sides of \,\eqref{Eq:linearized-VL_f=f-e^phi}\, and integrating in $(x,v)\in\Omega\times\R^3$,
we get
\begin{equation} \label{Eq:linearized-VL-e^phi-int}
\begin{split}
&\frac{1}{2}\,\frac{\dd}{\dd t}\, \big\|e^{\phi} f\big\|_{2}^2
\,+\, \frac{1}{2}\iint_{\Omega\times\R^3}\! \nabla_{\!x}\cdot \left\{v \big(e^{\phi} f\big)^2\right\} \\
&\,+\, \iint_{\Omega\times\R^3}\! \big(e^{\phi} f\big)\, L\big(e^{\phi} f\big)
\,-\, \iint_{\Omega\times\R^3}\! 2\,\big(e^{2\phi} f\big) \sqrt{\mu}\,v\cdot \mathbf{E} \\
\,=\;& \iint_{\Omega\times\R^3}\! \big(e^{2\phi} f\big)\, \Gamma[f,f]
\,-\, \frac{1}{2}\iint_{\Omega\times\R^3}\! e^{2\phi}\,\mathbf{E}\cdot\nabla_{\!v}\big(f^2\big)
\,+\, \iint_{\Omega\times\R^3}\! \big(e^{\phi} f\big)^2\, \partial_t\phi \;.
\end{split}
\end{equation}
We need to estimate every term of \,\eqref{Eq:linearized-VL-e^phi-int}\, in terms of the energy norm $\|\cdot\|_{2}$ or dissipation rate $\|\cdot\|_{\sigma}$,
for which we will use repeatedly $e^{\phi} \sim 1$ with $\|\phi\|_{\infty} \lesssim \|f\|_{\infty} \ll 1$ (see the proof of Lemma\;\ref{Lem:E_f-L^p}).

First we observe that the following two terms vanish by using integration by parts (divergence theorem)\,:
\begin{equation} \label{energy-est-1}
\frac{1}{2}\iint_{\Omega\times\R^3}\! e^{2\phi}\,\mathbf{E}\cdot\nabla_{\!v}\big(f^2\big) \,\dd v\dd x
\,=\, 0 \;,
\end{equation}
\vspace{3pt}
\begin{equation} \label{energy-est-2}
\begin{split}
\frac{1}{2}\iint_{\Omega\times\R^3}\! \nabla_{\!x}\cdot \left\{v \big(e^{\phi} f\big)^2\right\} \dd v\dd x
&\,=\, \frac{1}{2}\iint_{\partial\Omega\times\R^3}\! \big(e^{\phi} f\big)^2 (v\cdot n_x) \,\dd v\dd S_x \\
&\,=\, \frac{1}{2}\int_{\partial\Omega}\! e^{2\phi} \bigg[ \bigg(\int_{\{v\cdot n_x >0\}} \!- \int_{\{v\cdot n_x <0\}} \bigg) f^2 |v\cdot n_x| \,\dd v \bigg] \dd S_x \\
&\,=\, 0 \;.
\end{split}
\end{equation}
Here the boundary terms cancel each other in view of the specular-reflection boundary condition (\ref{Specular-BC_f}).

For the other nonlinear terms, we have
\begin{equation} \label{energy-est-3}
\iint_{\Omega\times\R^3}\! \big(e^{2\phi} f\big)\, \Gamma[f,f]
\,\lesssim\, \|f\|_{\infty} \big\|e^{\phi} f\big\|_{\sigma}^2
\,\lesssim\, \|f\|_{\infty} \|f\|_{\sigma}^2
\end{equation}
by (\ref{Gamma-est-1}) in Lemma\;\ref{Lem:Gamma-est}, and
\begin{equation} \label{energy-est-4}
\iint_{\Omega\times\R^3}\! \big(e^{\phi} f\big)^2\, \partial_t\phi
\,\lesssim\, \|\partial_t\phi\|_{\infty} \big\|e^{\phi} f\big\|_{2}^2
\,\lesssim\, \|\partial_t f\|_{2} \|f\|_{2}^2 \;,
\end{equation}
where for the last inequality, we modify the proof of Lemma\;\ref{Lem:E_f-L^p}.
Note that we cannot bound this term in terms of $\|f\|_{\sigma}^2$, which ultimately leads to an extra term on the RHS in (\ref{Eq:linearized-VL-e^phi-int-6}) not being absorbed, because the $\sigma$-norm is not strong enough to control the $L^2$\;norm.

The last term on the LHS of \,\eqref{Eq:linearized-VL-e^phi-int}\, is more troublesome to deal with.
We first manipulate this term as follows:
\begin{equation} \label{Eq:linearized-VL-e^phi-int-2''}
\begin{split}
- \iint_{\Omega\times\R^3}\! 2\,\big(e^{2\phi} f\big) \sqrt{\mu}\,v\cdot \mathbf{E}
&\,=\, \iint_{\Omega\times\R^3}\! 2\,\big(e^{2\phi} f\big) \sqrt{\mu}\,v\cdot \nabla_{\!x}\phi
\,=\, \iint_{\Omega\times\R^3}\! f\sqrt{\mu}\, v\cdot \nabla_{\!x} \big(e^{2\phi} \big) \\
&\,=\, \iint_{\Omega\times\R^3}\! \nabla_{\!x}\cdot \left\{v\sqrt{\mu}\, \big(e^{2\phi} f\big)\right\}
- e^{2\phi}\! \sqrt{\mu}\,v\cdot \nabla_{\!x} f \\
&\,=\, \iint_{\partial\Omega\times\R^3}\! \big(e^{2\phi} f\big)\sqrt{\mu}\, (v\cdot n_x) \,\dd v\dd S_x
\,- \iint_{\Omega\times\R^3}\! e^{2\phi}\! \sqrt{\mu}\,v\cdot \nabla_{\!x} f \\
&\,=\, - \iint_{\Omega\times\R^3}\! e^{2\phi}\! \sqrt{\mu}\,v\cdot \nabla_{\!x} f \;,
\end{split}
\end{equation}
where again, we use integration by parts (divergence theorem), and the boundary term on the third line vanishes by the specular boundary condition.
Then we want to control \mbox{$- \iint_{\Omega\times\R^3}\! e^{2\phi}\! \sqrt{\mu}\,v\cdot \nabla_{\!x}f$} from the original equation itself\,:
this time multiplying both sides of \eqref{Eq:linearized-VL_f=f}\, by the factor $e^{2\phi}\!\sqrt{\mu}$ and integrating over $\Omega\times\R^3$, we obtain
\begin{equation} \label{Eq:linearized-VL-e^phi-int-2}
\begin{split}
& \iint_{\Omega\times\R^3}\! e^{2\phi}\!\sqrt{\mu}\, \partial_t f
\,+\, \iint_{\Omega\times\R^3}\! e^{2\phi}\! \sqrt{\mu}\,v\cdot \nabla_{\!x} f
\,+\, \iint_{\Omega\times\R^3}\! e^{2\phi}\! \sqrt{\mu}\, Lf
\,-\, \iint_{\Omega\times\R^3}\! 2\,e^{2\phi} \mu\, v\cdot \mathbf{E} \\
\,=\;& \iint_{\Omega\times\R^3}\! e^{2\phi}\! \sqrt{\mu}\, \Gamma[f,f]
\,-\, \iint_{\Omega\times\R^3}\! e^{2\phi}\! \sqrt{\mu}\:\mathbf{E}\cdot \nabla_{\!v} f
\,+\, \iint_{\Omega\times\R^3}\! e^{2\phi}\! \sqrt{\mu}\, \big(v\cdot \mathbf{E} \big)f \;.
\end{split}
\end{equation}
Obviously, the last term on the LHS vanishes due to symmetry/oddness,
and by the orthogonality to the Landau operators, we see that
\begin{equation*}
\iint_{\Omega\times\R^3}\! e^{2\phi}\! \sqrt{\mu}\, Lf
\,= \iint_{\Omega\times\R^3}\! e^{2\phi}\! \sqrt{\mu}\, \Gamma[f,f]
\,=\, 0.
\end{equation*}
Also, integrating by parts gives
\begin{equation*}
- \iint_{\Omega\times\R^3}\! e^{2\phi}\! \sqrt{\mu}\:\mathbf{E}\cdot \nabla_{\!v} f
\,+ \iint_{\Omega\times\R^3}\! e^{2\phi}\! \sqrt{\mu}\, \big(v\cdot \mathbf{E} \big)f \,=\, 0 .
\end{equation*}
To sum up, we may equate the following two terms from \eqref{Eq:linearized-VL-e^phi-int-2} and further deduce that
\begin{equation} \label{Eq:linearized-VL-e^phi-int-2'}
\begin{split}
- \iint_{\Omega\times\R^3}\! e^{2\phi}\! \sqrt{\mu}\,v\cdot \nabla_{\!x} f
&\,=\, \iint_{\Omega\times\R^3}\! e^{2\phi}\!\sqrt{\mu}\, \partial_t f \\
&\,=\, \int_{\Omega} e^{2\phi}\, \partial_t\!\left(\int_{\R^3}\! \sqrt{\mu}\,f\,\dd v \right) \dd x
\;=\; - \int_{\Omega} e^{2\phi} \partial_t \Delta_x\phi \,\dd x \\
&\,=\, \int_{\Omega} 2\,e^{2\phi} (\nabla_{\!x}\phi)\cdot (\nabla_{\!x}\partial_t\phi) \,\dd x
\,- \int_{\partial\Omega}\! e^{2\phi}\, \partial_t\! \left(\frac{\partial_x \phi}{\partial n_x} \right) \dd S_x \\
&\,=\, \int_{\Omega} e^{2\phi} \partial_t |\nabla_{\!x}\phi|^2 \,\dd x
\;=\; \int_{\Omega} \partial_t\! \left( e^{2\phi} |\nabla_{\!x}\phi|^2 \right)
-\! \int_{\Omega} 2e^{2\phi} |\nabla_{\!x}\phi|^2 \partial_t\phi \\
&\,=\, \frac{\dd}{\dd t} \int_{\Omega} e^{2\phi} |\mathbf{E}|^2
-\! \int_{\Omega} 2e^{2\phi} |\mathbf{E}|^2 \partial_t\phi \;,
\end{split}
\end{equation}
by using the Poisson equation \,\eqref{Pf-Eq:Poisson}\, on the second line, followed by integration by parts (divergence theorem).
Note also that the boundary term on the third line vanishes under either the zero\,-\,Neumann BC ($\frac{\partial\phi}{\partial n}\big|_{\partial\Omega} \equiv 0$) or the zero\,-\,Dirichlet BC ($\phi|_{\partial\Omega} \equiv 0$), in which case this term reduces to
\begin{equation*}
\frac{\dd}{\dd t} \int_{\partial\Omega} \frac{\partial\phi}{\partial n}(t) \,\dd S
\,=\, -\,\frac{\dd}{\dd t} \int_{\partial\Omega} \mathbf{E}(t)\cdot n \,\dd S
\,\equiv\, 0,
\end{equation*}
thanks to the conservation of flux for the self-consistent electric field (see (\ref{flux-conservation})).
Hence, combining (\ref{Eq:linearized-VL-e^phi-int-2''}) and (\ref{Eq:linearized-VL-e^phi-int-2'}) reveals that
\begin{equation} \label{energy-est-5}
- \iint_{\Omega\times\R^3}\! 2\,\big(e^{2\phi} f\big) \sqrt{\mu}\,v\cdot \mathbf{E}
\;=\; \frac{\dd}{\dd t} \int_{\Omega} e^{2\phi} |\mathbf{E}|^2
\,- \int_{\Omega} 2e^{2\phi} |\mathbf{E}|^2 \partial_t\phi \;.
\end{equation}
We will combine the first term above with the first term of \eqref{Eq:linearized-VL-e^phi-int}.
After moving the second to the other side, we bound it as
\begin{equation} \label{energy-est-6}
\int_{\Omega} 2e^{2\phi} |\mathbf{E}|^2 \partial_t\phi
\,\lesssim\, \|\partial_t\phi\|_{\infty} \big\|e^{\phi} \mathbf{E}\big\|_{2}^2
\,\lesssim\, \|\partial_t f\|_{2} \|f\|_{\sigma}^2 \;.
\end{equation}
Here we justify the last inequality by modifying the proof of Lemma\;\ref{Lem:E_f-L^p}:
\begin{equation} \label{E-L^2->f-sigma}
\begin{split}
\big\|e^{\phi} \mathbf{E}\big\|_{2}
\,\lesssim\, \|\mathbf{E}\|_{2}
&\,\lesssim\, \big\|\rho[f]\big\|_{2}
\,=\, \left\|\int_{\R^3}\! \sqrt{\mu}\,f\,\dd v \right\|_{L^2_{x}} \\
&\,\leq\, \int_{\R^3}\! \big\|\sqrt{\mu}\,f\big\|_{L^2_{x}} \dd v
\,= \int_{\R^3}\! \langle v\rangle^{\frac{1}{2}}\! \sqrt{\mu}\, \big\|\langle v\rangle^{-\frac{1}{2}} f\big\|_{L^2_{x}} \,\dd v \\[3pt]
&\,\leq\, \big\|\langle v\rangle^{-\frac{1}{2}} f\big\|_{L^2_{x,v}} \!\cdot
\big|\langle v\rangle^{\frac{1}{2}}\!\sqrt{\mu}\,\big|_{L^{2}_v} \\[2pt]
&\,\lesssim\; \|f\|_{2,-\frac{1}{2}}
\;\lesssim\; \|f\|_{\sigma} \;.
\end{split}
\end{equation}

Now in \eqref{Eq:linearized-VL-e^phi-int} we are left with the third term on the LHS, which at first can be bounded below as
\begin{equation} \label{energy-est-7}
\iint_{\Omega\times\R^3}\! \big(e^{\phi} f\big)\, L\big(e^{\phi} f\big)
\,=\, \left(L\big(e^{\phi} f\big),\, e^{\phi} f \right)
\;\geq\; \delta\, \big\|e^{\phi} (I-\mathcal{P})f\big\|_{\sigma}^2
\;\gtrsim\; \big\|e^{\phi}(I-\mathcal{P})f\big\|_{\sigma}^2 \;,
\end{equation}
due to the semi-positivity of $L$ (see Lemma\;\ref{Lem:L-semi-positivity}).

Summarizing all above, inserting (\ref{energy-est-1}), (\ref{energy-est-2}),
(\ref{energy-est-3}), (\ref{energy-est-4}), (\ref{energy-est-5}) ,(\ref{energy-est-6}), (\ref{energy-est-7}) into \eqref{Eq:linearized-VL-e^phi-int}, we arrive at
\begin{equation*} 
\frac{\dd}{\dd t} \left(\, \frac{1}{2}\big\|e^{\phi} f\big\|_{2}^2 + \big\|e^{\phi} \mathbf{E}\big\|_{2}^2 \,\right)
\,+\, \big\|e^{\phi}(I-\mathcal{P})f\big\|_{\sigma}^2
\;\lesssim\; \big(\,\|f\|_{\infty} + \|\partial_t f\|_{2} \big) \|f\|_{\sigma}^2 \,+\, \|\partial_t f\|_{2} \|f\|_{2}^2 \;,
\end{equation*}
and it naturally follows that
\begin{equation} \label{Eq:linearized-VL-e^phi-int-3.2}
\begin{split}
&\frac{\dd}{\dd t} \left(\, \frac{1}{2}\big\|e^{\phi} f\big\|_{2}^2 + \big\|e^{\phi} \mathbf{E}\big\|_{2}^2 \,\right)
\,+\, \big\|e^{\phi}(I-\mathcal{P})f\big\|_{\sigma}^2 \\
\;\lesssim\;& \big(\,\|f\|_{\infty} + \|\partial_t f\|_{2} \big) \|f\|_{\sigma}^2 \,+\, \|\partial_t f\|_{2} \left(\, \frac{1}{2}\big\|e^{\phi} f\big\|_{2}^2 + \big\|e^{\phi} \mathbf{E}\big\|_{2}^2 \,\right) \;.
\end{split}
\end{equation}
Define
\begin{align*}
    \psi(t):=-\int_0^t \big\|\dt f(\tau)\big\|_2\,\dd\tau.
\end{align*}
From the bootstrap assumption (\ref{Bootstrap-assp:smallness-decay-L^2}), we know $\abs{\psi(t)}\ll 1$, and thus $e^{\psi} \sim 1$.
Then multiplying both sides of (\ref{Eq:linearized-VL-e^phi-int-3.2}) above by the integrating factor $e^{\psi}$, we have
\begin{equation*}
\frac{\dd}{\dd t} \left[e^{\psi} \left(\, \frac{1}{2}\big\|e^{\phi} f\big\|_{2}^2 + \big\|e^{\phi} \mathbf{E}\big\|_{2}^2 \,\right)\right]
\,+\, e^{\psi}\big\|e^{\phi}(I-\mathcal{P})f\big\|_{\sigma}^2
\;\lesssim\; \big(\,\|f\|_{\infty} + \|\partial_t f\|_{2} \big) \|f\|_{\sigma}^2 \, \;.
\end{equation*}
In order to get the coercivity bound (\ref{coercivity-est-L}),
we need to first integrate over time (for any $0\leq s<t$) and obtain
\begin{equation*}
\begin{split}
&\left(\, \frac{1}{2}\big\|e^{\psi(t)} e^{\phi(t)} f(t)\big\|_{2}^2 + \big\|e^{\psi(t)} e^{\phi(t)} \mathbf{E}(t)\big\|_{2}^2 \,\right)
- \left(\, \frac{1}{2}\big\|e^{\psi(s)} e^{\phi(s)} f(s)\big\|_{2}^2 + \big\|e^{\psi(s)} e^{\phi(s)} \mathbf{E}(s)\big\|_{2}^2 \,\right)\\
&\,+\, \int_s^t\!e^{\phi(\tau)}e^{\phi(\tau)}\big\|(I-\mathcal{P})f(\tau)\big\|_{\sigma}^2 \,\dd\tau \\
\,\lesssim\,&\; \Big(\sup_{\tau\geq 0}\|f\|_{\infty} + \sup_{\tau\geq 0}\|\partial_t f\|_{2} \Big)\cdot\! \int_s^t\! \big\|f(\tau)\big\|_{\sigma}^2 \,\dd\tau
\,\;.
\end{split}
\end{equation*}

Since $1\ls e^{\phi(t)} \ls 1$
and $1\ls e^{\psi(t)} \ls 1$ uniformly for any $t$, applying Lemma\;\ref{Lem:abc} (following the proof of Corollary\;\ref{Cor:coercivity-est-L}) yields
\begin{align*} 
&\left(\,\big\|e^{\psi(t)} e^{\phi(t)}f(t)\big\|_{2}^2 + \big\|e^{\psi(t)} e^{\phi(t)}\mathbf{E}(t)\big\|_{2}^2 \,\right)
\,+\, \delta' \left\{\int_s^t\!\big\|f(\tau)\big\|_{\sigma}^2 \,\dd\tau
\,+ \int_s^t\!\big\|\mathbf{E}(\tau)\big\|_{2}^2 \,\dd\tau \right\} \no\\
\,\lesssim\,&\;  \left(\,\big\|e^{\psi(s)} e^{\phi(s)}f(s)\big\|_{2}^2 + \big\|e^{\psi(s)} e^{\phi(s)}\mathbf{E}(s)\big\|_{2}^2 \,\right)+\Big(\sup_{\tau\geq 0}\|f\|_{\infty} + \sup_{\tau\geq 0}\|\partial_t f\|_{2} \Big)\cdot\! \int_s^t\! \big\|f(\tau)\big\|_{\sigma}^2 \,\dd\tau
\end{align*}
for all $0\leq s<t$ with $\abs{t-s} \in\mathbb{Z}^+$.

Finally, under the a priori assumption
$$\sup_{t}\|f\|_{\infty} +\, \sup_{t}\|\partial_t f\|_{2} \,\ll\, 1 ,$$
this corresponding term can be absorbed into the LHS,
and therefore we end up with
\begin{align} \label{Eq:linearized-VL-e^phi-int-6}
&\left(\,\big\|e^{\psi(t)} e^{\phi(t)}f(t)\big\|_{2}^2 + \big\|e^{\psi(t)} e^{\phi(t)}\mathbf{E}(t)\big\|_{2}^2 \,\right)
\,+\, \delta' \left\{\int_s^t\!\big\|f(\tau)\big\|_{\sigma}^2 \,\dd\tau
\,+ \int_s^t\!\big\|\mathbf{E}(\tau)\big\|_{2}^2 \,\dd\tau \right\} \\
\;\leq\;& \,\big\|e^{\psi(s)} e^{\phi(s)}f(s)\big\|_{2}^2 + \big\|e^{\psi(s)} e^{\phi(s)}\mathbf{E}(s)\big\|_{2}^2 \, \;.\no
\end{align}

\smallskip
{\it \underline{Step\;2}. Weighted Energy Estimate ($\vartheta >0$)\,}.
Multiplying both sides of \,\eqref{Eq:linearized-VL_f=f-e^phi}\, by $\langle v\rangle^{2\vartheta} e^\phi f$ for $\vartheta >0$ and integrating over $\Omega\times\R^3$,
we get
\begin{equation} \label{Eq:linearized-VL-e^phi-int-weighted}
\begin{split}
&\frac{1}{2}\,\frac{\dd}{\dd t}\, \big\|\langle v\rangle^{\vartheta} e^{\phi} f\big\|_{2}^2
\,+\, \frac{1}{2}\iint_{\Omega\times\R^3}\! \nabla_{\!x}\cdot \left\{v \langle v\rangle^{2\vartheta} \big(e^{\phi} f\big)^2\right\} \\
&\,+\, \iint_{\Omega\times\R^3}\! \langle v\rangle^{2\vartheta} \big(e^{\phi} f\big)\, L\big(e^{\phi} f\big)
\,-\, \iint_{\Omega\times\R^3}\! 2\,\langle v\rangle^{2\vartheta} \big(e^{2\phi} f\big) \sqrt{\mu}\,v\cdot \mathbf{E} \\
\,=\;& \iint_{\Omega\times\R^3}\! \langle v\rangle^{2\vartheta} \big(e^{2\phi} f\big)\, \Gamma[f,f]
\,-\, \frac{1}{2}\iint_{\Omega\times\R^3}\! \langle v\rangle^{2\vartheta} e^{2\phi}\,\mathbf{E}\cdot\nabla_{\!v}\big(f^2\big)
\,+\, \iint_{\Omega\times\R^3}\! \langle v\rangle^{2\vartheta} \big(e^{\phi} f\big)^2\, \partial_t\phi \;.
\end{split}
\end{equation}
This time we aim to estimate every term of \,\eqref{Eq:linearized-VL-e^phi-int-weighted}\, in terms of the weighted energy norm $\|\cdot\|_{2,\vartheta}$ or (weighted) dissipation rate $\|\cdot\|_{\sigma,\vartheta}$.

For the LHS,
using the divergence theorem, the second term vanishes since
\begin{equation} \label{weighted-energy-est-LHS-2}
\frac{1}{2}\iint_{\Omega\times\R^3}\! \nabla_{\!x}\cdot \left\{v \langle v\rangle^{2\vartheta} \big(e^{\phi} f\big)^2\right\}
\,=\, \frac{1}{2} \bigg(\iint_{\gamma_+} \!- \iint_{\gamma_-}\bigg)
\big(e^{\phi} f\big)^2 \langle v\rangle^{2\vartheta} |v\cdot n_x| \,\dd v\dd S_x
\,=\, 0 \;,
\end{equation}
again by symmetry/oddness with the specular boundary condition on $f$.
As for the third term, we instead use (\ref{L-est}) in Lemma\;\ref{Lem:L-est} to bound it below as
\begin{equation} \label{weighted-energy-est-LHS-3}
\iint_{\Omega\times\R^3}\! \langle v\rangle^{2\vartheta} \big(e^{\phi} f\big)\, L\big(e^{\phi} f\big)
\,=\, \left(\langle v\rangle^{2\vartheta} L\big(e^{\phi} f\big) ,\, e^{\phi} f \,\right)
\,\gtrsim\, \big\|e^{\phi} f\big\|_{\sigma,\vartheta}^2 - C_\vartheta \big\|e^{\phi} f\big\|_{\sigma}^2 \;.
\end{equation}
Also, for the fourth term, the trick in $\vartheta=0$ case does not apply, so we move this term to the RHS and bound it directly using H\"{o}lder's inequality as well as the Cauchy-Schwarz inequality with a small parameter $\varepsilon>0$:
\begin{equation} \label{weighted-energy-est-LHS-4}
\begin{split}
\iint_{\Omega\times\R^3}\! 2\,\langle v\rangle^{2\vartheta} \big(e^{2\phi} f\big) \sqrt{\mu}\,v\cdot \mathbf{E}
&\;\lesssim\; \big|\langle v\rangle^{\vartheta+\frac{3}{2}} \sqrt{\mu}\,\big|_{\infty} \big\|e^{\phi}f\big\|_{2,\vartheta-\frac{1}{2}} \big\|e^{\phi}\mathbf{E}\big\|_{2} \\[-5pt]
&\;\leq\; C_{\vartheta} \big\|e^{\phi}f\big\|_{\sigma,\vartheta} \|f\|_{\sigma} \\
&\;\leq\; \varepsilon\, \|f\|_{\sigma,\vartheta}^2 \,+\, C_{\vartheta,\varepsilon}\, \|f\|_{\sigma}^2 \;.
\end{split}
\end{equation}
Here the second inequality holds because of the relation $\|\,\cdot\,\|_{2,\vartheta-1/2} \,\lesssim\, \|\,\cdot\,\|_{\sigma,\vartheta}$, and also due to the same reason as in (\ref{E-L^2->f-sigma}).

For the RHS,
we estimate (similarly to the case without weight)
\begin{equation} \label{weighted-energy-est-RHS-1}
\iint_{\Omega\times\R^3}\! \langle v\rangle^{2\vartheta} \big(e^{2\phi} f\big)\, \Gamma[f,f]
\;\lesssim\; \|f\|_{\infty} \big\|e^{\phi} f\big\|_{\sigma,\vartheta}^2
\;\lesssim\; \|f\|_{\infty} \|f\|_{\sigma,\vartheta}^2
\end{equation}
by (\ref{Gamma-est-1}), and
\begin{equation} \label{weighted-energy-est-RHS-2}
\begin{split}
-\, \frac{1}{2}\iint_{\Omega\times\R^3}\! \langle v\rangle^{2\vartheta} e^{2\phi}\,\mathbf{E}\cdot\nabla_{\!v}\big(f^2\big)
&\,=\; \frac{1}{2}\iint_{\Omega\times\R^3}\! \mathbf{E}\cdot\nabla_{\!v}\Big(\langle v\rangle^{2\vartheta}\Big)\, e^{2\phi}f^2 \\
&\;\lesssim\; \|\mathbf{E}\|_{\infty} \big\|e^{\phi}f\big\|_{2,\vartheta-\frac{1}{2}}^2
\;\lesssim\; \|f\|_{\infty} \|f\|_{\sigma,\vartheta}^2
\end{split}
\end{equation}
via integration by parts in $v$, and
\begin{equation} \label{weighted-energy-est-RHS-3}
\iint_{\Omega\times\R^3}\! \langle v\rangle^{2\vartheta} \big(e^{\phi} f\big)^2\, \partial_t\phi
\;\lesssim\; \|\partial_t\phi\|_{\infty} \big\|e^{\phi} f\big\|_{2,\vartheta}^2
\;\lesssim\; \|\partial_t f\|_{2} \|f\|_{2,\vartheta}^2 \;,
\end{equation}
where in the last inequality we modify
the proof of Lemma\;\ref{Lem:E_f-L^p}.

Collecting all the estimates in (\ref{weighted-energy-est-LHS-2})\,--\,(\ref{weighted-energy-est-RHS-3}) above and inserting them into \eqref{Eq:linearized-VL-e^phi-int-weighted}, we obtain
\begin{equation*} 
\frac{\dd}{\dd t}\, \big\|e^{\phi} f\big\|_{2,\vartheta}^2
\,+\, \big\|e^{\phi} f\big\|_{\sigma,\vartheta}^2
\;\lesssim\; \big(\|f\|_{\infty} + \varepsilon \big) \|f\|_{\sigma,\vartheta}^2
\,+\, C_\vartheta \big\|e^{\phi} f\big\|_{\sigma}^2
\,+\, C_{\vartheta,\varepsilon}\, \|f\|_{\sigma}^2
\,+\, \|\partial_t f\|_{2} \|f\|_{2,\vartheta}^2 \;,
\end{equation*}
which further yields
\begin{equation*} 
\frac{\dd}{\dd t}\, \big\|e^{\phi} f\big\|_{2,\vartheta}^2
\,+\, \|f\|_{\sigma,\vartheta}^2
\;\lesssim\; C_\vartheta \|f\|_{\sigma}^2
\,+\, \|\partial_t f\|_{2} \|f\|_{2,\vartheta}^2
\end{equation*}
by choosing $\varepsilon>0$ small enough
and assuming $\sup_{t}\|f\|_{\infty} \ll 1$ so the corresponding term can be absorbed into the LHS.
Hence, multiplying $e^{\psi}$ on both sides (to kill the last term on RHS), then integrating over time (for $0\leq s<t$), we finally get
\begin{equation} \label{Eq:linearized-VL-e^phi-int-weighted-4}
\begin{split}
&\big\|e^{\psi(t)} e^{\phi(t)}f(t)\big\|_{2,\vartheta}^2
\;+\, \int_s^t\!\big\|f(\tau)\big\|_{\sigma,\vartheta}^2 \,\dd\tau
\;\lesssim\,\; \big\|e^{\psi(s)} e^{\phi(s)}f(s)\big\|_{2,\vartheta}^2+ C_\vartheta\! \int_s^t\!\big\|f(\tau)\big\|_{\sigma}^2 \,\dd\tau
\;.
\end{split}
\end{equation}

\smallskip
{\it \underline{Step\;3}. Uniform Weighted Energy Bounds}.
The idea of the following two steps is inspired by the so-called ``Two-tier'' energy method in \cite{Guo.Tice2013=}.
The uniform energy bound (\ref{energy-bound-2}) and time-decay result (\ref{L^2-decay-2}) follow from weighted energy estimates with arbitrarily strong velocity-weight and sufficiently fast decay-rate.
To be specific,
we need to obtain a weighted energy inequality
($\vartheta',\vartheta \in \mathbb{N}$):
\begin{equation} \label{weighted-energy-inequality'}
\begin{split}
&\bigg(\,\sum_{\vartheta'=0}^{2\vartheta} \big\|e^{\psi(t)} e^{\phi(t)}f(t)\big\|_{2,\vartheta'\!/2}^2 + \big\|e^{\psi(t)} e^{\phi(t)}\mathbf{E}(t)\big\|_{2}^2 \,\bigg)\\
&\,+\, \delta'\, \Bigg\{ \sum_{\vartheta'=0}^{2\vartheta} \int_s^t\!\big\|f(\tau)\big\|_{\sigma,\vartheta'\!/2}^2 \,\dd\tau
\,+ \int_s^t\!\big\|\mathbf{E}(\tau)\big\|_{2}^2 \,\dd\tau \Bigg\}\\
\;\ls&\, \,\sum_{\vartheta'=0}^{2\vartheta} \big\|e^{\psi(s)} e^{\phi(s)}f(s)\big\|_{2,\vartheta'\!/2}^2 + \big\|e^{\psi(s)} e^{\phi(s)}\mathbf{E}(s)\big\|_{2}^2
\end{split}
\end{equation}
by summing up (\ref{Eq:linearized-VL-e^phi-int-6}) and (\ref{Eq:linearized-VL-e^phi-int-weighted-4}) from previous two steps over $\vartheta'\! \in [0,2\vartheta]\cap \mathbb{N}$ (for any given $\vartheta \in \mathbb{N}$) with proper weight (summation factor), such that the term $\sum_{\vartheta'} C_{\vartheta'}\! \int_s^t\!\big\|f(\tau)\big\|_{\sigma}^2 \dd\tau$ can be absorbed to LHS.
Alternatively, we may derive this energy inequality in a more rigorous way via induction on the weight-power $\vartheta \in \mathbb{N}$ (in a similar fashion to the proof of \cite[Theorem\;1.4]{Kim.Guo.Hwang2020}\,).
Note that we cannot absorb RHS here into LHS because $\sigma$-norm cannot control the $L^2$\;norm for fixed $\vartheta$ (at least for the largest $\vartheta$).

The estimate \eqref{weighted-energy-inequality'} only holds when $t-s\in\mathbb{Z}^+$ due to Lemma \ref{Lem:abc}. For general interval $(s,t)$, we write
\begin{align*}
    t-s=N+r,
\end{align*}
where $N\in\mathbb{N}$ and $r\in[0,1)$. Then for the interval $[s+r,t]$, using the similar argument as \eqref{weighted-energy-inequality'}, and consider $\ue^{\phi}\sim 1$ and $\ue^{\psi}\sim 1$, we have

\begin{equation} \label{weighted-energy-inequality''}
\begin{split}
&\bigg(\,\sum_{\vartheta'=0}^{2\vartheta} \big\|f(t)\big\|_{2,\vartheta'\!/2}^2 + \big\|\mathbf{E}(t)\big\|_{2}^2 \,\bigg)\\
&\,+\, \delta'\, \Bigg\{ \sum_{\vartheta'=0}^{2\vartheta} \int_{s+r}^t\!\big\|f(\tau)\big\|_{\sigma,\vartheta'\!/2}^2 \,\dd\tau
\,+ \int_{s+r}^t\!\big\|\mathbf{E}(\tau)\big\|_{2}^2 \,\dd\tau \Bigg\}\\
\;\ls&\, \,\sum_{\vartheta'=0}^{2\vartheta} \big\|f(s+r)\big\|_{2,\vartheta'\!/2}^2 + \big\|\mathbf{E}(s+r)\big\|_{2}^2 .
\end{split}
\end{equation}
Based on the local well-posedness result (see Theorem \ref{local-wellposedness}), we obtain the bounds in the interval $[s,s+r]$:
\begin{equation} \label{weighted-energy-inequality'''}
\begin{split}
&\bigg(\,\sum_{\vartheta'=0}^{2\vartheta} \big\|f(s+r)\big\|_{2,\vartheta'\!/2}^2 + \big\|\mathbf{E}(s+r)\big\|_{2}^2 \,\bigg)\\
&\,+\, \, \Bigg\{ \sum_{\vartheta'=0}^{2\vartheta} \int_s^{s+r}\!\big\|f(\tau)\big\|_{\sigma,\vartheta'\!/2}^2 \,\dd\tau
\,+ \int_s^{s+r}\!\big\|\mathbf{E}(\tau)\big\|_{2}^2 \,\dd\tau \Bigg\}\\
\;\ls&\, \,\sum_{\vartheta'=0}^{2\vartheta} \big\|f(s)\big\|_{2,\vartheta'\!/2}^2 + \big\|\mathbf{E}(s)\big\|_{2}^2 .
\end{split}
\end{equation}
Note that the initial condition (\ref{smallness-assumption-local}) and the bootstrap assumptions ensure a universal time extension for all $s\geq 0$. Therefore, it is valid to apply Theorem \ref{local-wellposedness} as long as we have the smallness of $\|f(s)\|_{\infty,\vartheta}$ for $s>0$.

Summing up \eqref{weighted-energy-inequality''} and \eqref{weighted-energy-inequality'''} and absorbing $\sum_{\vartheta'=0}^{2\vartheta} \big\|f(s+r)\big\|_{2,\vartheta'\!/2}^2 + \big\|\mathbf{E}(s+r)\big\|_{2}^2$ to the LHS, we obtain the bounds in the interval $[s,t]$
\begin{equation} \label{weighted-energy-inequality}
\begin{split}
&\bigg(\,\sum_{\vartheta'=0}^{2\vartheta} \big\|f(t)\big\|_{2,\vartheta'\!/2}^2 + \big\|\mathbf{E}(t)\big\|_{2}^2 \,\bigg)\\
&\,+\, \delta'\, \Bigg\{ \sum_{\vartheta'=0}^{2\vartheta} \int_{s}^t\!\big\|f(\tau)\big\|_{\sigma,\vartheta'\!/2}^2 \,\dd\tau
\,+ \int_{s}^t\!\big\|\mathbf{E}(\tau)\big\|_{2}^2 \,\dd\tau \Bigg\}\\
\;\ls&\, \,\sum_{\vartheta'=0}^{2\vartheta} \big\|f(s)\big\|_{2,\vartheta'\!/2}^2 + \big\|\mathbf{E}(s)\big\|_{2}^2 .
\end{split}
\end{equation}
Define an instantaneous weighted energy functional
\begin{equation*}
\begin{split}
\mathcal{W}_{\vartheta}(t) \,&:=\, \sum_{\vartheta'=0}^{2\vartheta} \big\|f(t)\big\|_{2,\vartheta'\!/2}^2 \,+\, \big\|\mathbf{E}(t)\big\|_{2}^2  \\
&\;\simeq\; \big\|f(t)\big\|_{2,\vartheta}^2 \,+\, \big\|\mathbf{E}(t)\big\|_{2}^2
\;,
\end{split}
\end{equation*}
and the weighted dissipation rate
\begin{equation*}
\begin{split}
\mathcal{V}_{\vartheta}(t) \,&:=\, \sum_{\vartheta'=0}^{2\vartheta} \big\|f(t)\big\|_{\sigma,\vartheta'\!/2}^2 \,+\, \big\|\mathbf{E}(t)\big\|_{2}^2 \\
&\;\simeq\; \big\|f(t)\big\|_{\sigma,\vartheta}^2 \,+\, \big\|\mathbf{E}(t)\big\|_{2}^2 \;.
\end{split}
\end{equation*}
Then the estimate (\ref{weighted-energy-inequality}) above is actually
\begin{equation}\label{temp}
\mathcal{W}_{\vartheta}(t)  \,+\, C_\vartheta \!\int_s^t\!\mathcal{V}_{\vartheta}(\tau)\,\dd\tau  \;\leq\; \mathcal{W}_{\vartheta}(s) \;.
\end{equation}

As the first portion
of our theorem, we show that the weighted energy (defined in Section\;\ref{SubsubSec:Dissipation-Energy})
$$\mathcal{E}_{\vartheta}[f(t)]
\,:=\; \mathcal{I}_{\vartheta}[f(t)] \,+ \int_{0}^{t}\! \mathcal{D}_{\vartheta}[f(\tau)] \,\dd\tau
\;\simeq\; \mathcal{W}_{\vartheta}(t) \,+ \int_{0}^{t}\! \mathcal{V}_{\vartheta}(\tau) \,\dd\tau$$
is uniformly bounded for arbitrarily large $\vartheta$
(under the assumption of integrable decay for $\|\partial_t f\|_{2}$).
Taking $s=0$ in (\ref{temp})
results in
\begin{equation} \label{apriori-decay-f_t}
\mathcal{E}_{\vartheta}[f(t)] \;\simeq\;
\mathcal{W}_{\vartheta}(t) \,+ \int_{0}^{t}\! \mathcal{V}_{\vartheta}(\tau) \,\dd\tau
\;\lesssim_{\,\vartheta}\; \mathcal{W}_{\vartheta}(0)
\;\simeq\; \mathcal{E}_{\vartheta}[f(0)]
\,\ls\, \varepsilon_0^{\,2}\;,
\end{equation}
which concludes the uniform weighted energy bound (\ref{energy-bound-2}).

\smallskip
{\it \underline{Step\;4}. $L^2$\;Time-Decay}.
With the boundedness, we now show that the weighted $L^2$\;norms decay at any algebraic rate.

Based on the observation that for fixed $\vartheta$, the dissipation rate $\|\cdot\|_{\sigma,\vartheta}$ is not stronger than the instant energy $\|\cdot\|_{2,\vartheta}$ by \,(\ref{energy-dissipation-relation}), we shall perform interpolation for $\langle v\rangle^{2\vartheta}$ between the weight functions $\langle v\rangle^{2\vartheta-1}$ and $\langle v\rangle^{2\vartheta+k}$ for any given $k\in\mathbb{N}$,
then bound the stronger norm by using the result (\ref{apriori-decay-f_t}) of last step.
This yields
\begin{equation*}
\big\|f(t)\big\|_{2,\vartheta} \,\leq\, \big\|f(t)\big\|_{2,\vartheta+\frac{k}{2}}^{\frac{1}{k+1}} \big\|f(t)\big\|_{2,\vartheta-\frac{1}{2}}^{\frac{k}{k+1}}
\,\leq\, \Big(C_{\vartheta,k}\, \|f_0\|_{2,\vartheta+\frac{k}{2}}\Big)^{\frac{1}{k+1}} \big\|f(t)\big\|_{2,\vartheta-\frac{1}{2}}^{\frac{k}{k+1}}
\;\lesssim\; \varepsilon_0^{\frac{1}{k+1}} \big\|f(t)\big\|_{\sigma,\vartheta}^{\frac{k}{k+1}} \;,
\end{equation*}
which further implies the dissipation rate
\begin{equation}
                    \label{eq11.29}
\mathcal{V}_{\vartheta}(t) \;\geq\; C_{\vartheta,k}\, \varepsilon_0^{-\frac{2}{k}}\, \mathcal{W}_{\vartheta}(t)^{\frac{k+1}{k}} \;.
\end{equation}

Let
\begin{align*} 
    \zz(s) := C_{\vartheta,k}\,\varepsilon_0^{-\frac{2}{k}}\int_s^{\infty}\!\mathcal{W}_{\vartheta}(\tau)^{\frac{k+1}{k}} \dd\tau.
\end{align*}
Then we get
\begin{align} \label{Z'-eq}
   \zz'(s) = -\,C_{\vartheta,k}\,\varepsilon_0^{-\frac{2}{k}}
   \,\mathcal{W}_{\vartheta}(s)^{\frac{k+1}{k}}.
\end{align}
From \eqref{temp} with $t=\infty$ and \eqref{eq11.29}, we have
\begin{align} \label{Z-ineq}
    \zz(s)\,\ls\, \mathcal{W}_{\vartheta}(s).
\end{align}
Combining (\ref{Z'-eq}) and (\ref{Z-ineq}) yields
\begin{align*} 
    \zz'(s) + \bar C_{\vartheta,k}\,\varepsilon_0^{-\frac{2}{k}}\zz(s)^{\frac{k+1}{k}}\leq 0.
\end{align*}
This is a Bernoulli-type differential inequality for $\zz(s)$.
We solve it over $[0,s]$ by standard ODE method with integrating factor $-\frac{1}{k}\zz^{\,-\frac{k+1}{k}}$
to obtain
\begin{align} \label{Z-deacy}
    \zz(s) \,\ls_{\vartheta,k}\, \varepsilon_0^{\,2} \left(1+\frac{s}{k}\right)^{-k}.
\end{align}
On the other hand,
raising both sides of \eqref{temp} to the power $\frac{k+1}{k}$,
we have that
for any $s\leq t$ with $|t-s|\geq 1$,
\begin{align*}
\mathcal{W}_{\vartheta}(t)^{\frac{k+1}{k}} \,\leq\, \mathcal{W}_{\vartheta}(s)^{\frac{k+1}{k}}.
\end{align*}
Integrating the above inequality over $s\in\left[\frac{t}{2},t-1\right]$ with for $t>2$, we obtain
\begin{align} \label{W<Z}
    \left(\frac{t}{2}-1\right)\mathcal{W}_{\vartheta}(t)^{\frac{k+1}{k}} \,\ls_{\vartheta,k}\, \varepsilon_0^{\,\frac{2}{k}}\,\zz\!\left(\frac{t}{2}\right).
\end{align}
Therefore, combining (\ref{W<Z}) with (\ref{Z-deacy}), we have (for $t$ large)
\begin{align*}
    \mathcal{W}_{\vartheta}(t)^{\frac{k+1}{k}}
    \,\ls_{\vartheta,k}\,\varepsilon_0^{2+\frac{2}{k}}\br{t}^{-1}\left(1+\frac{t}{k}\right)^{-k},
\end{align*}
which further yields
\begin{align*}
    \mathcal{W}_{\vartheta}(t)
    \,\ls_{\vartheta,k}\,\varepsilon_0^{\,2} \left(1+\frac{t}{k}\right)^{-k}.
\end{align*}
Finally, with $\big\|\mathbf{E}(t)\big\|_{H^1} \lesssim \big\|f(t)\big\|_{2}$ by Lemma\;\ref{Lem:E_f-L^p},
we conclude (\ref{L^2-decay-2}) in our theorem.

\smallskip
{\it \underline{Step\;5}. $\partial_t f$ Estimates}.
Lastly, to close the a priori estimates, we still need to prove the similar $L^2$\;decay bound for $\partial_t f$.
We follow the same procedure as before in this section.

Let $\dot{f} := \partial_t f$ and $\dot{\mathbf{E}} := \partial_t \mathbf{E}$.
Taking $\partial_t$ derivative of \eqref{Eq:linearized-VL_f=f}, we get the equation
\begin{equation} \label{Eq:linearized-VL_f_t}
\begin{split}
&\partial_t \dot{f} + v\cdot\nabla_{\!x}\dot{f} + L\dot{f} - 2\sqrt{\mu}\,v\cdot \dot{\mathbf{E}} \\
\,=\;&\, \Big\{ \Gamma[\dot{f},f] + \Gamma[f,\dot{f}] \Big\}
- \Big\{ \dot{\mathbf{E}}\cdot\nabla_{\!v}f + \mathbf{E}\cdot\nabla_{\!v}\dot{f} \Big\}
+ \Big\{ \big(v\cdot \dot{\mathbf{E}} \big)f + \big(v\cdot \mathbf{E} \big)\dot{f} \Big\} \;,
\end{split}
\end{equation}
where $\dot{\mathbf{E}} := \partial_t\, \mathbf{E}_f = \mathbf{E}_{\dot{f}} = -\nabla_{\!x}\phi_{\dot{f}}$ with $\dot{\phi} := \partial_t\, \phi_f = \phi_{\dot{f}}$ satisfying
\begin{equation*} 
-\Delta_x \dot{\phi} = \int_{\R^3}\! \sqrt{\mu}\,\dot{f}\,\dd v \,=: \rho[\dot{f}] \;.
\end{equation*}

First of all, we check that the kernel estimate and coercivity result hold by replacing $f$ with $\dot f$ of \,Lemma\;\ref{Lem:abc}\, and \,Corollary\;\ref{Cor:coercivity-est-L}\, for the new equations with harmless temporal derivatives.
The argument is almost identical, so we omit the details.

Next, similar to the (weighted) energy estimate for $f$, upon multiplying \eqref{Eq:linearized-VL_f_t} by the factor $e^\phi$, the last term $(v\cdot \mathbf{E} )\dot{f}$ is cancelled. We first multiply $e^\phi\dot{f}$ (resp.\;$\langle v\rangle^{2\vartheta} e^\phi\dot{f}$ for the weighted case) on both sides and integrate over $\Omega\times\R^3$ to get
\begin{equation*} 
\begin{split}
&\frac{1}{2}\,\frac{\dd}{\dd t}\, \big\|e^{\phi} \dot{f}\big\|_{2}^2
\,+\, \frac{1}{2}\iint_{\Omega\times\R^3}\! \nabla_{\!x}\cdot \left\{v \big(e^{\phi} \dot{f}\,\big)^2\right\}
\,+ \iint_{\Omega\times\R^3}\! \big(e^{\phi} \dot{f}\,\big)\, L\big(e^{\phi} \dot{f}\,\big)
\,- \iint_{\Omega\times\R^3}\! 2\,\big(e^{2\phi} \dot{f}\,\big) \sqrt{\mu}\,v\cdot \dot{\mathbf{E}} \\
\,=\;& \iint_{\Omega\times\R^3}\! \big(e^{2\phi} \dot{f}\,\big)\, \Gamma[\dot{f},f]
\,+ \iint_{\Omega\times\R^3}\! \big(e^{2\phi} \dot{f}\,\big)\, \Gamma[f,\dot{f}]
\,- \iint_{\Omega\times\R^3}\! e^{2\phi}\dot{f}\,\dot{\mathbf{E}}\cdot\nabla_{\!v}f
\,-\, \frac{1}{2}\iint_{\Omega\times\R^3}\! e^{2\phi}\,\mathbf{E}\cdot\nabla_{\!v}\big(\dot{f}^2\big) \\
&\,+ \iint_{\Omega\times\R^3}\! e^{2\phi} \big(v\cdot \dot{\mathbf{E}} \big)f\dot{f}
\,+ \iint_{\Omega\times\R^3}\! \big(e^{\phi} \dot{f}\,\big)^2 \dot{\phi} \;,
\end{split}
\end{equation*}
and then we examine this resulting equation term by term.
The estimates from Step\;1 to Step\;3 of similar form are valid for $\dot{f}$ with only a few changes in the proof which we point out below.

In the case without weight ($\vartheta=0$):
For the fourth term on the LHS, in order to avoid producing higher-order derivatives terms (which we cannot control),
we modify our previous argument in a more direct way.
Extracting the main contribution of this term,
we deduce that
\begin{equation*}
\begin{split}
- \iint_{\Omega\times\R^3}\! 2\,\dot{f} \sqrt{\mu}\,v\cdot \dot{\mathbf{E}}
&\;=\; \iint_{\Omega\times\R^3}\! 2\,\dot{f} \sqrt{\mu}\,v\cdot \nabla_{\!x}\dot{\phi}
\;=\; - \iint_{\Omega\times\R^3}\! 2\,\dot{\phi} \sqrt{\mu}\,v\cdot \nabla_{\!x}\dot{f} \\
&\;=\; - \int_{\Omega}\! 2\,\dot{\phi}\: \nabla_{\!x}\!\cdot\!\left( \int_{\R^3}\!
v\sqrt{\mu}\,\dot{f}\,\dd v\right) \dd x
\;=\; \int_{\Omega}\! 2\,\dot{\phi}\: \partial_t\!\left( \int_{\R^3}\!
\sqrt{\mu}\,\dot{f}\,\dd v\right) \dd x \\
&\;=\; - \int_{\Omega}\! 2\,\dot{\phi}\, \partial_t \Delta_x\dot{\phi} \,\dd x
\;=\; \int_{\Omega}\! 2\, (\nabla_{\!x}\dot{\phi})\cdot (\nabla_{\!x}\partial_t\dot{\phi}) \,\dd x
\;=\; \frac{\dd}{\dd t} \int_{\Omega}\! |\dot{\mathbf{E}}|^2 \;.
\end{split}
\end{equation*}
Here the four equality is due to the continuity equation \eqref{continuity-eq} of conservation laws, which is essentially the same kind of conserved quantity derived in a similar way from the equation \,\eqref{Eq:linearized-VL_f=f}\, in Step\;1.
Then the remainder can be easily controlled by
\begin{equation*}
\begin{split}
- \iint_{\Omega\times\R^3}\! 2\,\big(e^{2\phi} \!-1 \big) \dot{f} \sqrt{\mu}\,v\cdot \dot{\mathbf{E}}
&\;\lesssim\; \big\|e^{2\phi} \!-1\big\|_{\infty} \big\|v\sqrt{\mu}\,\dot{f}\big\|_{2}\, \|\dot{\mathbf{E}}\|_{2} \\
&\;\lesssim\; \|\phi\|_{\infty} \|\dot{f}\|_{\sigma}^2 \;,
\end{split}
\end{equation*}
and thus can be absorbed as $\|\phi\|_{\infty} \lesssim \|f\|_{\infty} \ll 1$.

Another difficulty is that, the nonlinearity generates two more terms containing $\|f\|_2^2$ and $\|f\|_{\sigma}^2$ related to $f$ rather than $\dot{f}$ (which cannot be estimated directly, but requires some special treatment).
To be specific,
\begin{equation*}
\iint_{\Omega\times\R^3}\! \big(e^{2\phi} \dot{f}\,\big)\, \Gamma[\dot{f},f]
\;\lesssim\; \|\dot{f}\|_{\infty} \|f\|_{\sigma} \|\dot{f}\|_{\sigma}
\;\lesssim\;  \|\dot{f}\|_{\infty} \big( \|f\|_{\sigma}^2 + \|\dot{f}\|_{\sigma}^2 \big) \;,
\end{equation*}
and
\begin{equation*}
\begin{split}
- \iint_{\Omega\times\R^3}\! e^{2\phi}\dot{f}\,\dot{\mathbf{E}}\cdot\nabla_{\!v}f
&\;=\, \iint_{\Omega\times\R^3}\! e^{2\phi}f\,\dot{\mathbf{E}}\cdot\nabla_{\!v}\dot{f}
\;\lesssim\; \|\nabla_{\!v}\dot{f}\|_{\infty}\|f\|_{2}\|\dot{\mathbf{E}}\|_{2} \\
&\;\lesssim\; \|\nabla_{\!v}\dot{f}\|_{\infty}\|f\|_{2}\|\dot{f}\|_{\sigma}
\;\lesssim\; \|\nabla_{\!v}\dot{f}\|_{\infty} \big( \|f\|_{2}^2 + \|\dot{f}\|_{\sigma}^2 \big) \;,
\end{split}
\end{equation*}
which requires the additional assumption of smallness of $\|\nabla_{\!v}\dot{f}\|_{\infty}$
Also, the third extra term
\begin{equation*}
\iint_{\Omega\times\R^3}\! e^{2\phi} \big(v\cdot \dot{\mathbf{E}} \big)f\dot{f}
\;\lesssim\; \big\|\langle v\rangle^{\frac{1}{2}} v\,f\big\|_{\infty} \|\dot{f}\|_{\sigma} \|\dot{\mathbf{E}}\|_{2}
\;\lesssim\; \|f\|_{\infty,\frac{3}{2}} \|\dot{f}\|_{\sigma}^2 
\end{equation*}
needs the smallness of weighted $L^\infty$\;norm.

In summary, under the assumption that
\begin{equation*}
\sup_{t}\|f\|_{\infty,\bar{\vartheta}} \,+\, \sup_{t}\|\dot{f}\|_{\infty,\bar{\vartheta}}
\,+\, \sup_{t}\|\nabla_{\!v}\dot{f}\|_{\infty} \,\ll\, 1
\end{equation*}
for some $\bar{\vartheta} \geq 3$ and using the same technique to multiply $\ue^{\psi}$ on both sides,
we have
\begin{equation} \label{energy-est-d_t}
\begin{split}
&\left(\,\big\|e^{\psi(t)} e^{\phi(t)}\dot{f}(t)\big\|_{2}^2 \,+\, \big\|e^{\psi(t)} e^{\phi(t)}\dot{\mathbf{E}}(t)\big\|_{2}^2 \right)
\,+\, \delta' \left\{\int_s^t\!\big\|\dot{f}(\tau)\big\|_{\sigma}^2 \,\dd\tau
\,+ \int_s^t\!\big\|\dot{\mathbf{E}}(\tau)\big\|_{2}^2 \,\dd\tau \right\} \\
\;\lesssim&\;
 \left(\,\big\|e^{\psi(s)} e^{\phi(s)}\dot{f}(s)\big\|_{2}^2 \,+\, \big\|e^{\psi(s)} e^{\phi(s)}\dot{\mathbf{E}}(s)\big\|_{2}^2 \right)\\
&+ \int_s^t\! \big\|\dot{f}(\tau)\big\|_{\infty} \big\|f(\tau)\big\|_{\sigma}^2 \,\dd\tau
\,+ \int_s^t\! \big\|\nabla_{\!v}\dot{f}(\tau)\big\|_{\infty} \big\|f(\tau)\big\|_{2}^2 \,\dd\tau \\
\;\leq&\left(\,\big\|e^{\psi(s)} e^{\phi(s)}\dot{f}(s)\big\|_{2}^2 \,+\, \big\|e^{\psi(s)} e^{\phi(s)}\dot{\mathbf{E}}(s)\big\|_{2}^2 \right)\\
&+\sup_{\tau\geq 0}\|\dot{f}\|_{\infty} \cdot\! \int_s^t\! \big\|f(\tau)\big\|_{\sigma}^2 \,\dd\tau
\,+\, \sup_{\tau\geq 0}\|\nabla_{\!v}\dot{f}\|_{\infty} \cdot\! \int_s^t\! \big\|f(\tau)\big\|_{2}^2 \,\dd\tau 
\end{split}
\end{equation}
for any $0\leq s<t$ with $|t-s|\in\mathbb{Z}^+$.

For the weighted case ($\vartheta>0$),
most of the estimates are similar to those of weighted version for $f$ in Step\;2 simply replacing certain norms on $f$ with those on $\dot{f}$, except for the three additional terms from nonlinearity.
In particular,
\begin{equation*}
\begin{split}
- \iint_{\Omega\times\R^3}\! \langle v\rangle^{2\vartheta} e^{2\phi}\dot{f}\,\dot{\mathbf{E}}\cdot\nabla_{\!v}f
&\;=\, \iint_{\Omega\times\R^3}\! e^{2\phi}\, \dot{\mathbf{E}}\cdot\nabla_{\!v}\Big(\langle v\rangle^{2\vartheta}\Big)\, f\dot{f}
\,+ \iint_{\Omega\times\R^3}\! \langle v\rangle^{2\vartheta} e^{2\phi}f\,\dot{\mathbf{E}}\cdot\nabla_{\!v}\dot{f} \\
&\;\lesssim\; \|\dot{\mathbf{E}}\|_{\infty} \|f\|_{\sigma,\vartheta} \|\dot{f}\|_{\sigma,\vartheta}
\,+\, \|\nabla_{\!v}\dot{f}\|_{\infty} \|f\|_{2,2\vartheta} \|\dot{\mathbf{E}}\|_{2} \\
&\;\lesssim\; \|f\|_{\sigma,\vartheta} \|\dot{f}\|_{\sigma,\vartheta}^2
\,+\, \|\nabla_{\!v}\dot{f}\|_{\infty} \big( \|f\|_{2,2\vartheta}^2 + \|\dot{f}\|_{\sigma,\vartheta}^2 \big) \;.
\end{split}
\end{equation*}

The rest of this substep is a simple modification of the $\vartheta=0$ case.
Eventually, under the assumption that
\begin{equation*}
\sup_{t}\|f\|_{\sigma,\vartheta}
\,+\, \sup_{t}\|f\|_{\infty,\vartheta+\frac{3}{2}}
\,+\, \sup_{t}\|\dot{f}\|_{\infty}
\,+\, \sup_{t}\|\nabla_{\!v}\dot{f}\|_{\infty} \,\ll\, 1
\end{equation*}
and using the techniques of multiplying $\ue^{\psi}$ on both sides, we obtain the weighted energy estimate
\begin{equation} \label{weighted-energy-est-d_t}
\begin{split}
&\big\|e^{\psi(t)} e^{\phi(t)}\dot{f}(t)\big\|_{2,\vartheta}^2
\;+\, \int_s^t\!\big\|\dot{f}(\tau)\big\|_{\sigma,\vartheta}^2 \,\dd\tau \\
\;\lesssim\;&\, \big\|e^{\psi(s)} e^{\phi(s)}\dot{f}(s)\big\|_{2,\vartheta}^2+ C_\vartheta\! \int_s^t\!\big\|\dot{f}(\tau)\big\|_{\sigma}^2 \,\dd\tau\\
&\,+\, \sup_{\tau\geq 0}\|\dot{f}\|_{\infty} \cdot\! \int_s^t\! \big\|f(\tau)\big\|_{\sigma,\vartheta}^2 \,\dd\tau
\,+\, \sup_{\tau\geq 0}\|\nabla_{\!v}\dot{f}\|_{\infty} \cdot\! \int_s^t\! \big\|f(\tau)\big\|_{2,2\vartheta}^2 \,\dd\tau \;.
\end{split}
\end{equation}

Now we use these two inequalities (\ref{energy-est-d_t}) and (\ref{weighted-energy-est-d_t}) to conclude
\begin{equation} \label{weighted-energy-inequality-d_t}
\begin{split}
&\bigg(\,\sum_{\vartheta'=0}^{\vartheta} \big\|e^{\psi(t)} e^{\phi(t)}\dot{f}(t)\big\|_{2,\vartheta'\!/2}^2 + \big\|e^{\psi(t)} e^{\phi(t)}\dot{\mathbf{E}}(t)\big\|_{2}^2 \,\bigg) \\
&\,+\, \delta'\, \Bigg\{ \sum_{\vartheta'=0}^{\vartheta} \int_s^t\!\big\|\dot{f}(\tau)\big\|_{\sigma,\vartheta'\!/2}^2 \,\dd\tau
\,+ \int_s^t\!\big\|\dot{\mathbf{E}}(\tau)\big\|_{2}^2 \,\dd\tau \Bigg\} \\
\;\lesssim&\;\, \bigg(\,\sum_{\vartheta'=0}^{\vartheta} \big\|e^{\psi(s)} e^{\phi(s)}\dot{f}(s)\big\|_{2,\vartheta'\!/2}^2 + \big\|e^{\psi(s)} e^{\phi(s)}\dot{\mathbf{E}}(s)\big\|_{2}^2 \,\bigg)\\
&+\sup_{\tau\geq 0}\|\dot{f}\|_{\infty} \!\cdot\! \sum_{\vartheta'=0}^{\vartheta} \int_s^t\! \big\|f(\tau)\big\|_{\sigma,\vartheta'\!/2}^2 \,\dd\tau
\,+\, \sup_{\tau\geq 0}\|\nabla_{\!v}\dot{f}\|_{\infty} \!\cdot\! \sum_{\vartheta'=0}^{\vartheta} \int_s^t\! \big\|f(\tau)\big\|_{2,\vartheta'}^2 \,\dd\tau,
\end{split}
\end{equation}
by summing up over $\vartheta'\! \in [0,\vartheta]\cap \mathbb{N}$ for any given $\vartheta \in \mathbb{N}$ with proper weight.

Due to Lemma \ref{Lem:abc}, the above result only holds for $t-s\in\mathbb{Z}^+$. Using a similar argument as in Step 3, with the help of local well-posedness result, we can extend the estimate to arbitrary $t-s\geq1$.

Notice that there is an extra term of $f$ with the highest-power weight (serving as a ``source term''),
which forces us to combine the two weighted inequalities for $f$ and $\dot{f}$
with different range (summation limits)
so that this term can be absorbed.
Defining
\begin{equation*}
\begin{split}
\widetilde{\mathcal{W}}_{\vartheta}(t)
\,&:=\, \sum_{\vartheta'=0}^{2\vartheta+1} \big\|f(t)\big\|_{2,\vartheta'\!/2}^2
\,+ \sum_{\vartheta'=0}^{\vartheta} \big\|\dot{f}(t)\big\|_{2,\vartheta'\!/2}^2
\,+\, \big\|\mathbf{E}(t)\big\|_{2}^2 \,+\, \big\|\dot{\mathbf{E}}(t)\big\|_{2}^2
 \\
&\;\simeq\; \big\|f(t)\big\|_{2,\vartheta+\frac{1}{2}}^2 \,+\, \big\|\dot{f}(t)\big\|_{2,\frac{\vartheta}{2}}^2
\,+\, \big\|\mathbf{E}(t)\big\|_{2}^2 \,+\, \big\|\dot{\mathbf{E}}(t)\big\|_{2}^2
\;,
\end{split}
\end{equation*}
and
\begin{equation*}
\begin{split}
\widetilde{\mathcal{V}}_{\vartheta}(t)
\,&:=\, \sum_{\vartheta'=0}^{2\vartheta+1} \big\|f(t)\big\|_{\sigma,\vartheta'\!/2}^2
\,+ \sum_{\vartheta'=0}^{\vartheta} \big\|\dot{f}(t)\big\|_{\sigma,\vartheta'\!/2}^2
\,+\, \big\|\mathbf{E}(t)\big\|_{2}^2 \,+\, \big\|\dot{\mathbf{E}}(t)\big\|_{2}^2 \\
&\;\simeq\; \big\|f(t)\big\|_{\sigma,\vartheta+\frac{1}{2}}^2 \,+\, \big\|\dot{f}(t)\big\|_{\sigma,\frac{\vartheta}{2}}^2
\,+\, \big\|\mathbf{E}(t)\big\|_{2}^2 \,+\, \big\|\dot{\mathbf{E}}(t)\big\|_{2}^2 \;.
\end{split}
\end{equation*}
Then (\ref{weighted-energy-inequality-d_t}) together with (\ref{weighted-energy-inequality}) leads to
\begin{equation*}
\widetilde{\mathcal{W}}_{\vartheta}(t)  \,+ \int_s^t\!\widetilde{\mathcal{V}}_{\vartheta}(\tau)\,\dd\tau  \,\ls\,  \widetilde{\mathcal{W}}_{\vartheta}(s).
\end{equation*}
Finally,
mimicking the derivation of (\ref{energy-bound-2}) and (\ref{L^2-decay-2}), the desired bounds for $\partial_t f$ follow. Therefore we complete the proof of the theorem.
\end{proof}

\section{Preliminaries for the Ultraparabolic Equations} \label{Sec:Lemmas}

In the following, we will discuss the setup of $S^p$ estimates. Since this theory is developed for the whole space, we will temporarily write $x\in\r^3$ in this section and introduce the extension $\Omega\rt\r^3$ in the next section.

\subsection{Structure of the Ultraparabolic Operator} \label{SubSec:ultraparabolic-opt-structure}

In order to prove the $S^p$ estimates and further regularity results,
we rewrite the linearized equation \eqref{Eq:reformulated-VL_f} so that it takes the form of an ultraparabolic equation (cf.\;\cite{Bramanti.Cerutti.Manfredini1996, Polidoro.Ragusa1998, Dong.Yastrzhembskiy2021}):
\begin{equation} \label{Eq:ultraparabolic-VL_f}
\partial_t f + v\cdot\nabla_{\!x}f - \sigma_{\!G}^{ij}\partial_{v_i v_j} f \,=\, \mathcal{S},
\end{equation}
where the source term
\begin{equation*} 
\mathcal{S}(t,x,v) \,:=\, \,\partial_{v_i}\sigma_{\!G}^{ij}\partial_{v_j} f
+ \big\{a_g - \mathbf{E}_g\big\}\cdot\nabla_{\!v}f
+ \Big\{ \bar{K}_{\!g} f + \big(v\cdot\mathbf{E}_g\big)f + 2\sqrt{\mu}\,v\cdot\mathbf{E}_f \Big\}.
\end{equation*}
Our equation \eqref{Eq:ultraparabolic-VL_f} corresponds to a special case of the class of ultraparabolic operators
\begin{equation*} 
\mathcal{L} := \sum_{i,\,j=1}^{m_0} a_{ij}(\mathbf{z})\,\partial_{x_i x_j}
+ \sum_{i,\,j=1}^{N} b_{ij}\,x_i\partial_{x_j} - \partial_t
\end{equation*}
with $m_0=3$, $N=6$, $\mathbf{z}:=(\mathbf{x},t) \in\R^{6+1}$, $\mathbf{x}:=(x_i)_{1\leq i\leq 6}=(v,x) \in\R_v^3\times\R_x^3$, and
\begin{equation*}
a_{ij}(\mathbf{z}) = \sigma_{\!G}^{ij}(t,x,v)
= \left\{\Phi^{ij}\ast\big[\mu+\mu^{1/2}g(t,x,v)\big]\right\}(v) ,
\end{equation*}
$b_{14}=b_{25}=b_{36}=-1$ and the rest of $b_{ij}$ are zeros,
so that
\begin{equation*}
\mathcal{Y} := \sum_{i,\,j=1}^{N} b_{ij}\,x_i\partial_{x_j} - \partial_t = - \big(\partial_t + v\cdot\nabla_{\!x}\big) .
\end{equation*}

We now verify that the hypotheses made on the coefficients of $\mathcal{L}$ are satisfied in our case:

\smallskip
\textbf{(H.1)}\; The matrix of the second-order coefficients
$$\mathsf{A}(\mathbf{z}) := \big[a_{ij}(\mathbf{z})\big]_{1\leq i,\,j\leq 3} = \sigma_{\!G}(t,x,v)$$
is symmetric, and if $\|g\|_{\infty}<\varepsilon$ is sufficiently small,
there are constants $c_1,c_2>0$ such that
\begin{equation*}
c_1 \Lambda_1(v)\,|\eta|^2 \,\leq\, \eta^T\! A(\mathbf{z})\,\eta \,\leq\, c_2 \Lambda_2(v)\,|\eta|^2 
\end{equation*}
for every $\mathbf{z}=(t,x,v)\in\R^{6+1}$ and $\eta\in\R^3$.
Here $\Lambda_1(v), \Lambda_2(v) >0$ are equivalent to the eigenvalues of $\sigma_{\!G}$, satisfying
\begin{equation*}
\Lambda_1(v) \simeq (1+|v|)^{-3}, \quad \Lambda_2(v) \simeq (1+|v|)^{-1} .
\end{equation*}
Therefore, the principal part of $\mathcal{L}$ is an elliptic operator on $\R^3_v$, also it is uniformly elliptic restricted on a bounded subset of $v\in\R^3$ (see Lemma\;\ref{Lem:Landau-kernel-structure}\, in Section \ref{Sec:Prelim}).

\smallskip
\textbf{(H.2)}\; The constant matrix
\begin{equation*}
\mathsf{B} := \big[b_{ij}\big]_{1\leq i,\,j\leq 6} =
\begin{bmatrix}
\vspace{0.6pt}
\;0 & \mathsf{B}_1\, \\
\;0 & 0 \,
\end{bmatrix} ,
\qquad \mathsf{B}_1 =
\begin{pmatrix}
\vspace{0.5pt}
\,-1 & 0 & 0\, \\
\,0 & -1 & 0\, \\
\,0 & 0 & -1\, \\[0.5pt]
\end{pmatrix}
\end{equation*}
is upper triangular with $r=1$, $m_0=m_1=3$, and $\mathsf{B}_1$ is a $3\times 3$ block matrix of rank $3$.

It is known that under the conditions (H.1) and (H.2), the operator $\mathcal{L}_{\mathbf{z_0}}$ obtained by freezing the coefficients $a_{ij}$ at any fixed point $\mathbf{z_0}\in\R^{N+1}$ is called ``hypoelliptic''.
In our case, the equation \eqref{Eq:ultraparabolic-VL_f} is parabolic only in the velocity variable, while the (Liouville) transport operator $\mathcal{Y}$ has a mixing effect in the position-velocity phase space.

We  now introduce 
a quasi-distance in the space $\R_t\!\times\!\R_x^3\!\times\!\R_v^3$ and
$\R_x^3\!\times\!\R_v^3$.



\begin{definition}[Quasi-Distance]
For every $\mathbf{z}:=(t,x,v),\, \mathbf{w}:=(\tau,\xi,\nu) \in\R^{6+1}$, define
$$\mathsf{d}(\mathbf{z},\mathbf{w}) := \max\{|t-\tau|^{1/2},\, |x-\xi-(t-\tau)\nu|^{1/3},\, |v-\nu|\}.$$
Let $\hat{\mathsf{d}}\,(\hat{\mathbf{z}},\hat{\mathbf{w}})$ denote the restriction of the quasi-distance $\mathsf{d}$ on the phase-space $\R_x^3\times\R_v^3$, where $\hat{\mathbf{z}}:=(x,v)$, $\hat{\mathbf{w}}:=(\xi,\nu)$.
\end{definition}

The quasi-distance $\hat{\mathsf{d}}$ induces a topology on $\R_x^3\!\times\!\R_v^3$ in a natural way, which allows us to define the H\"{o}lder space $C^{\alpha}_{ }$.

\begin{definition}[$C^{0,\alpha}$ Space] \label{Def:Holder-spaces}
Let $\Pi$ be an open set in $\R^{6}$. We say a real-valued function $f$ on $\Pi$ satisfies the H\"{o}lder condition, or has (uniform) H\"{o}lder continuity, if there exist constants $\alpha\in(0,1]$ and $M>0$ such that
$$\big|f(\hat{\mathbf{z}})-f(\hat{\mathbf{w}})\big| \,\leq\, M \big|\hat{\mathsf{d}}(\hat{\mathbf{z}},\hat{\mathbf{w}})\big|^{\alpha} $$
for every $\hat{\mathbf{z}}, \hat{\mathbf{w}} \in\Pi$.
In this case, $f$ can be equipped with a semi-norm
$$|f|_{C^{\,0,\alpha}_{ }(\Pi)} \,:= \sup_{\mathbf{z}\neq\mathbf{w}\in\Pi} \frac{|f(\hat{\mathbf{z}})-f(\hat{\mathbf{w}})|}{\big|\hat{\mathsf{d}}(\hat{\mathbf{z}},\hat{\mathbf{w}})\big|^{\alpha}} \;.$$
\end{definition}


To get the desired a priori estimates, we need to make an additional regularity assumption on the second-order coefficients of $\mathcal{L}$:

\smallskip
\textbf{(H.3)}\; For each $n\ge 0$ and $v_0\in \r^3$ such that $|v_0|\in [2^n-1,2^{n+1}-1)$, there exists a constant $R_0=R_0(n)\in (0,1)$ such that the coefficients $a_{ij}(\mathbf{z}) := \sigma_{\!G}^{ij}(t,x,v)$ satisfy for any $\mathbf{z}_0=(t_0,x_0,v_0)$ and $r \in (0, R_0]$,
\begin{equation}
			\label{eq3.1.0}
     \text{osc}_{x, v} (a, Q_r (\mathbf{z}_0)) \le \gamma_{  \star  },\quad\text{where}\,\gamma_{  \star  }  =   2^{ -3n\kappa(p)  }  \widetilde \gamma_{  \star } (p),
\end{equation}
\begin{align*}
  & \text{osc}_{x, v} (a, Q_r (\mathbf{z}_0))\\
&= r^{-(4d+2)} \int_{t_0 - r^2}^{t_0} \int_{  D_r (\mathbf{z}_0, t) \times D_r (\mathbf{z}_0, t)}
   |a (t, x_1, v_1) - a (t, x_2, v_2)| \, dx_1dv_1 dx_2dv_2 \, dt\notag,
\end{align*}
  and
 $$
    Q_r (\mathbf{z}_0) =  \{(t,x, v): t_0-r^2<t<t_0,|x  - x_0 - (t-t_0) v_0|^{1/3} < r, |v-v_0| < r      \},
  $$
$$
    D_r (\mathbf{z}_0, t) =  \{(x, v): |x  - x_0 - (t-t_0) v_0|^{1/3} < r, |v-v_0| < r      \}.
  $$
Here $\kappa(p)$ and $\widetilde \gamma_{  \star } (p)$ are positive constants depending on $p$, but are independent of $n$.

\subsection{Definition of $S^p$ Space}




\begin{definition}[$S^{p}_{ }$ Space]
Let $\Pi \subset \R^{6+1}$ be an open set. For $p\in (1,\infty)$,
we define the $S^p_{ }$ space
$$S^{p}_{ }(\Pi) \,:=\, \Big\{f\in L^{p}(\Pi):\, \partial_{v_i}f,\, \partial_{v_i v_j}f,\, \mathcal{Y}f \in L^{p}(\Pi),\; i,j=1,2,3 \Big\} ,$$
and assign the norm
\begin{equation} \label{S^p-norm}
\|f\|_{S^{p}_{ }(\Pi)} \,:=\, \left(\big\|f\big\|_{L^{p}(\Pi)}^p + \big\|D_vf\big\|_{L^{p}(\Pi)}^p + \big\|D^2_{vv}f\big\|_{L^{p}(\Pi)}^p + \big\|\mathcal{Y}f\big\|_{L^{p}(\Pi)}^p \right)^{\frac{1}{p}} ,
\end{equation}
where $\mathcal{Y} = - \big(\partial_t + v\cdot\nabla_{\!x}\big)$,
and we use the simplified notation
\begin{align*}
\big\|D_vf\big\|_{L^{p}(\Pi)}^p &:= \sum_{i=1}^{3}\big\|\partial_{v_i}f\big\|_{L^{p}(\Pi)}^p , \\
\big\|D^2_{vv}f\big\|_{L^{p}(\Pi)}^p &:=
\sum_{i,\,j=1}^{3}\!\big\|\partial_{v_i v_j}f\big\|_{L^{p}(\Pi)}^p .
\end{align*}
\end{definition}

Finally, we conclude this section by a Gagliardo-Nirenberg type
interpolation inequality, which gives the bound on the intermediate derivatives $D_v u$ from the control of the highest derivatives $D^2_{vv}u$ and additional information on the function $u$ itself.

\begin{lemma}[Interpolation] \label{Lem:interpolation}
Let $u\in S^{p}_{ }(\Pi)$ with $p\in (1,\infty)$, and $\Pi$ be a smooth open set in $\R^{6+1}$.
For any $0<\varepsilon< 1$, there exists a constant $C_\varepsilon = \frac{C}{\varepsilon} >0$ (depending only on $p$ and $\Omega$) such that
\begin{equation*} 
\big\|D_v u\big\|_{L^{p}(\Pi)} \,\leq\,
\varepsilon\, \big\|D^2_{vv}u\big\|_{L^{p}(\Pi)}
+ \,C_\varepsilon\, \|u\|_{L^{p}(\Pi)} .
\end{equation*}
\end{lemma}

\section{Extension Across the Specular-Reflection Boundary} \label{Sec:Boundary-Extension}

For the later use of $S^p$ theory, we need to extend the domain $\Omega$ to the whole space.

\subsection{Extension of Solutions to the Whole Space}
 In this subsection, we will show step by step the way of extending our equation satisfied on a bounded domain with specular-reflection BC to a whole space problem.

 \subsubsection{``Boundary-Flattening'' Transformation.}
 Let
 \begin{equation}
 \begin{array}{rcl}
 \vec{\Psi} :\;\; \overline{\Omega}\times \mathbb{R}^3 & \;\rightarrow\; & \overline{\mathbb{H}}_{-}\!\times \mathbb{R}^3 \\
 (x,v) \;& \;\mapsto\; & (y,w) \eqdef  \big(\vec{\psi}(x),A v \big)
 \end{array} \label{Phi}
 \end{equation}
 be the (local) transformation that flattens the boundary, where
 \begin{equation*}
 A\eqdef \Big[\frac{\partial y}{\partial x}\Big]=D\vec{\psi}
 \end{equation*}
 is a non-degenerate $3\!\times\! 3$ Jacobian matrix, and the explicit definition of $y\!=\!\vec{\psi}(x)$ will be given below.
 Let
 \begin{equation}
 \tilde{f}(t,y,w) \eqdef  f\big(t,\vec{\Psi}^{-1}(y,w)\big) = f\big(t,\vec{\psi}^{-1}(y),A^{-1}w\big) = f(t,x,v) \label{f-trans}
 \end{equation}
 denote the solution under the new coordinates.

 \begin{remark}
 It is crucial that we define our transformation $\vec{\Psi}$ for both $(x,v)$ variables in this certain form so that it preserves the characteristics and the transport operator as explained below (see also the subsection of \ref{Transformed-Equation} below for more details).  \end{remark}
 Suppose the boundary $\partial\Omega$ is (locally) given by the graph $x_3 = \rho(x_1,x_2)$, and $\big\{(x_1,x_2,x_3)\!\in\!\mathbb{R}^3: x_3\!<\!\rho(x_1,x_2)\big\} \subseteq \Omega$.
 Inspired by Lemma 15 in \cite{Guo.Kim.Tonon.Trescases2017}\footnote{The authors used spherical-type coordinates to make the map almost globally defined; here we just prefer the standard coordinates for simplicity.}, we define $y\!=\!\vec{\psi}(x)$ explicitly as follows:

 \begin{align*}
 \vec{\psi}^{-1} :\;\;
 \begin{pmatrix}
 y_1 \\ y_2 \\ y_3
 \end{pmatrix}
 & \mapsto\; \vec{\eta}(y_1,y_2) + y_3\cdot\vec{n}(y_1,y_2) \\
 & =\;
 \begin{pmatrix}
 y_1 \\ y_2 \\  \rho(y_1,y_2)
 \end{pmatrix}
 + y_3\cdot
 \begin{pmatrix}
 -\rho_1 \\ -\rho_2 \\ 1
 \end{pmatrix}\; =\;
 \begin{pmatrix}
 y_1 - y_3\cdot\rho_1 \\ y_2 - y_3\cdot\rho_2 \\ \rho + y_3
 \end{pmatrix}
 =:
 \begin{pmatrix}
 x_1 \\ x_2 \\ x_3
 \end{pmatrix},
 \end{align*}
 where we denote by $\rho=\rho(y_1,y_2)$, $\rho_i=\partial_i\rho(y_1,y_2),  i\!=\!1,2$, and
 \begin{align*}
 \vec{\eta}(y_1,y_2) &\eqdef  \big(y_1, y_2, \rho(y_1,y_2) \big)  \in \partial\Omega, \\[2pt]
 \partial_1\vec{\eta} &\eqdef  \frac{_{\partial\vec{\eta}}}{^{\partial y_1}} = \langle 1,0,\rho_1 \rangle, \\
 \partial_2\vec{\eta} &\eqdef  \frac{_{\partial\vec{\eta}}}{^{\partial y_2}} = \langle 0,1,\rho_2 \rangle.
 \end{align*}
Then the (outward) normal vector at the point $\vec{\eta}(y_1,y_2) \in\partial\Omega$ is chosen to be
 \begin{equation*}
 \vec{n}(y_1,y_2) \eqdef  \partial_1\vec{\eta} \times \partial_2\vec{\eta} = \langle -\rho_1,-\rho_2,1  \rangle.
 \end{equation*}

 From the above definition, we can see that the transformation $\vec{\psi}$ is ``boundary-flattening'' because it maps the points on the boundary $\{x_3 \!=\! \rho(x_1,x_2)\}$ to the plane $\{y_3\!=\!0\}$.
 We also remark that the map is locally well defined and is a smooth homeomorphism in a tubular neighborhood of the boundary (see Lemma 15 of \cite{Guo.Kim.Tonon.Trescases2017} for the rigorous proof).

 Directly we compute the Jacobian matrix
 \begin{align*}
 A^{-1} = D\vec{\psi}^{-1} = \Big[\frac{\partial x}{\partial y}\Big]
 =&\; \big[ \partial_1\vec{\eta}+y_3\!\cdot\!\partial_1\vec{n} ;  \partial_2\vec{\eta}+y_3\!\cdot\!\partial_2\vec{n} ;  \vec{n} \big] \\
 =&
 \begin{pmatrix}
  1\!-\!y_3\!\cdot\!\rho_{11} &\;\; -y_3\!\cdot\!\rho_{12} &\;\; -\rho_1  \\[2pt]
 -y_3\!\cdot\!\rho_{12} &\;\; 1\!-\!y_3\!\cdot\!\rho_{22} &\;\; -\rho_2 \\[2pt]
 \rho_1 & \rho_2 & 1
 \end{pmatrix} \\
 \xrightarrow{\text{on}\;\partial\Omega :\;  y_3=0}&\;
 \big[ \partial_1\vec{\eta} ;  \partial_2\vec{\eta} ;  \vec{n} \big]
 =
 \begin{pmatrix}
 1 &\; 0 &\; -\rho_1 \\[2pt]
 0 &\; 1 &\; -\rho_2 \\[2pt]
  \rho_1 &\;  \rho_2 & 1
 \end{pmatrix}.
 \end{align*}

Thus we can write out $\vec{\Psi}^{-1}$ as
 \begin{equation*}
 \vec{\Psi}^{-1} :\;\; ( y,w) \; \;\mapsto\;  ( x,v) \eqdef  \big( \vec{\psi}^{-1}(y),  A^{-1}w \big),
 \end{equation*}
 \begin{align*}
 \begin{pmatrix}
 w_1 \\ w_2 \\ w_3
 \end{pmatrix}
 \mapsto &\;
 \begin{pmatrix}
  1\!-\!y_3\!\cdot\!\rho_{11} &\;\; -y_3\!\cdot\!\rho_{12} &\;\; -\rho_1  \\[2pt]
 -y_3\!\cdot\!\rho_{12} &\;\; 1\!-\!y_3\!\cdot\!\rho_{22} &\;\; -\rho_2 \\[2pt]
 \rho_1 & \rho_2 & 1
 \end{pmatrix}
 \begin{pmatrix}
 w_1 \\ w_2 \\ w_3
 \end{pmatrix} \\[4pt]
 = &\;
 \begin{pmatrix}
 (1\!-\!y_3 \rho_{11})\cdot w_1 - y_3 \rho_{12}\cdot w_2 - \rho_1\cdot w_3 \\[2pt]
 -y_3 \rho_{12}\cdot w_1 + (1\!-\!y_3 \rho_{22})\cdot w_2 - \rho_2\cdot w_3 \\[2pt]
 \rho_1\cdot w_1 + \rho_2\cdot w_2 + w_3
 \end{pmatrix}
 =:
 \begin{pmatrix}
 v_1 \\ v_2 \\ v_3
 \end{pmatrix}.
 \end{align*}
 Restricted on the boundary $\partial\Omega$,  i.e., $\{y_3\!=\!0\}$, the map becomes
 \begin{align*}
 \begin{pmatrix}
 v_1 \\ v_2 \\ v_3
 \end{pmatrix}
 &=  w_1\cdot\partial_1\vec{\eta} + w_2\cdot\partial_2\vec{\eta} + w_3\cdot\vec{n} \\
 &=
 \begin{pmatrix}
 1 &\; 0 &\; -\rho_1 \\[2pt]
 0 &\; 1 &\; -\rho_2 \\[2pt]
  \rho_1 &\;  \rho_2 & 1
 \end{pmatrix}
 \begin{pmatrix}
 w_1 \\ w_2 \\ w_3
 \end{pmatrix}
 =
 \begin{pmatrix}
 w_1 - \rho_1\!\cdot\! w_3 \\[2pt]
 w_2 - \rho_2\!\cdot\! w_3 \\[2pt]
  \rho_1\!\cdot\! w_1 + \rho_2\!\cdot\! w_2 + w_3
 \end{pmatrix}.
 \end{align*}

 Now we are ready to show the key feature of the transformation $\vec{\Psi}$ -- preserving the ``specular symmetry'' on the boundary: it sends any two points $(x,v), (x,R_x v)$ on the phase boundary $\gamma=\partial\Omega\times\mathbb{R}^3$ with specular-reflection relation to two points on $\{y_3\!=\!0\}\times\mathbb{R}^3$ which are also specular-symmetric to each other.

 In other words, we have the following commutative diagram (when $x\!\in\!\partial\Omega$  i.e., $y_3\!=\!0$):
 \begin{equation*}
 \xymatrix{
 	(y,w)\ar[r]^{\vec{\Psi}^{-1}}\ar[d]_{R_y} & (x,v)\ar[d]^{R_x} \\
 	(y,R_y w)\ar[r]^{\vec{\Psi}^{-1}} & (x,R_x v)
 }
 \end{equation*}
 from which we can equivalently write
 \begin{equation*}
 A^{-1}\big(R_y w\big) = R_x\big(A^{-1}w\big)\quad \text{if}\;  y_3\!=\!0.
 \end{equation*}
 This can be verified by noticing that $\vec{n} = \partial_1\vec{\eta} \times \partial_2\vec{\eta}$, so
 \begin{align*}
 A^{-1}\big(R_y w\big) &= A^{-1}\langle w_1,w_2,-w_3\rangle = w_1\cdot\partial_1\vec{\eta} + w_2\cdot\partial_2\vec{\eta} - w_3\cdot\vec{n} \\
 &= R_x\big(w_1\cdot\partial_1\vec{\eta} + w_2\cdot\partial_2\vec{\eta} + w_3\cdot\vec{n}\big)
 = R_x\big(A^{-1}w\big).
 \end{align*}

 Having this property, the specular reflection boundary condition on the solutions is also preserved:
 \begin{equation*}
 \tilde{f}(t,y,w) = \tilde{f}(t,y,Rw)\quad \text{on}\;  \{y_3\!=\!0\},
 \end{equation*}
 where  $R\!\eqdef \!{\rm diag}\{1,1,-1\}$,
 which allows us to construct the mirror extension (as in the next subsection) that is consistent with this restriction and thus is automatically satisfied.

 To conclude this part, we carry out more computations for later use:
 \begin{equation}
 D\vec{\Psi}^{-1} = \left[\frac{\partial(x,v)}{\partial(y,w)}\right] = \left(
 \begin{array}{c|c}
 \frac{\partial x}{\partial y}  &  \frac{\partial x}{\partial w} \\[3pt]
 \hline \\[-12pt]
 \frac{\partial v}{\partial y}  &  \frac{\partial v}{\partial w}
 \end{array}
 \right)
 = \left(
 \begin{array}{c|c}
 A^{-1} &  0_{3\times 3} \\[1pt]
 \hline \\[-12pt]
 B & A^{-1}
 \end{array}
 \right) ,\label{DPhi}
 \end{equation}
 \begin{align*}
 B &\eqdef  \Big[\frac{\partial v}{\partial y}\Big] \\
 &=  \begin{pmatrix}
 - y_3 \rho_{1\!1\!1}\!\cdot\! w_1 \!-\! y_3 \rho_{1\!1\!2}\!\cdot\! w_2
 &\;\; - y_3 \rho_{1\!1\!2}\!\cdot\! w_1 \!-\! y_3 \rho_{1\!2\!2}\!\cdot\! w_2
 &\;\; - \rho_{11}\!\cdot\! w_1 \!-\! \rho_{12}\!\cdot\! w_2 \\[-2pt]
 - \rho_{11}\!\cdot\! w_3
 & - \rho_{12}\!\cdot\! w_3
 &  \\[3pt]
 - y_3 \rho_{1\!1\!2}\!\cdot\! w_1 \!-\! y_3 \rho_{1\!2\!2}\!\cdot\! w_2
 &\;\; - y_3 \rho_{1\!2\!2}\!\cdot\! w_1 \!-\! y_3 \rho_{2\!2\!2}\!\cdot\! w_2
 &\;\; - \rho_{12}\!\cdot\! w_1 \!-\! \rho_{22}\!\cdot\! w_2 \\[-2pt]
 - \rho_{12}\!\cdot\! w_3
 & - \rho_{22}\!\cdot\! w_3
 &  \\[3pt]
 \rho_{11}\!\cdot\! w_1 + \rho_{12}\!\cdot\! w_2
 & \rho_{12}\!\cdot\! w_1 + \rho_{22}\!\cdot\! w_2
 & 0
 \end{pmatrix}  ,
 \end{align*}
 \begin{align*}
 A &=  \frac{1}{\det\!\left(A^{-1}\right)}\cdot\left(A^{-1}\right)^{\ast} \\
 &= \left[ y_3^{ 2}\!\cdot\!\left(\rho_{11}\rho_{22}\!-\!\rho_{12}^{ 2}\right) + y_3\!\cdot\!\left(2\rho_1\rho_2\rho_{12}\!-\!\rho_2^{ 2}\rho_{11}\!-\!\rho_1^{ 2}\rho_{22}\!-\!\rho_{11}\!-\!\rho_{22}\right) + \left(\rho_1^{ 2}\!+\!\rho_2^{ 2}\!+\!1\right) \right]^{-1} \\[3pt]
 &\;\;\;\cdot
 \begin{pmatrix}
  (1\!+\!\rho_2^{ 2}) - y_3\!\cdot\!\rho_{22}
 &\;\; -\rho_1\rho_2 + y_3\!\cdot\!\rho_{12}
 &\;\; \rho_1 + y_3\!\cdot\!(\rho_2\rho_{12}\!-\!\rho_1\rho_{22})  \\[5pt]
 -\rho_1\rho_2 + y_3\!\cdot\!\rho_{12}
 &\;\; (1\!+\!\rho_1^{ 2}) - y_3\!\cdot\!\rho_{11}
 &\;\; \rho_2  + y_3\!\cdot\!(\rho_1\rho_{12}\!-\!\rho_2\rho_{11})  \\[5pt]
 -\rho_1 + y_3\!\cdot\!(\rho_1\rho_{22}\!-\!\rho_2\rho_{12})
 &\;\; -\rho_2 + y_3\!\cdot\!(\rho_2\rho_{11}\!-\!\rho_1\rho_{12})
 & 1 - y_3\!\cdot\!(\rho_{11}\!+\!\rho_{22})  \\[-1pt]
 & & + y_3^{ 2}\!\cdot\!(\rho_{11}\rho_{22}\!-\!\rho_{12}^{ 2})
 \end{pmatrix},
 \end{align*}

 \begin{align*}
 C &\eqdef  A^{-T}\!A^{-1}\\
 &=  \begin{pmatrix}
 y_3^{ 2}\!\cdot\!(\rho_{11}^{ 2}\!+\!\rho_{12}^{ 2})
 &\;\; y_3^{ 2}\!\cdot\!\rho_{12}(\rho_{11}\!+\!\rho_{22})
 &\;\; y_3\!\cdot\!(\rho_1\rho_{11}\!+\!\rho_2\rho_{12})  \\[-.5pt]
 - y_3\!\cdot\!2\rho_{11} + (\rho_1^{ 2}\!+\!1)
 & - y_3\!\cdot\!2\rho_{12} + \rho_1\rho_2
 &  \\[5pt]
 y_3^{ 2}\!\cdot\!\rho_{12}(\rho_{11}\!+\!\rho_{22})
 &\;\; y_3^{ 2}\!\cdot\!(\rho_{12}^{ 2}\!+\!\rho_{22}^{ 2})
 &\;\; y_3\!\cdot\!(\rho_1\rho_{12}\!+\!\rho_2\rho_{22})  \\[-.5pt]
 - y_3\!\cdot\!2\rho_{12} + \rho_1\rho_2
 & - y_3\!\cdot\!2\rho_{22} + (\rho_2^{ 2}\!+\!1)
 &  \\[5pt]
 y_3\!\cdot\!(\rho_1\rho_{11}\!+\!\rho_2\rho_{12})
 & y_3\!\cdot\!(\rho_1\rho_{12}\!+\!\rho_2\rho_{22})
 & \rho_1^{ 2}\!+\!\rho_2^{ 2}\!+\!1
 \end{pmatrix}, \\[10pt]
 & C^{-1} = AA^T =  \frac{1}{\det(C)}\cdot C^{\ast}.
 \end{align*}\begin{remark}
 	Here we just assume $\rho$ is (locally) smooth enough and its derivatives remain uniformly bounded, so that all the coefficients of transformed equations where the above matrices appear will keep roughly the same size as the original ones.
 \end{remark}

 \subsubsection{Mirror Extension Across the Specular-Reflection Boundary.}

 After flattening the boundary, we then ``flip over'' $\tilde{f}$ to the upper half space by setting
 \begin{equation}
 \bar{f}(t,y,w) \stackrel{\text{def}}{=} \left\{
 \begin{array}{rcl}
 & \tilde{f}(t,y,w), & \text{if}\;\; y\!\in \overline{\mathbb{H}}_{-} \\[4pt]
 & \tilde{f}(t,Ry,Rw),\quad & \text{if}\;\; y\!\in \overline{\mathbb{H}}_{+}
 \end{array}\right.,
 \label{f-bar}
 \end{equation}
 where  $R\!\eqdef \!{\rm diag}\{1,1,-1\}$. The similar notation also applies to other variables. Combined with the corresponding partition of unity, we are able to define our solutions in the whole space (locally).
 \begin{remark}
  	The above construction of extension coincides with the specular reflection boundary condition, which in turn makes it a well-defined and continuous extension across the boundary. This observation suggests that, unfortunately, we cannot apply the same kind of extension to other boundary condition cases.
 \end{remark}
 Also, it is worth pointing out the necessity of ``continuity of $\bar{f}$ across the boundary'' lies in that, on one hand, it ensures $\bar{f}$ is indeed a solution (at least) in the weak sense in the whole space (see Section \ref{Well-definedness}); on the other hand, $g$ will appear in the coefficients of the ultra-parabolic form, and we require some kind of continuity of the second-order coefficient for the $S^p$ estimate.

 \subsubsection{Transformed Equations.} \label{Transformed-Equation}
 By using the chain rule with our definitions (\ref{Phi}) and (\ref{f-trans}) of the transformation $\vec{\Psi}$, we first compute the transformed equation satisfied by $\tilde{f}$ in the lower half space:\footnote{We use the column vector convention in the following matrix operation expressions.}
 \begin{equation*}
 \partial_t f = \partial_t \tilde{f},
 \end{equation*}
 \begin{align*}
 v\cdot\nabla_{\!x} f &= \big(A^{-1}w\big)^T \left\{A^T  \nabla_{\!y}\tilde{f} + \big[\frac{_{\partial w}}{^{\partial x}}\big]^T \nabla_{\!w}\tilde{f}\right\} \\
 &= w^T\big(A^{-T}A^T\big)\nabla_{\!y}\tilde{f} + \big(A^{-1}w\big)^T \big(A\big[\frac{_{\partial v}}{^{\partial y}}\big]A\big)^T \nabla_{\!w}\tilde{f} \\
 &= w\cdot\nabla_{\!y} \tilde{f} + \big(ABw\big)\cdot\nabla_{\!w}\tilde{f},
 \end{align*}
 \begin{equation*}
 a_g\cdot\nabla_{\!v}f = \widetilde{a_g}\cdot\big(A^T  \nabla_{\!w}\tilde{f} \big) = \big(A \widetilde{a_g}\big)\cdot\nabla_{\!w}\tilde{f},
 \end{equation*}
 \begin{equation*}
 \mathbf{E}_g\cdot\nabla_{\!v}f = \widetilde{\mathbf{E}_g}\cdot\big(A^T  \nabla_{\!w}\tilde{f} \big) = \big(A \widetilde{\mathbf{E}_g}\big)\cdot\nabla_{\!w}\tilde{f},
 \end{equation*}
 and
 \begin{equation*}
 \nabla_v\cdot\big(\sigma_{\!G}\nabla_{\!v} f\big) = \nabla_w\cdot\left(\big[A \widetilde{\sigma}_{\!G}A^T\big] \nabla_{\!w} \tilde{f} \right).
 \end{equation*}
 See (\ref{DPhi}) for explicit definition of $A, B$, and
 \begin{align*}
 \widetilde{a_g}(t,y,w) &\eqdef  a_g\big(t,\vec{\Psi}^{-1}(y,w)\big) = a_g(t,x,v), \\
 \widetilde{\mathbf{E}_g}(t,y) &\eqdef  \mathbf{E}_g\big(t,\vec{\psi}^{-1}(y)\big) = \mathbf{E}_g(t,x), \\
 \widetilde{\sigma}_{\!G}(t,y,w) &\eqdef  \sigma_{\!G}\big(t,\vec{\Psi}^{-1}(y,w)\big) = \sigma_{\!G}(t,x,v).
 \end{align*}

 Based on our construction of the extension (\ref{f-bar}), we then go on deriving the equation satisfied by $\bar{f}$ for the upper half space:
\begin{equation*}
\partial_t \bar{f}(t,y,w)= \partial_t \tilde{f}(t,Ry,Rw) ,
 \end{equation*}
 \begin{equation*}
w\cdot\nabla_{\!y} \bar{f}(t,y,w)
=w^{ T}\big(R^T\!R^T\big)\nabla_{\!y}\bar{f}(t,y,w)
=\big(Rw\big)^T R^T  \nabla_{\!y}\bar{f}(t,y,w)
= w\cdot\nabla_{\!y} \tilde{f} (t,Ry,Rw),
 \end{equation*}
 \begin{equation*}
\big(R\bar{A}\bar{B}Rw\big)\cdot \nabla_{\!w}\bar{f}(t,y,w)= \big(\bar{A}\bar{B}Rw\big)\cdot R^T\nabla_{\!w}\bar{f}(t,y,w) =
 \big(ABw\big)\cdot\nabla_{\!w}\tilde{f} (t,Ry,Rw) ,
 \end{equation*}
 \begin{equation*}
 \big(R\bar{A} \overline{a}_g\big)\cdot \nabla_{\!w}\bar{f}(t,y,w)
= \big(\bar{A} \overline{a}_g\big)\cdot R^T\nabla_{\!w}\bar{f} (t,y,w)
= \big(A \widetilde{a_g}\big)\cdot\nabla_{\!w}\tilde{f}(t,Ry,Rw)  ,
 \end{equation*}
 \begin{equation*}
 \big(R\bar{A} \overline{\mathbf{E}}_g\big)\cdot \nabla_{\!w}\bar{f}(t,y,w)
= \big(\bar{A} \overline{\mathbf{E}}_g\big)\cdot R^T\nabla_{\!w}\bar{f}(t,y,w)
=  \big(A \widetilde{\mathbf{E}_g}\big)\cdot\nabla_{\!w}\tilde{f} (t,Ry,Rw),
 \end{equation*}
 \begin{equation*}
 \nabla_{\!w}\cdot\Big(\big[R\bar{A} \overline{\sigma}_{\!G}\bar{A}^T\!R\big] \nabla_{\!w} \bar{f} (t,y,w)\Big)=
 \nabla_w\cdot\left(\big[A \widetilde{\sigma}_{\!G}A^T\big] \nabla_{\!w} \tilde{f}(t,Ry,Rw) \right),
 \end{equation*}
 where
 \begin{align*}
 \bar{A}(y) \eqdef  A(Ry), \quad
 \bar{B}(y,w) \eqdef  B(Ry,Rw),
 \end{align*}
 and $\overline{a}_g$, $\overline{\mathbf{E}}_g$, $\overline{\sigma}_{\!G}$ are $a_g$, $\mathbf{E}_g$, $\sigma_{\!G}$ defined with $(t,y,w)$, respectively.

 Summing up the above computations, we now obtain that $\bar{f}$ satisfies the following equation in the lower and upper space, respectively:
 \begin{equation}
 \partial_t \bar{f} + w\cdot\nabla_{\!y} \bar{f} = \nabla_{w}\cdot\big(\mathbb{A} \nabla_{\!w} \bar{f} \big)+\mathbb{B}\cdot\nabla_{\!w}\bar{f}+\mathbb{C}\bar{f},
 \label{f-bar-eq}
 \end{equation}
 where the coefficients $\mathbb{A}$, $\mathbb{B}$, and $\mathbb{C}$ are piecewise-defined:
 \begin{equation}
 \mathbb{A}(t,y,w) \eqdef  \left\{
 \begin{array}{rcl}
 & \widetilde{\mathbb{A}} \eqdef  A \widetilde{\sigma}_{\!G}A^T, & \text{if}\;\; y\!\in \mathbb{H}_{-} \\[4pt]
 & \overline{\mathbb{A}} \eqdef  R\bar{A} \overline{\sigma}_{\!G}\bar{A}^T\!R,\quad & \text{if}\;\; y\!\in \mathbb{H}_{+}
 \end{array}\right.,
 \label{coeff-A}
 \end{equation}

 \begin{equation}\label{temp 1}
 \mathbb{B}(t,y,w) \eqdef  \left\{
 \begin{array}{rcl}
 & \widetilde{\mathbb{B}} \eqdef  ABw + A \widetilde{a_g} - A \widetilde{\mathbf{E}_g}, & \text{if}\;\; y\!\in \mathbb{H}_{-} \\[4pt]
 & \overline{\mathbb{B}} \eqdef  R\bar{A}\bar{B}Rw + R\bar{A} \overline{a}_g - R\bar{A} \overline{\mathbf{E}}_g,\quad & \text{if}\;\; y\!\in \mathbb{H}_{+}
 \end{array}\right.,
 \end{equation}

 \begin{equation}\label{temp 2}
 \mathbb{C}(t,y,w) \eqdef  \left\{
 \begin{array}{rcl}
 & \widetilde{\mathbb{C}} \eqdef  \widetilde{\bar K_gf}+(A^{-1}w)^T\widetilde{\mathbf{E}}_g + 2\sqrt{\mu(w)}(A^{-1}w)^T\widetilde{\mathbf{E}}_g, & \text{if}\;\; y\!\in \mathbb{H}_{-} \\[4pt]
 & \overline{\mathbb{C}} \eqdef  \overline{\bar{K}_gf}+(\bar A^{-1}Rw)^T\overline{\mathbf{E}}_g + 2\sqrt{\mu(w)}(\bar A^{-1}Rw)^T\overline{\mathbf{E}}_g, & \text{if}\;\; y\!\in \mathbb{H}_{+}
 \end{array}\right..
 \end{equation}

 \begin{remark}
 Thanks to our design of the form of transformation (\ref{Phi}) and extension (\ref{f-bar}), the transport operator of the equation remains {\em invariant} after change of variables, which is vital for our future analysis.
 \end{remark}
 It is also worth noting that the new second-order coefficient $\mathbb{A}$ preserves the positivity of $\sigma_{\!G}$, and thus the hypo-ellipticity of the equation, since $A$ and $R$ are non-degenerate. 

 \subsection{Weak Formulation of Extended Equations} \label{Well-definedness}
 After doing the extension, it is important to make sure that across the boundary the equation (\ref{f-bar-eq}) is satisfied by $\bar{f}$ in some proper sense (at least in the weak sense). That means $\bar{f}$ should satisfy the following weak formation of equation (\ref{f-bar-eq}) in the whole space:
 \begin{align*}
 & \iint_{\mathbb{R}^3\!\times\mathbb{R}^3} \left[(\bar{f}\varphi)(t)-(\bar{f}\varphi)(0) \right] dydw \\
 & = \int_0^t\!\iint_{\mathbb{R}^3\!\times\mathbb{R}^3} \Big\{\bar{f}  \big[(\partial_s+w\!\cdot\!\nabla_{\!y})\varphi-\mathbb{B}\cdot\nabla_{\!w}\varphi\big] + \mathbb{C}f\varphi - \nabla_{\!w}\bar{f}\cdot(\mathbb{A} \nabla_{\!w}\varphi) \Big\} dydwds.
 \end{align*}

 Normally a weak formation is obtained by multiplying the equation by some suitable test function $\varphi$ and then integrating by parts over the domain where the equation(s) are defined i.e., $(0,t)\times(\mathbb{H}_{-}\!\cup\mathbb{H}_{+})\times\mathbb{R}^3$. This process yields
 \begin{align*}
 & \iint_{\widetilde{\Omega}\times\mathbb{R}^3} \left[(\bar{f}\varphi)(t)-(\bar{f}\varphi)(0) \right] dydw \\
 & = \int_0^t\!\iint_{\widetilde{\Omega}\times\mathbb{R}^3} \Big\{\bar{f}  \big[(\partial_s+w\!\cdot\!\nabla_{\!y})\varphi-\mathbb{B}\cdot\nabla_{\!w}\varphi\big] + \mathbb{C}f\varphi - \nabla_{\!w}\bar{f}\cdot(\mathbb{A} \nabla_{\!w}\varphi) \Big\} dydwds \\
 &\quad -\int_0^t\!\!\int_{\tilde{\gamma}} \bar{f}\varphi  d\tilde{\gamma} ds.
 \end{align*}
 Here  $\widetilde{\Omega} \eqdef  \mathbb{H}_{-}\!\cup\mathbb{H}_{+}$, $\tilde{\gamma} \eqdef  \partial\widetilde{\Omega}\times\mathbb{R}^3 = (\partial\mathbb{H}_{-}\!\cup\partial\mathbb{H}_{+})\times\mathbb{R}^3$, and $d\tilde{\gamma} \eqdef  (w\!\cdot n_{y}) dS_{y}dw$.
 \begin{remark}
 	 	The only boundary-integral term $I_{\tilde{\gamma}} \eqdef  \int_0^t\!\int_{\tilde{\gamma}} \bar{f}\varphi  d\tilde{\gamma} ds$ above comes from integration by parts in $y$. Note that integration by parts in $w$ does not produce any boundary terms.
 \end{remark}

 Compared with the above definition, this is equivalent to saying that we have to be sure the boundary term vanishes:
 \begin{equation*}
 \int_{\tilde{\gamma}} \bar{f}\varphi  d\tilde{\gamma} = \left(\iint_{\partial\mathbb{H}_{-}\!\times\mathbb{R}^3} \!+ \iint_{\partial\mathbb{H}_{+}\!\times\mathbb{R}^3}\right) \bar{f}\varphi  (w\!\cdot n_{y}) dS_{y}dw = 0,
 \end{equation*}
 which is indeed true since
 \begin{equation*}
 \bar{f}(t ;y_1,y_2,0- ;w) = \bar{f}(t ;y_1,y_2,0+ ;w)
 \end{equation*}
 due to continuity of $\bar{f}$ across the boundary, while the normal vectors at same point of outer and inner boundary are of opposite directions
 \begin{equation*}
 n_{y}(y_3\!=\!0-)|_{\partial\mathbb{H}_{-}} = - n_{y}(y_3\!=\!0+)|_{\partial\mathbb{H}_{+}},
 \end{equation*}
 plus the coincidence of $y$-derivative term (transport operator) on two sides.

 Therefore, we can now conclude that $\bar{f}$ is a (weak) solution to the equation (\ref{f-bar-eq}) in the whole space.

 \subsection{Continuity of the Coefficients Across the Boundary} \label{Continuity-Coefficients}
 The $S^p$ estimate is based on a reformulation of the equation, which is of the form of a class of kinetic Fokker-Planck equations (also called hypoelliptic   or ultraparabolic   of Kolmogorov type) with rough coefficients:
 \begin{equation*}
  \partial_t f + v\cdot\nabla_{\!x} f = \nabla_v\cdot(\mathbf{A} \nabla_{\!v} f )+\mathbf{B}\cdot\nabla_{\!v}f + \mathbf{C}f  .
 \end{equation*}

 The properties of the coefficients used in \cite{Kim.Guo.Hwang2020} for the estimates to hold are as follows: if $\|g\|_{\infty}$ is sufficiently small, \\[3pt]
 $\mathbf{A}(t,x,v)\eqdef \sigma_{\!G}:\ 3\!\times\!3$ non-negative matrix, but \textit{not uniformly} elliptic, $0<(1\!+\!|v|)^{-3}I \lesssim \mathbf{A}(v) \lesssim (1\!+\!|v|)^{-1}I$\; (Lemma 2.4 in \cite{Kim.Guo.Hwang2020}); continuous.  \\[3pt]
 $\mathbf{B}(t,x,v)\eqdef a_g-\mathbf{E}_g$ : \;essentially bounded $3d$-vector, $\|\mathbf{B}[g]\|_{\infty}\lesssim \|g\|_{\infty}\ll 1$. \\[3pt]

 The ellipticity and boundedness of the new coefficients after extension are easy to check (look back the transformed equations in Section \ref{Transformed-Equation}).
Thus we are left with one main task -- checking the continuity of the second-order coefficient $\mathbb{A}$ across the boundary, which is necessary only for the $S^p$ estimate (see Theorem 7.2 and Lemma 7.5 in \cite{Kim.Guo.Hwang2020}).
 A direct computation on (\ref{coeff-A}) gives
 \begin{align*}
 \widetilde{\sigma}_{\!G} &= \left|\det\!\left(A^{-1}\right)\right|\cdot \widetilde{\psi} \ast \left(\tilde{\mu}\!+\!\tilde{\mu}^{1\!/2}\tilde{g}\right),  \\[2pt]
 \psi(v) &= |v|^{-1}\!\cdot\! I - |v|^{-3}\!\cdot\!\left(vv^T\right), \\[2pt]
 \widetilde{\psi}(y,w) &= (w^T\!A^{-T}\!A^{-1}w)^{-1\!/2}\!\cdot\! I - (w^T\!A^{-T}\!A^{-1}w)^{-3/2}\!\cdot\!\left[A^{-1}ww^TA^{-T}\right], \\
 &= (w^T\! Cw)^{-1\!/2}\!\cdot\! I - (w^T\! Cw)^{-3/2}\!\cdot\!\left[A^{-1}ww^TA^{-T}\right],\ \text{and}\  \\[2pt]
 \tilde{\mu}(y,w) &= e^{-w^T\! Cw},\quad \bar{\mu}(y,w) = e^{-w^T\!R \bar{C}Rw}.
 \end{align*}
Then we have (see also \cite{Dong.Guo.Yastrzhembskiy2021})
 \begin{align*}
 \widetilde{\mathbb{A}} &= A \widetilde{\sigma}_{\!G}A^T \\
 &= \left|\det\!\left(A^{-1}\right)\right|\!\cdot \Big\{(w^T\! Cw)^{-1\!/2}\!\cdot\! \left[AA^T\right] - (w^T\! Cw)^{-3/2}\!\cdot\!\left[ww^T\right]\!\Big\}
 \ast \left(\tilde{\mu}\!+\!\tilde{\mu}^{1\!/2}\tilde{g}\right), \\[3pt]
 \overline{\mathbb{A}} &= R\bar{A} \overline{\sigma}_{\!G}\bar{A}^T\!R \\
 &= \left|\det\!\left(\bar{A}^{-1}\right)\right|\!\cdot \Big\{(w^T\!R \bar{C}Rw)^{-1\!/2}\!\cdot\! \left[R\bar{A}\bar{A}^T\!R\right] - (w^T\!R \bar{C}Rw)^{-3/2}\!\cdot\!\left[w w^T\right]\!\Big\} \\
 &\qquad\qquad\qquad \ast \left(\bar{\mu}\!+\!\bar{\mu}^{1\!/2}\bar{g}\right).
 \end{align*}
Based on our definition of the extension, since we have continuity of $\bar A$ and $\bar C$, and decay of $G$ (i.e., vanishing at infinity), 
the continuity of $\overline{\mathbb{A}}$ in $(x,v)$  naturally follows.

 \subsection{Conclusion}
 By designing a suitable ``boundary-flattening'' transformation, we are able to extend our solutions to a neighborhood $\Omega$ for the specular reflection boundary condition case while preserving the form of transformed equation to the largest extent.



In order to utilize $S^p$ theory, we only need to make sure the transformed/extended equation of
\begin{align*}
\partial_t\bar f+w\cdot\nabla_{y}\bar f-\bar\sigma_G^{ij}\partial_{w_i'w_j'}\bar f=\bar S,
\end{align*}
where
\begin{align}
            \label{temp 3}
\bar S&=\partial_{w_i'}\bar\sigma_G^{ij}\partial_{w_j'}\bar f+\mathbb{B}\cdot\nabla_{w}\bar f+\mathbb{C}\bar f
\end{align}
for $\mathbb{B}$ and $\mathbb{C}$ defined in \eqref{temp 1} and \eqref{temp 2}, also satisfies
\begin{itemize}
\item
(H.1): Ellipticity (eigenvalue bounds) of $\mathbb{A} = \sigma_G$: remain equivalent.
\item
(H.2): Structure of transport operator: the structure of operator is invariant.
\item
(H.3): $\mathbb{A}=[\bar\sigma_G^{ij}]$ is continuous across the boundary.
\end{itemize}
This makes it available to implement $S^p$ techniques in this extended domain.
Since the electric field does not change the structure/property of the ultra-parabolic operator, and the construction of flattening/reflection-transformation only
depends on the geometry of domain $\Omega$, our extension trick should preserve these properties above.

\begin{remark}
For the $S^p$ estimates, we need to estimate the $L^p$ norm of \eqref{temp 3} term by term (in the transformed variable $(t,y,w)$). Compared to using the original variable $(t,x,v)$, for most terms there is no essential difference. The only exception is that we need to raise one more power of weight in estimating $\mathbb{B}\cdot\nabla_{w}\bar f$ due to the extra $w$ factor coming from the transformation.
\end{remark}


\begin{remark}
Denote $\tilde\Omega$ to be an extension of $\Omega$:
\begin{align*}
    \tilde\Omega=\left\{x\in\r^3: \text{dist}(x,\Omega)<\delta\right\},
\end{align*}
and denote $\widetilde{\p\Omega}$ to be an extension of $\p\Omega$:
\begin{align*}
    \widetilde{\p\Omega}=\left\{x\in\r^3: \text{dist}(x,\p\Omega)<\delta\right\},
\end{align*}
where $\delta>0$ is sufficiently small. Also, define the $\delta$-interior of $\Omega$:
\begin{align*}
    \hat\Omega=\left\{x\in\Omega: \text{dist}(x,\Omega)<\delta/2\right\}.
\end{align*}
Consider a partition of unity of $\tilde\Omega$: $\{\chi_k(x)\}_{k=0}^n$
\begin{align*}
    \sum_{k=0}^n\chi_k(x)=1,
\end{align*}
satisfying
\begin{align*}
    \supp{\chi_0}\subset\hat\Omega,\quad \bigcup_{k=1}^n\supp(\chi_k)\subset\widetilde{\p\Omega},
\end{align*}
where in each $\supp{\chi_k}$, we may locally flatten the boundary as described in this section. Denote a $C_c^{\infty}$ cutoff function
\begin{align*}
    \tilde\chi(x)=\left\{\begin{array}{ll}1&\ \ \text{if}\ \ \text{dist}(x,\Omega)<\delta/2,\\
    0&\ \ \text{if}\ \ \text{dist}(x,\Omega)>\delta.\end{array}\right.
\end{align*}
It is also natural to write $\chi_k$ and $\tilde\chi$ in the $(y,w)$ chart when we consider the near-boundary region. In order to study the extended $f$, it suffices to consider the following:
\begin{align}
    f_{I}(t,x,v)=&\chi_0(x)f(t,x,v),\label{tt 111}\\
    f_k(t,y,w)=&\tilde\chi(\vec{\psi}^{-1}(y))
    \chi_k(\vec{\psi}^{-1}(y))\bar f_k(t,y,w)\ \ \text{for}\ \ 1\leq k\leq n,\label{tt 222}
\end{align}
where $\bar f_k$ is the corresponding locally extended solution in $\supp{\chi_k}$. In particular, $f_I$ satisfies
\begin{equation}\label{tt 333}
\partial_t f_I + v\cdot\nabla_{\!x}f_I - \sigma_{\!G}^{ij}\partial_{v_i v_j} f_I \,=\,  S_0,
\end{equation}
where
\begin{equation*}
S_0=\partial_{v_i}\sigma_G^{ij}\chi_0\partial_{v_j} f
+\chi_0\mathbb{B}\cdot\nabla_{v} f+\chi_0\mathbb{C} f-v\cdot\nabla_x\chi_0 f,
\end{equation*}
and $f_k$ satisfies
\begin{equation}\label{tt 444}
\partial_t f_k+w\cdot\nabla_{y} f_k-\bar\sigma_G^{ij}\partial_{w_i'w_j'} f_k \,=\,  S_k,
\end{equation}
where
\begin{equation*}
 S_k=\partial_{w_i}\bar\sigma_G^{ij}\chi_k\partial_{w_j} f+\chi_k\mathbb{B}\cdot\nabla_{w} f+\chi_k\mathbb{C} f-w\cdot\nabla_y(\tilde\chi(\vec{\psi}^{-1}(y))
    \chi_k(\vec{\psi}^{-1}(y))) \bar f_k.
\end{equation*}
\end{remark}

\section{\texorpdfstring{$S^p$}{Sp} Estimates, H\"{o}lder Continuity, and weighted \texorpdfstring{$L^\infty$}{} estimates} \label{Sec:S^p-est}

\subsection{$S^p_{ }$ Bound} \label{SubSec:S^p-bound}

In this section we prove the $S^p_{ }$ bound in (\ref{Bootstrap-est:finteness-S^p}).

\begin{proposition} \label{Prop:S^p-bound}
With the notation and hypothesis in Proposition \ref{Prop:Bootstrap-est}, we have
\begin{equation*}
\|f\|_{S^{p}\left((0,T)\times\Omega\times\R^3\right)} \,\lesssim\, \varepsilon_1^{\,s}
\end{equation*}
for some $s>1$.
\end{proposition}

\begin{proof}
Denote $\Pi=(0,T)\times\R^3\times\R^3$. Based on the discussion in Section\;\ref{Sec:Boundary-Extension},
(after extension) $\mathsf{A}(\mathbf{z}) = \sigma_{\!G}(t,x,v)$ or $\bar\sigma_G(t,y,w)$ satisfies the
conditions (H.1)--(H.3).

We recall the
localization \eqref{tt 111}, \eqref{tt 222}, \eqref{tt 333}, \eqref{tt 444}, and apply the $S^p$ estimate for each $f_k$.
We may apply the similar argument as in the proof of \cite[Proposition 5.9, Theorem 1.11]{Dong.Guo.Yastrzhembskiy2021} to obtain
\begin{align}
\big\|f_k\big\|_{S^{p}(\Pi)}^p
&\,\ls\,  \|\br{v}^{\vartheta}S_k\|_{L^{p}(\Pi)}^p + \|\br{v}^{\vartheta}f_k\|_{L^p(\Pi)}^p+\nm{f_0}_{O}^p \label{Yf-L^p}
\end{align}
and a similar estimate for $f_I$,
where
\begin{align*}
    \nm{f_0}_{O}&:=\nm{\br{v}^{\vartheta}v\cdot\nabla_xf_0}_{L^p(\Omega\times\r^3)}
    +\nm{\br{v}^{\vartheta}D_{vv}f_0}_{L^p(\Omega\times\r^3)}+\nm{\br{v}^{\vartheta}\gamma f_0}_{L^p_{\gamma_-}}\\
    &\quad +\|f_0\|_{L^p(\Omega\times\r^3)}\ls\varepsilon_0.
\end{align*}
Here in view of \eqref{eq3.1.0}, by taking $\varepsilon_1$ in \eqref{Bootstrap-assp:finteness-S^p} sufficiently small and the $S^p$ embedding into H\"older spaces, we may choose $R_0$ and the implicit constant in \eqref{Yf-L^p} to be uniform.
Also $\vartheta(<\vartheta_1)$ in Proposition \ref{Prop:Bootstrap-est} is a sufficiently large constant depending on the constant $\kappa(p)$ in \eqref{eq3.1.0}.

For $1\leq k\leq n$, the above estimate is  written in $(y,w)$ chart, so we further need to pull back to $(x,v)$ chart. This will not change the form of estimates based on the extension in previous section.

For notational simplicity, from now on, we denote $f$ for $f_I$ and each $f_k$,  and only use $(x,v)$ to represent the variables. The final estimate of $f$ will rely on a summation over $k$.
The similar convention also applies to other quantities.

In this fashion, we have the bound
\begin{align}
\big\|f\big\|_{S^{p}(\Pi)}^p
&\,\ls\,  \|\br{v}^{\vartheta}\mathcal{S}\|_{L^{p}(\Pi)}^p + \|\br{v}^{\vartheta}f\|_{L^p(\Pi)}^p+\nm{f_0}_{O}^p  , \label{Yf-L^ppp}
\end{align}
where
\begin{equation*}
\mathcal{S} :=\, \partial_{v_i}\sigma_{\!G}^{ij}\partial_{v_j} f
+ \big\{a_g - \mathbf{E}_g\big\}\cdot\nabla_{\!v}f
+ \Big\{ \bar{K}_{\!g} f + \big(v\cdot\mathbf{E}_g\big)f + 2\sqrt{\mu}\,v\cdot\mathbf{E}_f \Big\} .
\end{equation*}
Therefore, by (\ref{Yf-L^ppp})
\begin{equation} \label{S^p-local}
\begin{split}
\|f\|_{S^{p}_{ }(\Pi)}^p
&\,\ls\,  \varepsilon_0^p+\|\br{v}^{\vartheta}\mathcal{S}\|_{L^{p}(\Pi)}^p + \|\br{v}^{\vartheta}f\|_{L^p(\Pi)}^p  \\[3pt]
&\,\ls\,
\varepsilon_0^p+\big\|\br{v}^{\vartheta}\partial_{v_i}\sigma_{\!G}^{ij}\partial_{v_j} f \big\|_{L^{p}(\Pi)}^p
+ \big\|\br{v}^{\vartheta}\big(a_g - \mathbf{E}_g\big)\cdot\nabla_{\!v}f \big\|_{L^{p}(\Pi)}^p \\
&\qquad + \big\|\br{v}^{\vartheta}\bar{K}_{\!g} f \big\|_{L^{p}(\Pi)}^p
+ \big\|\br{v}^{\vartheta}\big(v\cdot\mathbf{E}_g\big)f \big\|_{L^{p}(\Pi)}^p\\
&\qquad+ \big\|\br{v}^{\vartheta}2\sqrt{\mu}\,v\cdot\mathbf{E}_f \big\|_{L^{p}(\Pi)}^p
+ \|\br{v}^{\vartheta}f\|_{L^p(\Pi)}^p  .
\end{split}
\end{equation}

Since $\big\|\partial_{v_i}\sigma_{\!G}^{ij}\big\|_{\infty} \lesssim 1$, $\|a_g\|_{\infty} \lesssim \|g\|_{\infty}$ (Corollary\;\ref{Cor:sigma_G-a_g}), and $\|\mathbf{E}_g\|_{\infty} \lesssim \|g\|_{\infty}$ (Lemma\;\ref{Lem:E_f-L^p}), we have
\begin{equation} \label{h-L^p-2}
\begin{split}
&
\big\|\br{v}^{\vartheta}\partial_{v_i}\sigma_{\!G}^{ij}\partial_{v_j} f \big\|_{L^{p}(\Pi)}^p
+ \big\|\br{v}^{\vartheta}\big(a_g - \mathbf{E}_g\big)\cdot\nabla_{\!v}f \big\|_{L^{p}(\Pi)}^p  \\
\,\lesssim\;&  \big\|\br{v}^{\vartheta}D_vf\big\|_{L^{p}(\Pi)}^p \\
\,\lesssim\;& \varepsilon\, \big\|\br{v}^{\vartheta}D^2_{vv}f\big\|_{L^{p}(\Pi)}^p
+ \,\varepsilon^{-1}\, \|\br{v}^{\vartheta}f\|_{L^{p}(\Pi)}^p .
\end{split}
\end{equation}

Similarly, by Lemma\;\ref{Lem:K_g-L^p}\, as well as Lemma \ref{Lem:interpolation},
\begin{equation} \label{h-L^p-3}
\begin{split}
&  \big\|\br{v}^{\vartheta}\bar{K}_{\!g} f \big\|_{L^{p}(\Pi)}^p \\
\,\lesssim\;& \|\br{v}^{\vartheta}f\|_{L^{p}\left(\Pi\right)}^p + \big\|\br{v}^{\vartheta}D_vf\big\|_{L^{p}\left(\Pi\right)}^p  \\
\,\lesssim\;&\, \varepsilon \big\|\br{v}^{\vartheta}D^2_{vv}f\big\|_{L^{p}\left(\Pi\right)}^p
\,+\, \varepsilon^{-1} \|\br{v}^{\vartheta}f\|_{L^{p}\left(\Pi\right)}^p.
\end{split}
\end{equation}
Using $\|\mathbf{E}_g\|_{\infty} \lesssim \|g\|_{\infty}$ and $\|\mathbf{E}_f\|_{L^p} \lesssim \|f\|_{L^p}$ (see Lemma\;\ref{Lem:E_f-L^p}), we also have
\begin{equation} \label{h-L^p-4}
\begin{split}
  \big\|\br{v}^{\vartheta}\big(v\cdot\mathbf{E}_g\big)f \big\|_{L^{p}(\Pi)}^p
&\,\lesssim\;
\|g\|_{\infty}^p \big\|\langle v\rangle^{\vartheta+1}f\big\|_{L^p(\Pi)}^p \\
&\,\lesssim\; \varepsilon_1^{\,p}\, \big\|\langle v\rangle^{\vartheta+1}f\big\|_{L^{p}\left(\Pi\right)}^p,
\end{split}
\end{equation}
and
\begin{equation} \label{h-L^p-5}
\begin{split}
 \big\|\br{v}^{\vartheta}2\sqrt{\mu}\,v\cdot\mathbf{E}_f \big\|_{L^{p}(\Pi)}^p
&\,\lesssim\; \big\|\mathbf{E}_f\big\|_{L^{p}\left((0,T)\times\R^3\right)}^p \cdot
\big|\,2\langle v\rangle^{\vartheta+1}\! \sqrt{\mu}\,\big|_{L^p(\R^3)}^p \\
&\,\lesssim\; \|f\|_{L^{p}\left(\Pi\right)}^p .
\end{split}
\end{equation}

Combining (\ref{S^p-local})\,--\,(\ref{h-L^p-5}), we obtain
\begin{equation} \label{S^p-global}
\begin{split}
\|f\|_{S^{p}_{ }\left(\Pi\right)}^p
&\,\lesssim\; \varepsilon^{-1} \big\|\langle v\rangle^{\vartheta}\!f\big\|_{L^{p}\left(\Pi\right)}^p
+\, o\big(\varepsilon_1^{\,p}\big) .
\end{split}
\end{equation}

Finally, it remains to bound the weighted $L^p$ norm: using the standard interpolation (for $2<p<\infty$, by H\"{o}lder inequality), followed by the strong (almost-exponential) $L^2$ decay (Theorem\;\ref{Thm:L^2-decay}, which allows us to utilize the ``stronger'' initial condition (\ref{smallness-assumption}) with sufficiently large weight, as well as the bootstrap assumption (\ref{Bootstrap-assp:smallness-decay-L^infty})\,),
we see that
\begin{equation} \label{weighted-L^p-bound}
\begin{split}
\big\|\langle v\rangle^{\vartheta}\!f\big\|_{L^{p}\left((0,T)\times\Omega\times\R^3\right)}^p
\,=\, \int_0^T \!\big\|f(t)\big\|_{p,\vartheta}^p \,\dd t
&\,\leq\; \int_0^T \!\big\|f(t)\big\|_{2,\vartheta}^2 \big\|f(t)\big\|_{\infty,\vartheta}^{p-2} \,\dd t \\
&\,\lesssim\; \varepsilon_1^{\,p-2} \int_0^T \langle t\rangle^{-2k} \|f_0\|_{2,\vartheta\!+k}^2 \,\dd t \\
&\,\lesssim\; \varepsilon_1^{\,p-2} \varepsilon_0^{\,2} \int_0^T \langle t\rangle^{-2k} \,\dd t \\
&\,\lesssim\; \varepsilon_1^{\,p-2+2q} \,=\, \varepsilon_1^{\,p\big(1+\frac{2q-2}{p}\big)} .
\end{split}
\end{equation}
Note that for any $k>\frac{1}{2}$, the time integral above is uniformly bounded with respect to $T$.
Hence, upon letting $s' = 1+\frac{2q-2}{p} >1$ and $\varepsilon' = \varepsilon_1^{\,p\big(\frac{s'-1}{2}\big)} = \varepsilon_1^{\,q-1}$, we deduce from (\ref{S^p-global}) and (\ref{weighted-L^p-bound}) that
$\|f\|_{S^{p}_{ }\left((0,T)\times\Omega\times\R^3\right)}^p
\,\lesssim\, \varepsilon_1^{\,sp}$
with $s=\frac{s'+1}{2} = 1+\frac{q-1}{p} >1$,
which completes the proof.
\end{proof}

\subsection{H\"{o}lder Continuity and $L^\infty$ bound}

As a direct corollary of the $S^p$ bound (Proportion\;\ref{Prop:S^p-bound}) and the embedding theorem (see also \cite[Theorem 1.11, Proposition 5.9]{Dong.Guo.Yastrzhembskiy2021}), we deduce the following uniform H\"{o}lder continuity results, which imply \eqref{eq12.33}, as well as the bound for $\|\nabla_{\!v} f\|_{L^\infty}$ in (\ref{Bootstrap-est:finteness-Dv}), a crucial element to ensure the uniqueness (see Section\;\ref{Sec:Main-Thm-Pf}\, for the proof).

\begin{corollary} \label{Cor:Holder-continuity}
With the notation and hypothesis in Proposition \ref{Prop:Bootstrap-est}, we have that
\begin{enumerate}
  \item[1)] if $p>7$, then with $\alpha_1 = \min\left\{1,\, 2-\frac{14}{p} \right\} \in (0,1]$,
  \begin{equation} \label{C^alpha-bound-1}
    \|f(t)\|_{C^{\,0,\alpha_1}_{ }\left(\Omega\times\R^3\right)} \,\lesssim\, \varepsilon_1^{\,s} \;;
  \end{equation}
  \item[2)] if $p>14$, then with $\alpha_2 = 1-\frac{14}{p} \,\in (0,1)$,
  \begin{equation} \label{C^alpha-bound-2}
    \|D_v f(t)\|_{C^{\,0,\alpha_2}_{ }\left(\Omega\times\R^3\right)} \,\lesssim\, \varepsilon_1^{\,s} \;,
  \end{equation}
\end{enumerate}
where $s>1$ can be chosen the same as in Proportion\;\ref{Prop:S^p-bound}. In particular, we have
\begin{equation*}
\|D_v f\|_{L^\infty\left((0,T)\times\Omega\times\R^3\right)} \,\lesssim\, \varepsilon_1^{\,s}.
\end{equation*}
\end{corollary}

We only note that the last assertion follows because if we take $t=\tau,\, x=\xi,\, v_j = \nu_j \;(j\neq i)$, then $\hat{\mathsf{d}}(\hat{\mathbf{z}},\hat{\mathbf{w}}) = |v_i - \nu_i|$, and so
\begin{equation*}
\big|\partial_{v_i} f\big| \,=\, \left|\lim_{\nu_i\rightarrow v_i} \frac{f(t,x,v) - f(\cdots, \nu_i, \cdots)}{v_i - \nu_i}\right|
\,\leq\, \sup_{\mathbf{z}\neq\mathbf{w}} \frac{|f(t,\hat{\mathbf{z}})-f(t,\hat{\mathbf{w}})|}{\big|\hat{\mathsf{d}}(\hat{\mathbf{z}},\hat{\mathbf{w}})\big|} \,=\, |f|_{C^{\,0,1}_{ }}.
\end{equation*}

%
\label{Sec:Pointwise-est-decay}

As a byproduct, we have the following result.

\begin{corollary}
With the notation and hypothesis in Proposition \ref{Prop:Bootstrap-est}, we have
\begin{equation*}
\|\mathbf{E}_f(t)\|_{C^{1,\alpha}\left(\Omega\right)}
\;\lesssim\; \varepsilon_1^{\,s}
\end{equation*}
for some $s>1$, and where $\alpha$ takes the same values as $\alpha_1$ in \eqref{C^alpha-bound-1}.
\end{corollary}

\begin{remark}
This regularity result actually includes the bound of $\|\mathbf{E}_f\|_{W^{1,\infty}_{t,x}}$ for $f$ being a solution to the coupled system, which is an improvement compared to the $L^\infty$\;estimate for the Poisson equation in Lemma\;\ref{Lem:E_f-L^p}.
\end{remark}

\begin{proof}

Since
$$
\mathbf{E}_f := -\nabla_{\!x}\phi_f \,=\, \nabla_{\!x}\, \Delta_{x}^{-1}\! \int_{\R^3}\!\sqrt{\mu}\,f\,\dd v =: \nabla_{\!x}\, \Delta_{x}^{-1} \rho[f],
$$
we apply the standard Schauder estimates for elliptic equations (cf. \cite[Section\;4.1]{Krylov1996}) to get
\begin{equation*}
\begin{split}
\big\|\mathbf{E}_f(t)\big\|_{C^{1,\alpha}_{x}} \,\lesssim\, \big\|\phi_f(t)\big\|_{C^{2,\alpha}_{x}} &\,\lesssim\, \big\|\rho[f](t)\big\|_{C^{0,\alpha}_{x}} \,=\, \left\|\int_{\R^3}\! \sqrt{\mu(v)}\,f(t,\cdot,v)\,\dd v \right\|_{C^{0,\alpha}_{x}} \\
&\,\leq\, \int_{\R^3}\! \sqrt{\mu(v)}\, \big\|f(t,\cdot,v)\big\|_{C^{0,\alpha}_{x}}\, \dd v \\
&\,\leq\, \|f\|_{C^{\,0,\alpha}_{ }} \cdot\! \int_{\R^3}\!\! \sqrt{\mu(v)}\,\dd v
\,\lesssim\, \|f\|_{C^{\,0,\alpha}_{ }} 
\end{split}
\end{equation*}
for any $t\geq 0$.
\end{proof}

\subsection{Weighted $L^\infty$ Decay (completion of the proof of Proposition \ref{Prop:Bootstrap-est})}

So far we have gained control of the energy ($L^2$\;norm) of our solutions (Theorem\;\ref{Thm:L^2-decay}) and have known that they are H\"{o}lder continuous (Corollary\;\ref{Cor:Holder-continuity}).


We further prove the weighted $L^\infty$\;decay bound in (\ref{Bootstrap-est:smallness-decay-L^infty}) based on our known weighted $L^2$\;decay and H\"{o}lder continuity results, and thus complete the proof of Proposition \ref{Prop:Bootstrap-est}.
The heuristics is that a continuous function with small $L^2$\;norm should have magnitude of comparable size.
In fact, it suffices to show the following.

\begin{proposition}
For any $\delta>0$, there exists $\varepsilon=\varepsilon(\delta)>0$ such that if $\langle t\rangle^{k_1}\|f(t)\|_{2,\vartheta_1} < \varepsilon$ for some $k_1, \vartheta_1 \geq 0$, and $|f(t)|_{C^{0,\alpha}_{ }} \leq M$ with some $\alpha\in (0,1)$ and $M>0$,
then $\langle t\rangle^{k_2}\|f(t)\|_{\infty,\vartheta_2} < \delta$, where $\left(1+\frac{6}{\alpha}\right)k_2 \leq k_1$ and $\left(1+\frac{6}{\alpha}\right)\vartheta_2 \leq \vartheta_1$.
\end{proposition}

\begin{proof}
We use the method of contradiction.
Assuming the contrary, there exists $\delta_0>0$ such that no matter how small $\varepsilon>0$ is, we always have $\langle t\rangle^{k_2}\|f(t)\|_{\infty,\vartheta_2} \,\geq\, \delta_0$. In other words, without loss of generality, there is a point $\mathbf{z}=(t,x,v)$ where
$\langle t\rangle^{k_2} \langle v\rangle^{\vartheta_2} f(t,x,v) \,\geq\, \delta_0 $, or, equivalently,
\begin{equation} \label{pointwise-est-1}
f(t,x,v) \,\geq\, \frac{\delta_0}{\langle t\rangle^{k_2} \langle v\rangle^{\vartheta_2}} \;.
\end{equation}

From $|f(t)|_{C^{0,\alpha}_{ }} \leq M$ we have for any point $\mathbf{w}=(t,\xi,\nu)=(t,\hat{\mathbf{w}})$,
\begin{equation*}
\big|f(t,\hat{\mathbf{z}})-f(t,\hat{\mathbf{w}})\big| \,\leq\, M \big|\hat{\mathsf{d}}(\hat{\mathbf{z}},\hat{\mathbf{w}})\big|^{\alpha} .
\end{equation*}
For fixed $t$, consider a $\hat{\mathsf{d}}$-disk
$$\mathfrak{D}^{t}_r(\hat{\mathbf{z}}) := \left\{\hat{\mathbf{w}}\in\R_x^3\times\R_v^3:\, \hat{\mathsf{d}}\,(\hat{\mathbf{z}},\hat{\mathbf{w}}) <r \right\} $$
in the phase-space centered at $\hat{\mathbf{z}}=(x,v)$ with radius $r=\left(\frac{\delta_0}{2M\langle t\rangle^{k_2} \langle v\rangle^{\vartheta_2}} \right)^{\frac{1}{\alpha}}$.
Then for every point $(t,\xi,\nu)$ with $\hat{\mathbf{w}}=(\xi,\nu) \in \mathfrak{D}^{t}_r(\hat{\mathbf{z}})$, we have
\begin{equation*}
\big|f(t,x,v)-f(t,\xi,\nu)\big| \,\leq\, \frac{\delta_0}{2\,\langle t\rangle^{k_2} \langle v\rangle^{\vartheta_2}} \;,
\end{equation*}
which, together with (\ref{pointwise-est-1}), implies
\begin{equation} \label{pointwise-est-2}
f(t,\xi,\nu) \,\geq\, \frac{\delta_0}{2\,\langle t\rangle^{k_2} \langle v\rangle^{\vartheta_2}} \;.
\end{equation}
Also, it is easy to check that $\langle v\rangle \sim \langle \nu\rangle$ in $\mathfrak{D}^{t}_r(\hat{\mathbf{z}})$ up to an error of $O(\delta_0)$.

Now we estimate the weighted $L^2$\;norm with (given) time-growth. By using \eqref{pointwise-est-2},
\begin{equation*}
\begin{split}
\langle t\rangle^{k_1}\|f(t)\|_{2,\vartheta_1}
&\,\geq\, \langle t\rangle^{k_1}\! \left(\iint_{\mathfrak{D}^{t}_r(\hat{\mathbf{z}})}\langle\nu\rangle^{2\vartheta_1} \big|f(t,\xi,\nu)\big|^2 \dd\nu\dd\xi \right)^{\frac{1}{2}} \\
&\,\gtrsim\, \langle t\rangle^{k_1} \langle v\rangle^{\vartheta_1} \frac{\delta_0}{\langle t\rangle^{k_2}\langle v\rangle^{\vartheta_2}}\cdot \big|\mathfrak{D}^{t}_r(\hat{\mathbf{z}})\big|^{\frac{1}{2}}
\\
&\,\simeq\, \langle t\rangle^{k_1} \langle v\rangle^{\vartheta_1} \frac{\delta_0}{\langle t\rangle^{k_2}\langle v\rangle^{\vartheta_2}} \left(\frac{\delta_0}{\langle t\rangle^{k_2} \langle v\rangle^{\vartheta_2}} \right)^{\frac{6}{\alpha}} \\
&\,=\, \delta_0^{\,1+\frac{6}{\alpha}} \,\langle t\rangle^{k_1-\left(1+\frac{6}{\alpha}\right)k_2}
\,\langle v\rangle^{\vartheta_1-\left(1+\frac{6}{\alpha}\right)\vartheta_2} \;.
\end{split}
\end{equation*}
Letting $k_1 - \left(1+\frac{6}{\alpha}\right)k_2 \,\geq\, 0$ and $\vartheta_1 - \left(1+\frac{6}{\alpha}\right)\vartheta_2 \,\geq\, 0$, we obtain
\begin{equation*}
\langle t\rangle^{k_1}\|f(t)\|_{2,\vartheta_1} \,\gtrsim\, \delta_0^{\,1+\frac{6}{\alpha}} .
\end{equation*}
Note that this lower bound is a fixed positive number.
However, on the other hand, our assumption is that $\langle t\rangle^{k_1}\|f(t)\|_{2,\vartheta_1} < \varepsilon$ for arbitrarily small $\varepsilon>0$.
We thus get a contradiction.
\end{proof}



In a similar manner as Step\;5 in the proof of Theorem\;\ref{Thm:L^2-decay} and using the same notation,
we take $\partial_t$ derivative of the equation \eqref{Eq:ultraparabolic-VL_f} (treated as nonlinear equation), and arrange the resulting equation in the desired ultraparabolic form for $\dot{f}$.
Then we may redo everything for $\dot{f}$ with apparent modifications.
Eventually, we arrive at the bounds (\ref{Bootstrap-est:smallness-decay-L^infty}), (\ref{Bootstrap-est:finteness-S^p}), and (\ref{Bootstrap-est:finteness-Dv}) for $\dot{f}$.

\section{Proof of the Main Theorem: Global Well-posedness} \label{Sec:Main-Thm-Pf}

In this final section, we will prove the global well-posedness of solutions stated in Theorem\;\ref{Thm:Main-thm}, including global existence, uniqueness, and the bounds (\ref{energy-bound})\,--\,(\ref{Dv-L^infty}) global in time.
Additionally, by a similar argument as \cite{Kim.Guo.Hwang2020} and \cite{Guo.Hwang.Jang.Ouyang2020}, we can show the non-negativity of the density-distribution function $F$.

\subsection{Global Regularity}
We start with the construction of global solutions to the Vlasov-Poisson-Landau system \eqref{Eq:Vlasov-Landau_f}-\eqref{Eq:Poisson_f}.

\smallskip
{\it \underline{Step\;1}. Local Existence}.
The proof of this result is an obvious modification of the argument in \cite[Sections\;2\,--\,5]{Guo.Hwang.Jang.Ouyang2020} and \cite{Dong.Guo.Yastrzhembskiy2021} combined with the standard approximation technique. Since this is not the major emphasis of this paper, we will simply record the result.

\begin{theorem}[Local well-posedness]\label{local-wellposedness}
There exists a sufficiently small constant $\varepsilon_0\ll 1$ such that
for some large velocity-weight exponent $\vartheta_0\in\R$, the initial data $f_0(x,v):\Omega\times\R^3 \rightarrow\R$ satisfies the smallness assumption
\begin{equation} \label{smallness-assumption-local}
\|f_0\|_{\infty,\vartheta_0}
\,+\, \|\nabla_{\!x} f_0\|_{\infty,\vartheta_0} \,+\,\|\nabla_{\!v}f_0\|_{\infty,\vartheta_0}
\,+\,\|\nabla_{\!v}^2f_0\|_{\infty,\vartheta_0}
\,\leq\, \varepsilon_0 .
\end{equation}
Let $F_0(x,v) = \mu + \sqrt{\mu}\,f_0(x,v) \geq 0$ and has the same mass as the Maxwellian $\mu$.
Then we have the following conclusions:
\begin{itemize}
  \item {\rm (Existence \& Uniqueness).}
  There exists a unique solution $f(t,x,v)$ on $[0,1]\times\Omega\times\r^3$ to the Vlasov-Poisson-Landau system \eqref{Eq:Vlasov-Landau_f}-\eqref{Eq:Poisson_f} for perturbation with the specular-reflection boundary condition \eqref{Specular-BC_f} and the conservation laws \eqref{mass-conservation_f}-\eqref{energy-conservation_f}.
  Also, $F(t,x,v) = \mu + \sqrt{\mu}\,f(t,x,v) \geq 0$ satisfies the system \eqref{Eq:Vlasov-Landau-model}-\eqref{Eq:Poisson-model} with \eqref{Specular-BC_F}.
  \vspace{2pt}
  \item {\rm (Energy estimates).} Moreover, the solution $f(t,x,v)$ satisfies the uniform weighted energy bounds, for $\vartheta < \vartheta_0 - \frac{3}{2}$,
  \begin{equation*} 
  \sup_{t\in [0,1]} \mathcal{E}_{\vartheta}[f(t)] \,\leq\, C_{\vartheta}\, \mathcal{E}_{\vartheta}[f(0)] \,\lesssim\, \varepsilon_0^{\,2} .
  \end{equation*}
  \vspace{2pt}
  \item {\rm ($L^\infty$ bounds).} There is $\vartheta'>0$ such that when $\vartheta  \leq \vartheta_0-\vartheta'$, the weighted pointwise bounds hold:
  \begin{equation*} 
  \|f(t)\|_{\infty,\vartheta} + \|\mathbf{E}_f(t)\|_{\infty} \,\lesssim\, \varepsilon_0 
  \end{equation*}
  for all $t\in[0,1]$.
  \vspace{2pt}
  \item {\rm ($S^{p}_{ }$ bounds).}
  For the $S^{p}_{ }$ norm defined in \eqref{S^p-norm}, we have
  \begin{equation*}
  \|f\|_{S^{p}_{ }\left((0,1)\times\Omega\times\R^3\right)}
  \,\lesssim\, \varepsilon_0 .
  \end{equation*}
  \item {\rm (Regularity results).} In addition, it holds that for any $t\in(0,1)$
  \begin{equation*} 
  \|f(t)\|_{C^{0,\alpha}\left(\Omega\times\R^3\right)} + \|\mathbf{E}_f(t)\|_{C^{1,\alpha}\left(\Omega\right)}
  \,\lesssim\, \varepsilon_0 
  \end{equation*}
  for some $\alpha\in (0,1]$, and
  \begin{equation*} 
  \|\nabla_{\!v} f\|_{L^\infty\left((0,1)\times\Omega\times\R^3\right)} \,\lesssim\, \varepsilon_0 . 
  \end{equation*}
  \item
  $\dt f$ also satisfies all the estimates above.
\end{itemize}
\end{theorem}

\smallskip
{\it \underline{Step\;2}. Bootstrapping}. 
The global regularity part of the theorem is a consequence of Proposition\;\ref{Prop:Bootstrap-est} and the local existence result, via a standard bootstrap-continuity argument.

Define
\begin{align*}
T_{\ast}(\varepsilon)&=\sup\bigg\{T>0:\ \text{for every choice of initial data satisfying $\mathcal{E}_{\vartheta}[f(0)]<\varepsilon$, there exists} \\
&\text{a solution on $[0,T]$ achieving the initial data and satisfying $\mathcal{E}_{\vartheta}[f(t)]\leq\e_1$ for $0<t\leq T$}\bigg\}.\no
\end{align*}
Local well-posedness theorem justifies that such $T_{\ast}>0$ is well-defined. Also, $T_{\ast}$ is non-increasing with respect to $\e$. Now we claim that there exists $0<\e_0<\e$ such that $T_{\ast}(\e_0)=\infty$.

If this is not true, then $0<T_{\ast}(\e_0)<\infty$. We will show the contradiction that actually the solution can be extended pass $T_{\ast}(\e_0)$.
For every $T_1\in(0,T_{\ast})$, we know there exists a solution achieving the initial data $\mathcal{E}_{\vartheta}[f(0]<\e_0<\e$ and satisfying $\mathcal{E}_{\vartheta}[f(t)]<\e_1$ for any $0<t\leq T_1$.
The  bootstrapping theorem justifies that now we actually have the improved estimate $\mathcal{E}_{\vartheta}[f(T_1)]<\e_1^s$.
Let $T_1$ be the new initial time. The local well-posedness theorem implies that now we have extended solution to $t\in[T_1,T_1+\tilde T]$. Since we can always take $T_1$ sufficiently close to $T_{\ast}$, such that $T_1+\tilde T>T_{\ast}$.
Then it remains the verify that for $t\in[T_1,T_1+\tilde T]$, we have $\mathcal{E}_{\vartheta}[f(t)]<\e_1$. Local well-posedness theorem tells us now in the extended interval, we have $\mathcal{E}_{\vartheta}[f(t)]\leq C_{\vartheta}\mathcal{E}_{\vartheta}[f(T_1)]\leq C_{\vartheta}\e_1^s$. Since $s>1$, for sufficiently small $\e_1$, we always have $C_{\vartheta}\e_1^s< \e_1$. Hence, the estimate also holds.

All in all, we extend the solution beyond $T_{\ast}$ which achieves the initial data and satisfies the estimate $\mathcal{E}_{\vartheta}[f(t)]<\e_1$. Then this is clearly a contradiction to the definition of $T_{\ast}$.
Therefore, the claim is proved and the solution exists globally.

\subsection{Uniqueness}
We now prove that the solution is unique.

\begin{proof}
Suppose that $(f_1,\mathbf{E}_1)$ and $(f_2,\mathbf{E}_2)$ are two global solutions to the VPL\,-\,specular problem \eqref{Eq:Vlasov-Landau_f}, \eqref{Eq:Poisson_f}, and \eqref{Specular-BC_f} with bounds (\ref{energy-bound})\,--\,(\ref{Dv-L^infty}) holding true globally.
Then for $i=1,2$, $(f_i,\mathbf{E}_i)$ satisfies
\begin{equation} \label{Eq:linearized-VL_fi}
\partial_t f_i + v\cdot\nabla_{\!x}f_i + Lf_i - 2\sqrt{\mu}\,v\cdot \mathbf{E}_i
\,=\, \Gamma[f_i,f_i] - \mathbf{E}_i\cdot\nabla_{\!v}f_i + \big(v\cdot \mathbf{E}_i \big)f_i ,
\end{equation}
where $\mathbf{E}_i := \mathbf{E}_{f_i} = -\nabla_{\!x}\phi_{f_i}$
with $\phi_i := \phi_{f_i}$ solved from the Poisson equation 
\begin{equation*}
-\Delta_x\phi_i = \int_{\R^3}\! \sqrt{\mu}\,f_i\,\dd v \,=: \rho[f_i] \;.
\end{equation*}
Note that \eqref{Eq:linearized-VL_fi} has all the linear terms on its LHS, while terms on RHS are the nonlinearities.
Let $\tilde{f} := f_1-f_2$ be the difference of two solutions and so $\widetilde{\mathbf{E}} := \mathbf{E}_1 - \mathbf{E}_2 = \mathbf{E}_{\tilde{f}}$.
Then from \eqref{Eq:linearized-VL_fi} we get the equation
\begin{equation*} 
\begin{split}
&\partial_t \tilde{f} + v\cdot\nabla_{\!x}\tilde{f} + L\tilde{f} - 2\sqrt{\mu}\,v\cdot \widetilde{\mathbf{E}} \\
\,=\;& \Big\{ \Gamma[\tilde{f},f_1] + \Gamma[f_2,\tilde{f}\,] \Big\}
- \Big\{ \widetilde{\mathbf{E}}\cdot\nabla_{\!v}f_1 + \mathbf{E}_2\cdot\nabla_{\!v}\tilde{f} \,\Big\}
+ \Big\{ \big(v\cdot \widetilde{\mathbf{E}} \big)f_1 + \big(v\cdot \mathbf{E}_2 \big)\tilde{f} \,\Big\} \;.
\end{split}
\end{equation*}

For a similar reason as in the proof of Theorem\;\ref{Thm:L^2-decay} (in the energy estimate), we multiply $e^{\phi_2}$ on both sides of the equation above, so that the last term on RHS can be merged into the second term on LHS.
This way the equation becomes
\begin{equation} \label{Eq:linearized-VL-difference-e^phi}
\begin{split}
&\partial_t \big(e^{\phi_2} \tilde{f}\,\big) + v\cdot\nabla_{\!x}\big(e^{\phi_2} \tilde{f}\,\big) + L\big(e^{\phi_2} \tilde{f}\,\big) - 2e^{\phi_2}\!\sqrt{\mu}\,v\cdot \widetilde{\mathbf{E}} \\
\,=\;& J_1 + J_2 + J_3 + J_4 \;,
\end{split}
\end{equation}
where
\begin{align*}
J_1 &\,:=\, e^{\phi_2}\, \Gamma[\tilde{f},f_1] \,+\, e^{\phi_2}\, \Gamma[f_2,\tilde{f}\,] \;, \\
J_2 &\,:=\, -\, \Big\{ e^{\phi_2}\, \widetilde{\mathbf{E}}\cdot\nabla_{\!v}f_1 \,+\, e^{\phi_2}\, \mathbf{E}_2\cdot\nabla_{\!v}\tilde{f} \,\Big\} \;, \\
J_3 &\,:=\, e^{\phi_2} \big(v\cdot \widetilde{\mathbf{E}} \big)f_1 \;, \\
J_4 &\,:=\, e^{\phi_2} \tilde{f}\, \partial_t\phi_2 \;.
\end{align*}
Multiplying both sides of \,\eqref{Eq:linearized-VL-difference-e^phi}\, by $\langle v\rangle^{2\vartheta} e^{\phi_2} \tilde{f}$\; for $\vartheta \leq -2$ (the same weight power as in Lemma\;\ref{Lem:Gamma-est}, in order to apply the estimate (\ref{Gamma-est-2}) for $\Gamma[\,\cdot,\,\cdot\,]$ term) and integrating over $(x,v)\in\Omega\times\R^3$,
we get
\begin{equation} \label{Eq:linearized-VL-difference-e^phi-int}
\begin{split}
&\frac{1}{2}\,\frac{\dd}{\dd t}\, \big\|e^{\phi_2} \tilde{f}\,\big\|_{2,\vartheta}^2
\,+\, \frac{1}{2}\iint_{\Omega\times\R^3}\! \langle v\rangle^{2\vartheta}\nabla_{\!x}\cdot \left\{v \big(e^{\phi_2} \tilde{f}\,\big)^2\right\} \\
&\,+\, \iint_{\Omega\times\R^3} \langle v\rangle^{2\vartheta} \big(e^{\phi_2} \tilde{f}\,\big)\, L\big(e^{\phi_2} \tilde{f}\,\big)
\,-\, \iint_{\Omega\times\R^3} 2\,\langle v\rangle^{2\vartheta} \big(e^{2\phi_2} \tilde{f}\,\big) \sqrt{\mu}\,v\cdot \widetilde{\mathbf{E}} \\
\,=\;& \iint_{\Omega\times\R^3} \langle v\rangle^{2\vartheta} \big(e^{\phi_2} \tilde{f}\,\big)\, J_1
\,+\, \iint_{\Omega\times\R^3} \langle v\rangle^{2\vartheta} \big(e^{\phi_2} \tilde{f}\,\big)\, J_2 \\
&\,+\, \iint_{\Omega\times\R^3} \langle v\rangle^{2\vartheta} \big(e^{\phi_2} \tilde{f}\,\big)\, J_3
\,+\, \iint_{\Omega\times\R^3} \langle v\rangle^{2\vartheta} \big(e^{\phi_2} \tilde{f}\,\big)\, J_4 \;.
\end{split}
\end{equation}
Now we estimate term by term in (\ref{Eq:linearized-VL-difference-e^phi-int}).
We first observe that $e^{\phi_2} \sim 1$ since $\|\phi_2\|_{\infty} \lesssim \|f_2\|_{\infty} \ll 1$ (see the proof of Lemma\;\ref{Lem:E_f-L^p}). 

On the LHS,
the second term vanishes since
\begin{equation} \label{uniqueness-LHS-2}
\frac{1}{2}\iint_{\Omega\times\R^3}\! \langle v\rangle^{2\vartheta}\nabla_{\!x}\cdot \left\{v \big(e^{\phi_2} \tilde{f}\,\big)^2\right\}
\,=\, \frac{1}{2} \bigg(\iint_{\gamma_+} \!- \iint_{\gamma_-}\bigg)
\langle v\rangle^{2\vartheta}\big(e^{\phi_2} \tilde{f}\,\big)^2 |v\cdot n_x| \,\dd v\dd S_x
\,=\, 0
\end{equation}
using the divergence theorem along with the specular boundary condition.
By (\ref{L-est-2}) in Lemma\;\ref{Lem:L-est}, the third term can be bounded below as
\begin{equation} \label{uniqueness-LHS-3}
\iint_{\Omega\times\R^3} \langle v\rangle^{2\vartheta} \big(e^{\phi_2} \tilde{f}\,\big)\, L\big(e^{\phi_2} \tilde{f}\,\big)
\,=\, \left(\langle v\rangle^{2\vartheta} L\big(e^{\phi_2} \tilde{f}\,\big) ,\, e^{\phi_2} \tilde{f}\, \right)
\,\gtrsim\, \big\|e^{\phi_2} \tilde{f}\,\big\|_{\sigma,\vartheta}^2 - C_\vartheta \big\|e^{\phi_2} \tilde{f}\,\big\|_{2,\vartheta}^2 \;.
\end{equation}
Also, we move the fourth term to the other side and bound it above using H\"{o}lder's inequality:
\begin{equation} \label{uniqueness-LHS-4}
\iint_{\Omega\times\R^3} 2\,\langle v\rangle^{2\vartheta} \big(e^{2\phi_2} \tilde{f}\,\big) \sqrt{\mu}\,v\cdot \widetilde{\mathbf{E}}
\;\lesssim\; \big|\langle v\rangle^{\vartheta+1}\! \sqrt{\mu}\,\big|_{\infty} \big\|e^{\phi_2}\tilde{f}\big\|_{2,\vartheta} \big\|e^{\phi_2}\widetilde{\mathbf{E}}\big\|_{2}
\;\lesssim\; \big\|e^{\phi_2}\tilde{f}\big\|_{2,\vartheta}^2 \;.
\end{equation}
The last inequality is valid because
\begin{equation*}
\begin{split}
\big\|e^{\phi_2}\widetilde{\mathbf{E}}\big\|_{2}
\,\lesssim\, \|\widetilde{\mathbf{E}}\|_{2}
&\,\lesssim\, \big\|\rho[\tilde{f}\,]\big\|_{2}
\,=\, \left\|\int_{\R^3}\! \sqrt{\mu}\,\tilde{f}\,\dd v \right\|_{L^2_{x}} \\
&\,\leq\, \int_{\R^3}\! \big\|\sqrt{\mu}\,\tilde{f}\big\|_{L^2_{x}} \dd v
\,= \int_{\R^3}\! \langle v\rangle^{-\vartheta}\! \sqrt{\mu}\, \big\|\langle v\rangle^{\vartheta} \tilde{f}\big\|_{L^2_{x}} \,\dd v \\[3pt]
&\,\leq\, \big\|\langle v\rangle^{\vartheta} \tilde{f}\big\|_{L^2_{x,v}} \!\cdot
\big|\langle v\rangle^{-\vartheta}\!\sqrt{\mu}\,\big|_{L^{2}_v} \\[2pt]
&\,\lesssim\; \|\tilde{f}\|_{2,\vartheta}
\;\lesssim\; \big\|e^{\phi_2}\tilde{f}\big\|_{2,\vartheta} \;,
\end{split}
\end{equation*}
by modifying the proof of Lemma\;\ref{Lem:E_f-L^p} with $\vartheta \leq -2$.

For the RHS, we estimate
\begin{equation} \label{uniqueness-RHS-1}
\begin{split}
\iint_{\Omega\times\R^3} \langle v\rangle^{2\vartheta} \big(e^{\phi_2} \tilde{f}\,\big)\, J_1
&\,=\, \iint_{\Omega\times\R^3} \langle v\rangle^{2\vartheta} \big(e^{\phi_2} \tilde{f}\,\big)
\Big\{e^{\phi_2}\, \Gamma[\tilde{f},f_1] \,+\, e^{\phi_2}\, \Gamma[f_2,\tilde{f}\,] \Big\} \\
&\,\lesssim\, \big(\|f_1\|_{\infty} + \|D_v f_1\|_{\infty} \big)\, \big\|e^{\phi_2}\tilde{f}\big\|_{\sigma,\vartheta}^2
+ \|f_2\|_{\infty} \big\|e^{\phi_2}\tilde{f}\big\|_{\sigma,\vartheta}^2 \\
&\,\lesssim\, \varepsilon_0\, \big\|e^{\phi_2}\tilde{f}\big\|_{\sigma,\vartheta}^2 \;,
\end{split}
\end{equation}
where we applied (\ref{Gamma-est-2}) to the first term and (\ref{Gamma-est-1}) to the second (see Lemma\;\ref{Lem:Gamma-est}). Also,
\begin{equation} \label{uniqueness-RHS-2}
\begin{split}
&\qquad \iint_{\Omega\times\R^3} \langle v\rangle^{2\vartheta} \big(e^{\phi_2} \tilde{f}\,\big)\, J_2 \\
&\,=\, - \iint_{\Omega\times\R^3} \langle v\rangle^{2\vartheta} \big(e^{\phi_2} \tilde{f}\,\big)
\Big\{ e^{\phi_2}\, \widetilde{\mathbf{E}}\cdot\nabla_{\!v}f_1 \,+\, e^{\phi_2}\, \mathbf{E}_2\cdot\nabla_{\!v}\tilde{f} \,\Big\} \\
&\,=\, - \iint_{\Omega\times\R^3} \langle v\rangle^{2\vartheta} \big(e^{\phi_2} \tilde{f}\,\big) e^{\phi_2}\, \widetilde{\mathbf{E}}\cdot\nabla_{\!v}f_1
\,+\, \frac{1}{2}\iint_{\Omega\times\R^3} \mathbf{E}_2\cdot\nabla_{\!v} \Big( \langle v\rangle^{2\vartheta}\Big) \big(e^{\phi_2} \tilde{f}\,\big)^2 \\
&\,\lesssim\, \|D_v f_1\|_{\infty} \|\widetilde{\mathbf{E}}\|_{\infty} \big\|e^{\phi_2}\tilde{f}\big\|_{2,\vartheta} \big|\langle v\rangle^{\vartheta} \big|_2
\,+\, \|\mathbf{E}_2\|_{\infty} \big\|e^{\phi_2}\tilde{f}\big\|_{2,\vartheta}^2  \\
&\,\lesssim\, \big\|e^{\phi_2}\tilde{f}\big\|_{2,\vartheta}^2 \;,
\end{split}
\end{equation}
noting that with $\vartheta\leq -2$\; we have $\big|\langle v\rangle^{\vartheta} \big|_2 <\infty$.
In addition,
\begin{equation} \label{uniqueness-RHS-3}
\begin{split}
\iint_{\Omega\times\R^3} \langle v\rangle^{2\vartheta} \big(e^{\phi_2} \tilde{f}\,\big)\, J_3
&\,=\, \iint_{\Omega\times\R^3} \langle v\rangle^{2\vartheta} \big(e^{\phi_2} \tilde{f}\,\big) e^{\phi_2}  \big(v\cdot \widetilde{\mathbf{E}} \big)f_1 \\
&\,\lesssim\, \|f_1\|_{\infty,\vartheta+1} \big\|e^{\phi_2}\tilde{f}\big\|_{2,\vartheta} \big\|e^{\phi_2}\widetilde{\mathbf{E}}\big\|_{2} \\
&\,\lesssim\, \big\|e^{\phi_2}\tilde{f}\big\|_{2,\vartheta}^2 \;,
\end{split}
\end{equation}
and
\begin{equation} \label{uniqueness-RHS-4}
\begin{split}
\iint_{\Omega\times\R^3} \langle v\rangle^{2\vartheta} \big(e^{\phi_2} \tilde{f}\,\big)\, J_4
&\,=\, \iint_{\Omega\times\R^3} \langle v\rangle^{2\vartheta} \big(e^{\phi_2} \tilde{f}\,\big)
e^{\phi_2} \tilde{f}\, \partial_t\phi_2 \\
&\,\lesssim\, \|\partial_t \phi_2\|_{\infty} \big\|e^{\phi_2}\tilde{f}\big\|_{2,\vartheta}^2 \\
&\,\lesssim\, \|\partial_t f_2\|_{2} \big\|e^{\phi_2}\tilde{f}\big\|_{2,\vartheta}^2
\,\lesssim\, \big\|e^{\phi_2}\tilde{f}\big\|_{2,\vartheta}^2 \;.
\end{split}
\end{equation}

Collecting all the estimates in (\ref{uniqueness-LHS-2})\,--\,(\ref{uniqueness-RHS-4}) above and plugging them into \eqref{Eq:linearized-VL-difference-e^phi-int},
we then absorb $\varepsilon_0 \big\|e^{\phi_2}\tilde{f}\big\|_{\sigma,\vartheta}^2$ from (\ref{uniqueness-RHS-1}) into LHS and obtain
\begin{equation*}
\frac{\dd}{\dd t}\, \big\|e^{\phi_2} \tilde{f}\,\big\|_{2,\vartheta}^2
\,\leq\, C \big\|e^{\phi_2} \tilde{f}\,\big\|_{2,\vartheta}^2 \;.
\end{equation*}
Hence, by the Gronwall inequality, we have that for every $t\in [\,0,\infty)$, 
\begin{equation*}
\big\|e^{\phi_2} \tilde{f}(t)\big\|_{2,\vartheta}^2
\,\leq\, e^{Ct} \big\|e^{\phi_2} \tilde{f}(0)\big\|_{2,\vartheta}^2 = 0,
\end{equation*}
noticing that $f_1(0) = f_2(0) = f_0$ and thus $\tilde{f}(0) \equiv 0$.
This implies that $\tilde{f}(t) \equiv 0$ for all $t\geq 0$, since $e^{\phi_2} >0$ and also due to the continuity of the solutions.
Therefore, we conclude that the solution exists uniquely.
\end{proof}

%

\subsection*{Acknowledgements}

The authors would like to thank Dr. Sona Akopian for many helpful discussions.
They are also grateful to Dr. Timur Yastrzhembskiy for careful reading and pointing out a few important issues in an earlier version of the manuscript.
Yan Guo's research is supported in part by NSF grant DMS-2106650.
Hongjie Dong is partially supported by the Simons Foundation, grant no. 709545, a Simons fellowship, grant no. 007638, and the NSF under agreement DMS-2055244.

%
%
\addtocontents{toc}{\protect\setcounter{tocdepth}{1}}

\bibliographystyle{siam}

\end{document}